\documentclass[english]{article}
\usepackage[T1]{fontenc}
\usepackage[latin9]{inputenc}
\usepackage{babel}
\usepackage{refstyle}
\usepackage{textcomp}
\usepackage{amsmath}
\usepackage{amsthm}
\usepackage{amssymb}
\usepackage{wasysym}
\usepackage[unicode=true,pdfusetitle,
 bookmarks=true,bookmarksnumbered=false,bookmarksopen=false,
 breaklinks=false,pdfborder={0 0 1},backref=false,colorlinks=false]
 {hyperref}

\makeatletter

\AtBeginDocument{\providecommand\eqref[1]{\ref{eq:#1}}}
\RS@ifundefined{subsecref}
  {\newref{subsec}{name = \RSsectxt}}
  {}
\RS@ifundefined{thmref}
  {\def\RSthmtxt{theorem~}\newref{thm}{name = \RSthmtxt}}
  {}
\RS@ifundefined{lemref}
  {\def\RSlemtxt{lemma~}\newref{lem}{name = \RSlemtxt}}
  {}

\numberwithin{equation}{section}
\numberwithin{figure}{section}
\theoremstyle{plain}
\newtheorem{thm}{\protect\theoremname}[section]
\theoremstyle{definition}
\newtheorem{defn}[thm]{\protect\definitionname}
\theoremstyle{definition}
\newtheorem{example}[thm]{\protect\examplename}
\theoremstyle{remark}
\newtheorem{notation}[thm]{\protect\notationname}
\theoremstyle{remark}
\newtheorem{rem}[thm]{\protect\remarkname}
\theoremstyle{plain}
\newtheorem*{lem*}{\protect\lemmaname}
\theoremstyle{plain}
\newtheorem*{thm*}{\protect\theoremname}
\theoremstyle{plain}
\newtheorem*{conjecture*}{\protect\conjecturename}
\theoremstyle{plain}
\newtheorem{prop}[thm]{\protect\propositionname}
\theoremstyle{plain}
\newtheorem{lem}[thm]{\protect\lemmaname}
\theoremstyle{plain}
\newtheorem{cor}[thm]{\protect\corollaryname}
\theoremstyle{plain}
\newtheorem{conjecture}[thm]{\protect\conjecturename}

\usepackage[all]{xy}
\usepackage{mathrsfs}
\usepackage{eucal}
\usepackage{bbm}
\usepackage{fullpage}
\usepackage{mathscinet}
\usepackage{attrib}
\usepackage{authblk}
\usepackage{enumerate}
\usepackage{pbox}

\hypersetup{hidelinks = true}

\newtheorem{propositiona}{Proposition}

\newcommand{\hooklongrightarrow}{\lhook\joinrel\longrightarrow}

\newcommand\<{\langle}
\renewcommand\>{\rangle}

\usepackage{graphicx}

\def\nicefracbig#1#2{
    \raise.5ex\hbox{$#1$}%
    \kern-.25em\bigr/\kern-.15em%
    \lower.35ex\hbox{$#2$}}

\newcommand{\xyR}[1]{%
\xydef@\xymatrixrowsep@{#1}
}

\newcommand{\xyC}[1]{%
\xydef@\xymatrixcolsep@{#1}
}

\AtEndDocument{%
  \par  
  \medskip 
  \begin{tabular}{@{}l@{}}%
    \textsc{Mathematical Institute, University of Oxford, Oxford OX2 6GG, United Kingdom}\\
      \textit{E-mail address}: \texttt{w.malten@gmail.com}   
  \end{tabular}}

\theoremstyle{definition}
\newtheorem{defthm}[thm]{Definition/Theorem}

\makeatother

\providecommand{\conjecturename}{Conjecture}
\providecommand{\corollaryname}{Corollary}
\providecommand{\definitionname}{Definition}
\providecommand{\examplename}{Example}
\providecommand{\lemmaname}{Lemma}
\providecommand{\notationname}{Notation}
\providecommand{\propositionname}{Proposition}
\providecommand{\remarkname}{Remark}
\providecommand{\theoremname}{Theorem}

\begin{document}
\title{Minimally dominant elements of finite Coxeter groups}
\author{Wicher Malten}
\maketitle
\begin{abstract}
Recently, Lusztig constructed for each reductive group a partition
by unions of sheets of conjugacy classes, which is indexed by a subset
of the set of conjugacy classes in the associated Weyl group. Sevostyanov
subsequently used certain elements in each of these Weyl group conjugacy
classes to construct strictly transverse slices to the conjugacy classes
in these strata, generalising the classical Steinberg slice, and similar
cross sections were built out of different Weyl group elements by
He-Lusztig.

In this paper we observe that He-Lusztig's and Sevostyanov's Weyl
group elements share a certain geometric property, which we call minimally
dominant; for example, we show that this property characterises involutions
of maximal length. Generalising He-Nie's work on twisted conjugacy
classes in finite Coxeter groups, we explain for various geometrically
defined subsets that their elements are conjugate by simple shifts,
cyclic shifts and strong conjugations. We furthermore derive for a
large class of elements that the principal Deligne-Garside factors
of their powers in the braid monoid are maximal in some sense. This
includes those that are used in He-Lusztig's and Sevostyanov's cross
sections, and explains their appearance there; in particular, all
minimally dominant elements in the aforementioned conjugacy classes
yield strictly transverse slices. These elements are conjugate by
cyclic shifts, their Artin-Tits braids are never pseudo-Anosov in
the conjectural Nielsen-Thurston classification and their Bruhat cells
should furnish an alternative construction of Lusztig's inverse to
the Kazhdan-Lusztig map and of his partition of reductive groups.

\end{abstract}
\tableofcontents{}

\section{Introduction}

During the 1990s, the character tables of finite Iwahori-Hecke algebras
were determined through case-by-case analysis and extensive computer
calculations on the conjugacy classes of finite Coxeter groups \cite{MR1778802}.
More specifically, these algebras have a basis $\{T_{w}\}$ indexed
by elements $w$ of the corresponding finite Coxeter group $W$, and
Geck-Pfeiffer discovered properties of conjugacy classes of $W$ which
imply that the value of a character $\chi$ is constant on these basis
elements when they correspond to minimal length elements in the same
conjugacy class, and moreover that those values can be used to compute
$\chi(T_{w})$ for arbitrary elements $w$ in this conjugacy class
\cite{MR1250466,MR1769289}. Furthermore, Geck-Michel noticed that
every conjugacy class of $W$ contains minimal length elements whose
powers in the braid monoid are particularly simple, which was subsequently
used to determine the ``rationality'' of these characters \cite{MR1425324,MR1769289,MR2355597}.
These results on conjugacy classes were later reused in other domains
involving Weyl group elements (e.g.\ Bruhat cells \cite{MR2042932,MR2657442,MR2833465},
Deligne-Lusztig varieties \cite{MR2427637,MR2427642}, 0-Hecke algebras
\cite{MR3495798} and partitions of the wonderful compactification
\cite{MR2355597}), and in particular He-Lusztig applied them to construct
cross sections in reductive groups out of elliptic Weyl group elements
of minimal length \cite{MR2904572}.

Recently, He-Nie gave a simple, geometric proof of these statements
by focusing on the eigenspace decomposition for $w$ acting in the
real reflection representation \cite{MR2999317}. Meanwhile, Sevostyanov
used similar data to construct transverse slices to conjugacy classes
in reductive groups \cite{MR2806525}, and verified that for any stratum
in Lusztig's partition \cite{MR3495802} a suitable choice yields
slices which are strictly transverse to its conjugacy classes \cite{MR3883243}.
His data can be recast in the same framework:
\begin{defn}
\label{def:parabolic} A \emph{twisted Coxeter group} is a group generated
by some automorphisms and some simple reflections of a Coxeter system.
We may assume that each of these automorphisms $\{\delta_{i}\}_{i=1}^{m}$
preserves the set of these simple reflections $\{s_{i}\}_{i=1}^{n}$;
more concretely, a twisted Coxeter group is then a semidirect product
$W=\Omega\ltimes\tilde{W}$, where $\tilde{W}=\<s_{i}\>_{i=1}^{n}$
is a Coxeter group and $\Omega=\<\delta_{i}\>_{i=1}^{m}$ is a group
acting on $\tilde{W}$ by automorphisms, each of which preserves its
simple reflections. We call an element $\delta$ of $\Omega$ a \emph{twist
of }$\tilde{W}$ and denote elements of $W$ by $\delta\tilde{w}:=(\delta,\tilde{w})$.
Given a twisted Coxeter group $W$ (with such data) we always denote
its untwisted part by $\tilde{W}$.
\end{defn}

When $W$ is generated by all automorphisms and simple reflections,
one obtains the group of automorphisms of the corresponding root (or
reflection) system $\mathfrak{R}$.  Twists were originally motivated
by the theory of Iwahori-Hecke algebras and non-split reductive groups
but coincidentally, they also played a crucial rôle in He-Nie's proof
of the aforementioned statements for ordinary finite Coxeter groups.
\begin{example}
When the Coxeter system is irreducible and of finite type, their classification
by Coxeter-Dynkin diagrams implies that every twist $\delta$ satisfies
$\mathrm{ord}(\delta)\leq3$.
\end{example}

\begin{defn}
We define the \emph{standard parabolic} subgroups of a twisted Coxeter
group $W$ to be the subgroups of the form $\Omega'\ltimes\tilde{W}'$,
where $\tilde{W}'$ is a standard parabolic subgroup of $\tilde{W}$
and $\Omega'$ is a subgroup of $\Omega$ whose elements preserve
(the set of simple reflections in) $\tilde{W}'$. Their conjugates
are called \emph{parabolic} subgroups.

Let $W$ be a twisted finite Coxeter group and let $W'$ be a standard
parabolic subgroup of $\tilde{W}$. By a \emph{$W'$-orbit} we then
mean an orbit of $W'$ on $W$ under the restriction of the conjugation
action to $W'$; every conjugacy class in $W$ is then partitioned
into $W'$-orbits.
\end{defn}

\begin{example}
In type $\mathsf{A}_{n}$, the only nontrivial twist coincides with
conjugation by the longest element, so $\tilde{W}$-orbits in $W$
agree with conjugacy classes.
\end{example}

\begin{example}
As twists permute Coxeter elements of minimal length\footnote{By a Coxeter element we mean a conjugate of the product (in any order)
of all of the simple reflections in the group.} and they are already conjugate under $\tilde{W}$, it again follows
that their $\tilde{W}$-orbit is the entire conjugacy class.
\end{example}

\begin{example}
Let $\delta$ be the twist in $\mathsf{D}_{4}$ interchanging the
third and fourth simple reflections\footnote{Throughout this paper we follow Bourbaki's labelling for simple roots
$\alpha_{i}$ in irreducible root systems \cite[§VI.4]{MR0240238},
so for $\mathsf{D}_{n}$ the degree 3 vertex has label $n-2$. Furthermore,
we write $\alpha_{i_{1}\cdots i_{m}}:=\sum_{j=1}^{m}\alpha_{i_{j}}$
if this sum of simple roots is a root.} and consider the twisted Coxeter group $W:=\<\delta\>\ltimes\tilde{W}$.
Then $s_{3}s_{1}$ and $\delta(s_{3}s_{1})=s_{4}s_{1}$ are not $\tilde{W}$-conjugate,
so the conjugacy class of $s_{3}s_{1}$ splits into the $\tilde{W}$-orbits
of $s_{3}s_{1}$ and $s_{4}s_{1}$.
\end{example}

\begin{notation}
Throughout this paper, we have for the sake of brevity employed symbols
like $\pm,\gtrless,\mathrm{max}/\mathrm{min}$. This always means
that there are \emph{two} statements being made simultaneously: one
where the top and left symbols are used at each such position, and
another one for the bottom and right symbols. This duality is explained
in Remark \ref{rem:duality}.
\end{notation}

\begin{defn}
Let $W$ be a twisted Coxeter group. When the untwisted part $\tilde{W}\subseteq W$
is finite, we will call $W$ an twisted\emph{ finite} Coxeter group
and denote the longest element of $\tilde{W}$ by $w_{\circ}$. We
denote its set of positive (resp.\ negative) roots by $\mathfrak{R}_{+}$
(resp.\ $\mathfrak{R}_{-}$), and for any element $w$ in $W$ we
denote its \emph{(right) (root) inversion set} by 
\[
\mathfrak{R}_{w}:=w^{-1}(\mathfrak{R}_{-})\cap\mathfrak{R}_{+}=\{\beta\in\mathfrak{R}_{+}:w(\beta)\in\mathfrak{R}_{-}\}.
\]
Given an element $w$ of $W$, we say that a root is \emph{stable}
if its $w$-orbit consists solely of positive or solely of negative
roots. We then denote by $\mathfrak{R}_{\mathrm{st}}^{w}$ the set
of stable roots, i.e.\
\[
\mathfrak{R}_{\mathrm{st}}^{w}=\{\beta\in\mathfrak{R}:w^{i}(\beta)\notin\mathfrak{R}_{w}\textrm{ for }1\leq i\leq\mathrm{ord}(w)\}.
\]
 In particular, there is an inclusion $\mathfrak{R}^{w}\subseteq\mathfrak{R}_{\mathrm{st}}^{w}$,
where $\mathfrak{R}^{w}$ denotes the set of roots that are fixed
by $w$.

We say that $w$ is\emph{ convex }when $\mathfrak{R}_{\mathrm{st}}^{w}$
forms a standard parabolic subsystem, and then denote by $\mathrm{pb}(w)$
its \emph{(braid) power bound}: this is the unique element of $\tilde{W}$
which makes negative all positive roots except for those in $\mathfrak{R}_{\mathrm{st}}^{w}$.
Thus $\mathrm{pb}(w)=w_{\circ}w_{\mathrm{st}}$, where $w_{\mathrm{st}}$
now denotes the longest element of the standard parabolic subgroup
corresponding to $\mathfrak{R}_{\mathrm{st}}^{w}$. If furthermore
we have an equality $\mathfrak{R}^{w}=\mathfrak{R}_{\mathrm{st}}^{w}$,
then we say that the element $w$ is \emph{firmly convex}. 
\end{defn}

The name braid power bound was derived from part (i) of Proposition
\ref{prop:dg-bound}. In Proposition \ref{prop:closed-complement}
we show that an element $w$ is convex if and only if the set of non-stable
positive roots $\mathfrak{R}_{+}\backslash\mathfrak{R}_{\mathrm{st}}^{w}$
is convex (see Definition \ref{def:convex-ordering}); this property
plays an important rôle in the sequel to this paper \cite{WM-cross}.
\begin{notation}
The action of a Coxeter group $\tilde{W}$ on its real reflection
representation $V$ (which is due to Tits \cite[§5.4.4]{MR0240238})
extends to twisted Coxeter groups $W=\Omega\ltimes\tilde{W}$ by letting
elements of $\Omega$ act on the simple roots as prescribed, and extending
linearly. We will then denote the $W$-invariant inner product on
$V$ by $(\cdot,\cdot)$. Given a standard parabolic subgroup $W'$,
we will typically interpret its reflection representation $V_{W'}$
as sitting inside of $V$ through the span of the corresponding simple
roots.

The \emph{rank} of $W$ is the dimension of $V$ (which equals the
number of simple reflections in $W$) and is sometimes denoted by
$\mathrm{rank}(W)$ or simply $\mathrm{rk}$. For any other element
$w'$ in $W$ we write $w'\geq w$ if and only if $\mathfrak{R}_{w}\subseteq\mathfrak{R}_{w'}$,
extending the weak left Bruhat-Chevalley (partial) order on $\tilde{W}$
to a preorder on $W$; antisymmetry holds up to left-multiplication
by twists. We let $V_{w}=\mathrm{im}(\mathrm{id}_{V}-w)\subseteq V$
denote the orthogonal complement to the subspace $V^{w}=\mathrm{ker}(\mathrm{id}_{V}-w)\subseteq V$
of $w$-fixed points in $V$.

Moreover, we denote the length of $w$ by $\ell(w):=|\mathfrak{R}_{w}|$
and the number of roots it fixes by $\ell_{f}(w):=|\mathfrak{R}^{w}|$.
For any subset $\CMcal O$ of $W$ we denote by $\CMcal O^{\mathrm{cx}}$
the subset of convex elements in $\CMcal O$ and write $\CMcal O_{\mathrm{max}/\mathrm{min}}$
for the subset of maximal/minimal length elements inside of $\CMcal O$.
If the length function is constant on $\CMcal O$ then we will simply
denote this value by $\ell(\CMcal O)$, and similarly for $\ell_{f}(\cdot)$
and the order (or period) function $\mathrm{ord}(\cdot)$.
\end{notation}

\begin{defn}
[{\cite[§1.3, §2.1]{MR2999317}}] Let $W$ be a twisted finite Coxeter
group and let $V'\subseteq V$ be a subset of its reflection representation.
Then we let $W_{V'}\subseteq W$ denote the subgroup of $W$ consisting
of elements fixing $V'$, and we denote the set of root hyperplanes
containing $V'$ by $\mathfrak{H}_{V'}$.

A vector $v$ in $V'$ is called \emph{a regular point of }$V'$ if
$\tilde{W}_{V'}=\tilde{W}_{\{v\}}$, which by Steinberg's fixed-point
theorem \cite[Theorem 1.5]{MR167535} is equivalent to $\mathfrak{H}_{V'}=\mathfrak{H}_{\{v\}}$.
If $V'$ is convex then such points form a dense open subset of $V'$,
by finiteness of the number of root hyperplanes. Given an element
$w$ of $W$, we denote by $\mathfrak{H}^{w}:=\mathfrak{H}_{V_{w}}$
the subset of root hyperplanes corresponding to roots of $\mathfrak{R}^{w}$.
\end{defn}

Proposition \ref{prop:parabolic} shows that an isotropy group $W_{V'}$
is a parabolic subgroup, and if the closure of the dominant Weyl chamber
contains a regular point (or open subset) of $V'$ then $W_{V'}$
is moreover a standard parabolic subgroup.
\begin{defn}
\label{def:sequence-eigenspaces} Let $W=\Omega\ltimes\tilde{W}$
be a twisted finite Coxeter group and pick an element $w$ in $W$.
As $w$ is a linear isometry it naturally decomposes the real reflection
representation
\[
V=\bigoplus_{\Im(\lambda)\geq0}V_{\lambda}^{w}=\bigoplus_{\Im(\lambda)\geq0}\bigoplus_{k}V_{\lambda,k}^{w}
\]
into an orthogonal sum of subspaces: given a complex eigenvalue $\lambda=e^{2\pi i\theta}$
whose imaginary part $\Im(\lambda)$ is nonnegative, the element $w$
acts as rotation with angle $2\pi\theta$ on the ``real eigenspace''
\[
V_{\lambda}^{w}:=\{v\in V:w(v)+w^{-1}(v)=2\cos(2\pi\theta)v\},
\]
 whose complexification is the sum of the usual complex eigenspaces
for $\lambda$ and $\lambda^{-1}$ in $V\otimes_{\mathbb{R}}\mathbb{C}$.
We may further (non-uniquely) decompose $V_{\lambda}^{w}$ into smaller
real eigenspaces $V_{\lambda,k}^{w}$ ($k\geq1$) of dimension $\geq1$
for $\lambda\in\{1,-1\}$ and of even dimension $\geq2$ for all other
$\lambda$. We will then consider sequences of distinct eigenspaces
\[
\underline{\Theta}=(V_{\lambda_{m},k_{m}}^{w},\ldots,V_{\lambda_{1},k_{1}}^{w})
\]
 and leave out the underscore to denote the underlying set of eigenspaces
by $\Theta$. Given an arbitrary set $\Theta=\{V_{m},\ldots,V_{1}\}$
of eigenspaces of some element $w$, we denote its span by
\[
V_{\Theta}^{\,}:=\bigoplus_{i=1}^{m}V_{i}.
\]

Given a sequence of eigenspaces $\underline{\Theta}=(V_{m},\ldots,V_{1})$,
we define a filtration on this span $V_{\Theta}\subseteq V$ via
$F_{i}:=\sum_{j=1}^{i}V_{j}$ for $0\leq i\leq m$. Then setting $\tilde{W}_{i}:=\tilde{W}_{F_{i}}$
for $0\leq i\leq m$, we obtain a filtration of parabolic subgroups
\cite[§5.2]{MR2999317}
\begin{equation}
\tilde{W}_{\Theta}:=\tilde{W}_{V_{\Theta}}=\tilde{W}_{m}\subseteq\cdots\subseteq\tilde{W}_{1}\subseteq\tilde{W}_{0}=\tilde{W}.\label{eq:parabolic-sequence}
\end{equation}
Equivalently, one obtains a sequence of root hyperplanes $\mathfrak{H}_{F_{i}}$;
when $\mathfrak{H}_{F_{i}}=\mathfrak{H}_{F_{i-1}}$ (or $\tilde{W}_{i}=\tilde{W}_{i-1}$)
one says that $V_{i}$ is \emph{redundant}. We will say that a Weyl
chamber $C$ is \emph{in good position with respect to $\underline{\Theta}$}
if the closure of this Weyl chamber contains a regular point of $F_{i}$
for each $i$. In Proposition \ref{prop:good-position} we prove that
this is equivalent to the definition given in \cite[§5.2]{MR2999317}.

Moreover, there is always at least one Weyl chamber in good position
\cite[Lemma 5.1]{MR2999317} and typically we will require that the
dominant Weyl chamber is one of them; as we've noted, this implies
that the parabolic subgroups $W_{i}$ are standard parabolic subgroups.
If indeed the dominant Weyl chamber is in good position and the set
of hyperplanes $\mathfrak{H}_{\Theta}:=\mathfrak{H}_{V_{\Theta}}$
containing $V_{\Theta}$ is a subset of $\mathfrak{H}^{w}$, then
we say that $\underline{\Theta}$ is \emph{braiding}; if moreover
$\mathfrak{H}_{\Theta}$ is the empty set then we say that it is \emph{complete}.

We will say that an element $w$ is \emph{quasiregular} if $\mathfrak{H}_{V_{\lambda}^{w}}\subseteq\mathfrak{H}^{w}$
for some eigenvalue $\lambda$ of $w$. In Proposition \ref{prop:quasiregular}
we explain that regular elements are quasiregular elements satisfying
$\mathfrak{H}_{V_{\lambda}^{w}}=\varnothing$ for some $\lambda$
and in Proposition \ref{prop:Coxeter-braiding} we explain that there
are only two Coxeter elements of minimal length in $W$ that admit
a braiding sequence of eigenspaces, namely the bipartite ones.

Given a sequence of eigenspaces $\underline{\Theta}=(V_{m},\ldots,V_{1})$
there is an underlying sequence of eigenvalues $(\lambda_{m},\ldots,\lambda_{1})$
whose principal arguments we will normalise to lie in $[0,1/2]$ and
then denote by 
\[
(\theta_{m},\ldots,\theta_{1}),
\]
so $\lambda_{j}=e^{2\pi i\theta_{j}}$ with $0\leq\theta_{j}\leq1/2$
for each $1\leq j\leq m$. If none of the eigenvalues of $\underline{\Theta}$
equals 1, then we will call the sequence \emph{anisotropic. } 

Given a set of eigenspaces $\Theta$, we write $\mathrm{eig}(\Theta):=\{\lambda_{m},\ldots,\lambda_{1}\}$
for the corresponding set of eigenvalues. We denote by $\underline{\Theta}_{\pm}$
the sequence obtained from the set $\Theta$ by merging eigenspaces
if their eigenvalues agree, and then ordering them so that $\theta_{m}<\cdots<\theta_{1}$
(resp.\ $\theta_{m}>\cdots>\theta_{1}$); we will call such a sequence
\emph{increasing}/\emph{decreasing}. We furthermore write $\Theta=\{\lambda_{m},\ldots,\lambda_{1}\}$
when $\Theta=\{V_{\lambda_{m}}^{w},\ldots,V_{\lambda_{1}}^{w}\}$.
\end{defn}

\begin{defn}
Let $W$ be a twisted finite Coxeter group. An element $w$ of $W$
is called \emph{elliptic} (or \emph{cuspidal} or \emph{anisotropic})
if it has no nonzero fixed points in the reflection representation
of $W$, and a subset of $W$ is labelled thusly all of its elements
have this property.
\end{defn}

Sevostyanov's data for his transverse slices is equivalent to a Weyl
group element with an anisotropic braiding sequence of eigenspaces
(with each eigenspace of dimension 1 if the eigenvalue is $-1$ and
of dimension 2 otherwise; he puts eigenspaces with eigenvalue 1 at
the end of the sequence, but in that position they can be safely ignored)
\cite[§2]{MR2806525}. On the other hand, He-Nie uses decreasing complete
sequences \cite[§5.2]{MR2999317}. Thus for elliptic elements Sevostyanov's
transverse slice conditions are slightly more general, whereas as
the sets of non-elliptic elements these two conditions describe are
entirely disjoint. In case-by-case work Sevostyanov later furnished
elements with an anisotropic braiding sequence which is decreasing,
i.e.\ $\theta_{m}>\cdots>\theta_{1}>0$, and verified that the corresponding
slice is strictly transverse \cite[§1]{MR3883243}. 
\begin{example}
Let $w$ be a non-simple reflection in type $\mathsf{B}_{2}$. Then
the subspace $V^{w}$ of fixed points is nontrivial and contains roots
but its intersection with the closure of the dominant Weyl chamber
is $\{0\}$, so $w$ has no decreasing complete sequence of eigenspaces,
whilst its only nontrivial eigenspace $V_{-1}^{w}$ yields an anisotropic
braiding one. If instead $w$ is a simple reflection then it has no
anisotropic braiding sequence, whilst $\underline{\Theta}:=(V_{-1}^{w},V_{1}^{w})$
is decreasing and complete.
\end{example}

\begin{defn}
\label{def:translates} Let $W$ be a twisted finite Coxeter group
with dominant Weyl chamber $C$, pick a standard parabolic subgroup
$W'\subseteq\tilde{W}$ and let $\CMcal O$ be a $W'$-orbit. Let
$\underline{\Theta}=(V_{m},\ldots,V_{1})$ be a sequence of eigenspaces
of an element $w$ of $\CMcal O$. Given an element $\tau$ in $W'$,
we then obtain a sequence of eigenspaces 
\[
\tau(\underline{\Theta}):=\bigl(\tau(V_{m}),\ldots,\tau(V_{1})\bigr)
\]
 for the element $\tau w\tau^{-1}$ in $\CMcal O$, and we similarly
define $\tau(\Theta)$. We may then define
\[
\CMcal O^{\Theta}:=\{\tau w\tau^{-1}\in\CMcal O:\tau\in W'\text{ and }C\text{ contains a regular point of }V_{\tau(\Theta)}\}
\]
and
\[
\CMcal O^{\underline{\Theta}}:=\{\tau w\tau^{-1}\in\CMcal O:\tau\in W'\text{ and }C\text{ is in good position with respect to }\tau(\underline{\Theta})\},
\]
and finally set $\CMcal O_{\pm}^{\Theta}:=\CMcal O^{\underline{\Theta}_{\pm}}$.
\end{defn}

Rephrasing the usual definition of cyclic shifts sheds some light
on the dualities that follow:
\begin{defn}
\label{def:shifts} Let $W$ be a twisted Coxeter group, let $\CMcal O$
be a subset of $W$ and pick two elements $w$ and $w'$ in $\CMcal O$.
When $\CMcal O$ is not mentioned we set $\CMcal O:=W$ for the following
notions:

\begin{itemize}

\item If there exists a sequence of elements $w'=w_{n+1},\ldots,w_{0}=w$
in $\CMcal O$ and simple reflections $s_{i_{0}},\ldots,s_{i_{n}}$
such that 
\[
s_{i_{j}}w_{j}s_{i_{j}}=w_{j+1}\qquad\textrm{and}\qquad\ell(w_{j})\leq\ell(w_{j+1})
\]
 holds for each $0\leq j\leq n$, then we write \emph{$w\overset{+}{\rightarrow}w'$
in $\CMcal O$} and \emph{$w'\overset{-}{\rightarrow}w$ in $\CMcal O$}.
If furthermore $\ell(w)=\ell(w')$ then we write $w\leftrightarrow w'$
\emph{in $\CMcal O$}, and call the set of such $w'$ the \emph{simple
shift class of $w$ in $\CMcal O$} \cite[§1]{MR1250466}. If on the
other hand there is a strict inequality $\ell(w_{j})<\ell(w_{j+1})$
for each $j$, we write
\[
w\overset{+}{\twoheadrightarrow}w'\textrm{ \emph{in} }\CMcal O\qquad\textrm{and}\qquad w'\overset{-}{\twoheadrightarrow}w\textrm{ \emph{in} }\CMcal O.
\]
\item If there exists a sequence of elements $w'=w_{n+1},\ldots,w_{0}=w$
in $\CMcal O$ and $\tau_{n},\ldots,\tau_{0}$ in $\tilde{W}$ such
that
\[
\tau_{j}^{\,}w_{j}^{\,}\tau_{j}^{-1}=w_{j+1}^{\,}\qquad\textrm{and}\qquad\ell(w_{j})\leq\ell(w_{j+1})
\]
 holds for each $0\leq j\leq n$, and furthermore for each such $j$
at least one of
\[
\ell(\tau_{j}w_{j})=\ell(w_{j})+\ell(\tau_{j})\qquad\textrm{or}\qquad\ell(w_{j}^{\,}\tau_{j}^{-1})=\ell(w_{j}^{\,})+\ell(\tau_{j}^{-1})
\]
 holds, then we write $w\overset{+}{\leadsto}w'$\emph{ in $\CMcal O$}.
If furthermore $\ell(w_{j})=\ell(w_{j+1})$ holds for each $j$, we
write $w\overset{+}{\sim}w'$ and call the set of such $w'$ the \emph{strong
conjugacy class of $w$ in $\CMcal O$ }\cite[§1]{MR1250466}\emph{.}

\item If instead $\ell(w_{j})\geq\ell(w_{j+1})$ and for each $0\leq j\leq n$
at least one of
\[
\ell(\tau_{j}w_{j})=\ell(w_{j})-\ell(\tau_{j})\qquad\textrm{or}\qquad\ell(w_{j}^{\,}\tau_{j}^{-1})=\ell(w_{j}^{\,})-\ell(\tau_{j}^{-1})
\]
 hods, then we write $w\overset{-}{\leadsto}w'$ \emph{in $\CMcal O$.
}If furthermore $\ell(w_{j})=\ell(w_{j+1})$ holds for each $j$,
we write $w\overset{-}{\sim}w'$ and call the set of such $w'$ the
\emph{cyclic shift class of $w$ in $\CMcal O$} \cite[Définition 3.16]{MR1429870}.

\item If such a sequence $\tau_{n},\ldots,\tau_{0}$ yields a combination
of cyclic shifts and strong conjugations in $\CMcal O$, we write
$w\overset{\times}{\sim}w'$ \emph{in $\CMcal O$ }and call the set
of such $w'$ the \emph{mixed shift class of $w$ in }$\CMcal O$\emph{.}

\item When $W'\subseteq\tilde{W}$ is a standard parabolic subgroup
then we write \emph{by }$W'$ to express that the conjugators all
lie in $W'$, e.g.\ for $w\overset{\times}{\sim}w'$ by $W'$ we
mean that the $\tau_{j}$ all lie in $W'$.

\end{itemize}
\end{defn}

\begin{rem}
\label{rem:duality} This duality between strong conjugations and
cyclic shifts can be explicitly observed by multiplying with $w_{\circ}$
(e.g.\ \cite[§2.9]{MR1769289}). If we furthermore pick a twist $\delta$
such that $\delta w_{\circ}$ acts as $-1$ then it also exhibits
a duality between eigenvalues, in particular between the eigenvalues
1 and $-1$.
\end{rem}

\begin{defn}
Let $W$ be a twisted finite Coxeter group and pick a standard parabolic
subgroup $W'\subseteq\tilde{W}$. Given two elements $w$ and $w'$
in $W$ and a sequence of eigenspaces $\underline{\Theta}$ of $w$,
we let $\mathrm{Tran}_{W'}^{\underline{\Theta},\star}(w,w')$ (resp.\
$\mathrm{Tran}_{W'}^{\Theta,\star}(w,w')$) for $\star\in\{\leftrightarrow,+,-,\times\}$
denote the set of sequences of elements $\tau_{n},\ldots,\tau_{0}$
in $W'$ such that the sequence of elements
\[
w,\qquad\tau_{0}^{\,}w\tau_{0}^{-1},\qquad\ldots,\qquad\tau_{n}^{\,}\cdots\tau_{0}^{\,}w\tau_{0}^{-1}\cdots\tau_{n}^{-1}=w'
\]
 lies in in $\CMcal O^{\underline{\Theta}}$ (resp.\ $\CMcal O^{\Theta}$)
and consists of simple shifts when $\star=\,\leftrightarrow$, of
strong conjugations when $\star=+$, etc. Multiplication then yields
a natural projection map of \emph{transporters}
\[
\mathrm{Tran}_{W'}^{\underline{\Theta},\star}(w,w')\subseteq\mathrm{Tran}_{W'}^{\Theta,\star}(w,w')\longrightarrow\mathrm{Tran}_{W'}^{\,}(w,w'):=\{\tau\in W':\tau w\tau^{-1}=w'\}.
\]
 If $w=w'$ then we simplify this notation to $\mathrm{Tran}_{W'}^{\underline{\Theta},\star}(w)$,
$\mathrm{Tran}_{W'}^{\Theta,\star}(w)$ and $\mathrm{Tran}_{W'}^{\,}(w)$
respectively. Furthermore when $V_{\Theta}=V$ or $W'=\tilde{W}$
we might drop $\Theta$ or $W'$ from notation.
\end{defn}

\begin{notation}
Every Coxeter group $\tilde{W}$ is the codomain of a natural projection
map $b\mapsto w_{b}$ coming from the corresponding Artin-Tits braid
group $\mathrm{Br}_{\tilde{W}}$. Matsumoto's theorem \cite{MR183818}
furnishes a natural section $\tilde{w}\mapsto b_{\tilde{w}}$, embedding
$\tilde{W}$ (as a set) into the the associated braid monoid $\mathrm{Br}_{\tilde{W}}^{+}\subseteq\mathrm{Br}_{\tilde{W}}^{\,}$.
This projection and this section naturally extend to twisted Coxeter
groups, e.g.\ for the inclusion we will write
\[
W:=\Omega\ltimes\tilde{W}\hooklongrightarrow\Omega\ltimes\mathrm{Br}_{\tilde{W}}^{+}=:\mathrm{Br}_{W}^{+}=:\mathrm{Br}^{+},\qquad w=\delta\tilde{w}\longmapsto\delta b_{\tilde{w}}=:b_{w},
\]
and such braids $b_{w}$ are often called \emph{reduced }(or\emph{
simple} or \emph{minimal})\emph{ braids. }We will also write $b_{i_{l}\cdots i_{1}}:=b_{s_{i_{l}}\cdots s_{i_{1}}}$
when $s_{i_{l}}\cdots s_{i_{1}}$ is a reduced decomposition.
\end{notation}

\begin{propositiona} \label{prop:mixed-shift} Let $W$ be a twisted
Coxeter group and pick two elements $w$ and $w'$ in $W$.

\begin{enumerate}[\normalfont(i)]

\item If $w\overset{\times}{\sim}w'$ via a sequence of elements
$\tau_{n},\ldots,\tau_{0}$ of $W$, then $(\tau_{n}\cdots\tau_{0})(\mathfrak{R}_{\mathrm{st}}^{w})=\mathfrak{R}_{\mathrm{st}}^{w'}$.

\end{enumerate}

In particular, if $w$ and $w'$ are convex and one of them is firmly
convex, then so is the other one.

\begin{enumerate}[\normalfont(i)]

\setcounter{enumi}{1}

\item Pick a standard parabolic subgroup $W'\subseteq\tilde{W}$.
If $\tilde{W}$ is finite, then the projection map
\[
\mathrm{Tran}_{W'}^{\times}(w,w')\longrightarrow\mathrm{Tran}_{W'}^{\,}(b_{w},b_{w'}):=\{\tau\in\mathrm{Br}_{W'}^{+}:\tau b_{w}=b_{w'}\tau\}
\]
sending a sequence of conjugators $\tau_{n},\ldots,\tau_{0}$ to the
braid $b_{\tau_{n}}\cdots b_{\tau_{0}}$, surjects.

\end{enumerate}

In particular, $w\overset{\times}{\sim}w'$ by $W'$ if and only if
$b_{w}$ and $b_{w'}$ are conjugate under the braid subgroup $\mathrm{Br}_{W'}$.

\end{propositiona}

Part (ii) is proven in the more general language of Garside categories,
whereas part (i) will be used in §\ref{subsec:Bruhat} and in proving
\cite[Proposition(iii)]{WM-cross}.
\begin{example}
Consider the conjugates $w=s_{1}s_{2}s_{3}s_{4}s_{1}$ and $w'=s_{2}s_{3}s_{4}s_{2}s_{3}$
of the element $s_{2}s_{3}s_{4}$ in type $\mathsf{A}_{4}$. Then
$\mathfrak{R}_{\mathrm{st}}^{w}=\varnothing$ but as $\mathfrak{R}_{\mathrm{st}}^{w'}$
contains (the orbit of) $\alpha_{1}$, it now follows that the braids
$b_{w}$ and $b_{w'}$ are not conjugate in the corresponding Artin-Tits
braid group.
\end{example}

Elements of $\mathrm{Tran}^{-}(w)$ yield universal homeomorphisms
of Deligne-Lusztig varieties (\cite[Case 1 of Theorem 1.6]{MR393266}
and \cite[§0.6]{MR2907958}). In his work on the relationship between
character sheaves and the character theory of finite reductive groups,
Lusztig conjectured that for an elliptic element of minimal length
$w$, the natural map from $\mathrm{Tran}^{\leftrightarrow}(w)\subseteq\mathrm{Tran}^{-}(w)$
onto the centraliser $\mathrm{Tran}(w)$ is a surjection \cite[§1.2]{MR2907958}.
This statement also plays a rôle in work of Broué, Digne and Michel
on consequences of the Broué abelian defect conjecture for finite
reductive groups \cite{MR1429870,MR3187647}. He-Nie subsequently
proved this conjecture by showing more generally that for any other
such element $w'$, the set $\mathrm{Tran}^{\leftrightarrow}(w,w')$
projects onto $\mathrm{Tran}(w,w')$ \cite[Theorem 4.2]{MR2999317}.

Unwinding and advancing the work of \cite[§1-§4]{MR2999317}, in Lemma
\ref{lem:shifts} we will study elements in $\CMcal O^{\underline{\Theta}}$
for sequences of eigenspaces $\underline{\Theta}$ satisfying $\mathfrak{H}_{\Theta}\subseteq\mathfrak{H}^{w}$;
this will be required for the second half of the following 

\begin{propositiona} \label{prop:shifts-braiding} Let $W$ be a
twisted finite Coxeter group. Let $\CMcal O$ be a $\tilde{W}$-orbit
and let $\Theta$ be a set of eigenspaces of some element $w$ in
$\CMcal O$.

\begin{enumerate}[\normalfont(i)]

\item There exists an element $w'$ in $\CMcal O_{\pm}^{\Theta}$
such that $w\overset{\pm}{\rightarrow}w'$ in $W$, $w\overset{\pm}{\leadsto}w'$
in $\CMcal O^{\Theta}$ and $w'\overset{\mp}{\leadsto}w$ in $\CMcal O^{\Theta}$.

\end{enumerate}

Now assume that $\mathfrak{H}_{\Theta}\subseteq\mathfrak{H}^{w}$.

\begin{enumerate}[\normalfont(i)]

\setcounter{enumi}{1}

\item Then $\CMcal O_{\pm}^{\Theta}$ is contained in $\CMcal O{}_{\mathrm{max}/\mathrm{min}}^{\Theta}$
and the length of these elements is given by
\begin{equation}
\ell(\CMcal O_{\mathrm{max}/\mathrm{min}}^{\Theta})=2\sum_{i=1}^{m}\theta_{i}\bigl|\mathfrak{H}_{F_{i-1}}\backslash\mathfrak{H}_{F_{i}}\bigr|,\label{eq:length-formula}
\end{equation}
where the filtration $F_{i}$ and ordering on the arguments $\theta_{i}$
are derived from the sequence $\underline{\Theta}_{\pm}$.

\item Suppose $w,w'$ lie in $\CMcal O_{\mathrm{max}/\mathrm{min}}^{\Theta}$.
Then the projection
\[
\mathrm{Tran}^{\Theta,\star}(w,w')\longrightarrow\mathrm{Tran}(w,w')\qquad\bigl(\textrm{resp.}\qquad\mathrm{Tran}^{\star}(w,w')\longrightarrow\mathrm{Tran}(w,w')\bigr)
\]
 surjects, with
\[
\star=\begin{cases}
\pm & \text{if }\mp1\notin\mathrm{eig}(\Theta),\\
\mp & \text{if }V_{\Theta}=V,\\
\times & \textrm{otherwise}.
\end{cases}\qquad\bigl(\textrm{resp.}\qquad\star=\begin{cases}
\leftrightarrow & \text{if }\mp1\notin\mathrm{eig}(\Theta)\text{ and }V_{\Theta}=V,\\
\leftrightarrow & \text{if }\mathrm{eig}(\Theta)\cap\{1,-1\}=\varnothing.
\end{cases}\bigr)
\]

\end{enumerate}

\end{propositiona}

Formula (\ref{eq:length-formula}) was proven for elliptic elements
of minimal length in \cite[Theorem 2.3]{yun2020minimal}.
\begin{example}
From (iii) one obtains the usual statement that minimal length elements
lie in the same strong conjugacy class, and furthermore lie in the
same simple shift class when they are elliptic.

Dually, they also imply that maximal length elements lie in the same
cyclic shift class, and furthermore lie in the same simple shift class
if $-1$ is not an eigenvalue.
\end{example}

\begin{example}
Let $\CMcal O$ be the conjugacy class of $s_{1}s_{3}$ in type $\mathsf{B}_{3}$.
If we set $\Theta=\{-1\}$ or $\Theta=\{e^{2\pi i/3},-1\}$ then
\[
\CMcal O_{+}^{\Theta}=\{w_{\circ}s_{1},w_{\circ}s_{2}\},
\]
so the involutions $w_{\circ}s_{1}$ and $w_{\circ}s_{2}$ lie in
the same cyclic shift class,\footnote{In some books and papers, it is claimed that $w\overset{-}{\sim}w'$
if and only if $w\leftrightarrow w'$. However, conjugating these
two elements in $\mathsf{B}_{3}$ by simple reflections either leaves
them unchanged or results in an element of lower length, thus providing
a counterexample.

This example also shows that the strong conjugations in \cite[Proposition 2.8(2) and 2.13(2)]{MR2657442}
must be replaced by cyclic shifts. The subsequent proofs \cite[Proposition 2.9 and 2.14]{MR2657442}
about Bruhat cells don't seem to work for either type of shift, but
they can be replaced with Proposition \ref{prop:conjugacy-cyclic}.} and indeed there are reduced decompositions $w_{\circ}s_{1}=xy$
and $w_{\circ}s_{2}=yx$ for e.g.\ $x:=s_{3}s_{2}s_{3}s_{1}s_{2}s_{3}$
and $y:=s_{1}s_{2}$.
\end{example}

\begin{example}
Consider $w=s_{1}s_{2}s_{1}$ and $v=s_{2}s_{3}s_{2}$ in type $\mathsf{A}_{3}$.
Then $v$ lies in $\mathrm{Tran}(w)$ but $\mathrm{Tran}^{\times}(w)$
is generated by $w$; the centraliser of $b_{w}$ in the braid group
is generated by
\[
b_{121}\qquad\textrm{and}\qquad b_{321}b_{123}.
\]
\end{example}

\begin{defn}
Let $W$ be a twisted finite Coxeter group. Let $\CMcal O$ be a $\tilde{W}$-orbit
and let $\Theta$ be all of its eigenvalues except 1, so $V_{\Theta}=V_{w}$
when we pick an element $w$ in $\CMcal O$. We will sometimes write
$\mathrm{dom}$ instead of $\Theta$, e.g.\
\[
\CMcal O^{\mathrm{dom}}:=\CMcal O^{\Theta},\qquad\textrm{so that}\qquad\CMcal O_{\mathrm{max}/\mathrm{min}}^{\mathrm{dom}}:=\CMcal O_{\mathrm{max}/\mathrm{min}}^{\Theta},
\]
and refer to their elements as \emph{dominant} and \emph{maximally/minimally
dominant }respectively.
\end{defn}

In other words, an element $w$ is dominant if and only if the closure
of the dominant Weyl chamber contains an open subset of $V_{w}$.
From Corollary \ref{cor:dom} it follows that for elements of $W$
there are implications
\[
\textrm{elliptic or has anisotropic braiding sequence of eigenspaces}\Longrightarrow\textrm{dominant}\Longrightarrow\textrm{firmly convex}\Longrightarrow\textrm{convex}.
\]

Some of the properties of maximally and minimally dominant elements
follow from simple geometric statements, others from specialising
Proposition \ref{prop:shifts-braiding}:
\begin{lem*}
Let $W$ be an irreducible twisted finite Coxeter group and let $\CMcal O$
be a $\tilde{W}$-orbit.

\begin{enumerate}[\normalfont(i)]

\item Then there is an equality of lengths $\ell(\CMcal O_{\mathrm{max}}^{\mathrm{dom}})=\ell(\CMcal O_{\mathrm{max}})$
and an inequality of lengths
\[
\ell(\CMcal O_{\mathrm{min}}^{\mathrm{dom}})\geq\ell(\CMcal O_{\mathrm{min}}),
\]
with equality if and only if $\CMcal O$ is either elliptic or trivial
(meaning, the orbit of the identity element).

\item There is an inclusion $\CMcal O_{\pm}^{\mathrm{dom}}\subseteq\CMcal O_{\mathrm{max}/\mathrm{min}}^{\mathrm{dom}}$,
and for any pair of elements $w$ and $w'$ in $\CMcal O_{\mathrm{max}/\mathrm{min}}^{\mathrm{dom}}$
the natural projection map
\[
\mathrm{Tran}^{\mathrm{dom},-}(w,w')\longrightarrow\mathrm{Tran}(w,w')
\]
 is surjective.

\item If $\CMcal O$ is nontrivial and lies in $\tilde{W}$ then
there are inequalities
\[
\ell(\CMcal O_{\mathrm{min}}^{\mathrm{dom}})\geq\frac{|\mathfrak{R}|-\ell_{f}(\CMcal O)}{\mathrm{ord}(\CMcal O)}\qquad\textrm{and}\qquad\ell(\CMcal O_{\mathrm{min}}^{\mathrm{dom}})\geq\mathrm{rank}(W),
\]
where the first one is equality if and only if $\CMcal O$ is quasiregular,
the second one if and only if $\CMcal O$ is the Coxeter class.

\end{enumerate}
\end{lem*}
\begin{example}
The conjugates $w=s_{1}s_{2}s_{3}s_{4}s_{2}s_{3}s_{1}$ and $w^{-1}$
of $s_{2}s_{3}s_{4}$ in type $\mathsf{B}_{4}$ are dominant, but
are neither maximally nor minimally dominant; they form two distinct
cyclic shift classes.
\end{example}

 Richardson showed that the minimal length involutions of a finite
Coxeter group are those involutions $w$ such that $\mathfrak{R}_{w}$
agrees with a standard parabolic subsystem \cite[Theorem A(a)]{MR679916},
and he used that to algorithmically classify conjugacy classes of
involutions \cite[§3]{MR679916} in terms of the corresponding Coxeter-Dynkin
diagram. A similar description for maximal length involutions was
given by Hart-Rowley \cite[Theorem 1.1(ii)]{MR2040170}, who also
gave, for each conjugacy class of involutions in every type, explicit
expressions for (and deduced from that the number of) maximal and
minimal length elements \cite{MR1892601}. We show that involutions
have maximal length precisely when they are dominant, and generalise
other statements regarding maximal and minimal length involutions
to orbits of standard parabolic subgroups:
\begin{notation}
Let $W$ be a twisted finite Coxeter group and $W'\subseteq\tilde{W}$
be a standard parabolic subgroup. Then we denote by $_{W'}\mathfrak{R}_{+}$
the subset of positive roots of the corresponding standard parabolic
subsystem. Given an element $w$ in $W$, we let $\mathfrak{R}_{-1}^{w}:=\mathfrak{R}_{w}\cap V_{-1}^{w}$
denote the set of roots on which it acts as $-1$.
\end{notation}

\begin{propositiona} \label{prop:involutions} Let $W$ be a twisted
Coxeter group, pick a standard parabolic subgroup $W'\subseteq\tilde{W}$,
let $\CMcal O$ be a $W'$-orbit consisting of involutions and pick
an element $w$ in $\CMcal O$.

\begin{enumerate}[\normalfont(i)]

\item There exists an element $w'$ in $\CMcal O_{\mathrm{max/min}}$
such that $w\overset{\pm}{\twoheadrightarrow}w'$ by $W'$.

\item The following are equivalent:

\begin{enumerate}[\normalfont(a)]

\item The element $w$ lies in $\CMcal O_{\mathrm{max/min}}$.

\item The set of roots $\mathfrak{R}_{w}\cap{}_{W'}\mathfrak{R}_{+}$
is given by 
\begin{equation}
_{W'}\mathfrak{R}_{+}\backslash\mathfrak{R}^{w}\qquad(\textrm{resp.}\quad{}_{W'}\mathfrak{R}_{+}\cap\mathfrak{R}_{-1}^{w}).\label{eq:max-min-involution}
\end{equation}

\item $w$ satisfies the explicit description of (\ref{eq:explicit-description-involution}),
which is amenable to an algorithmic classification.

\end{enumerate}

If $W'=\tilde{W}$, then this is also equivalent to

\begin{enumerate}[\normalfont(a)]

\setcounter{enumii}{3}

\item The closure of the dominant Weyl chamber contains an open subset
of $V_{\mp1}^{w}$.

\end{enumerate}

\end{enumerate}

In particular, there is an equality
\[
|\mathfrak{R}|=\ell_{f}(\CMcal O)+2\ell(\CMcal O^{\mathrm{dom}}).
\]

\end{propositiona}
\begin{example}
Let $\mathfrak{R}$ be a root system and let $W$ be its finite Coxeter
group. If $w$ is a reflection in some root, then $V_{w}$ equals
the line through it. Thus the only dominant element in the conjugacy
class $\CMcal O$ of reflecting in a long (resp.\ short) root is
reflecting in a root that lies in the closure of the dominant Weyl
chamber; when $\mathfrak{R}$ is crystallographic, it is well-known
that this is the highest (resp.\ highest short) root.\footnote{Outside of type $\mathsf{A}$ these are also the only convex reflections;
in type $\mathsf{A}$, the only other ones are the two simple reflections
corresponding to the endpoints of the Coxeter-Dynkin diagram.} Denoting by $h^{\vee}$ the dual Coxeter number\footnote{It follows from \cite{MR254174} that for Weyl groups the dual Coxeter
number $h^{\vee}$ agrees with $\bigl(\ell(r_{\beta})+3\bigr)/2$,
where $r_{\beta}$ is the reflection in a long root $\beta$ lying
in the closure of the dominant Weyl chamber; here we extend its definition
to non-crystallographic root systems by defining it at such. For $\mathsf{H}_{3}$,
$\mathsf{H}_{4}$ and $\mathsf{I}_{2}(m)$ we then find that $h^{\vee}$
is given by $8$, $24$ and $\lceil m/2\rceil+1$ respectively.} (resp.\ dual Coxeter number of the dual root system) of $\mathfrak{R}$,
the length of this reflection then yields
\[
\ell(\CMcal O^{\mathrm{dom}})=2h^{\vee}-3,\qquad\ell_{f}(\CMcal O)=|\mathfrak{R}|-4h^{\vee}+6.
\]
\end{example}

In his sole paper (which is based on his 1965 PhD thesis), Garside
gave a solution to the conjugacy problem (posed in \cite{MR3069440})
for the (type $\mathsf{A}$) braid group by introducing the natural
submonoid $\mathrm{Br}_{\tilde{W}}^{+}$ and proving that for any
of its elements $\tilde{b}\in\mathrm{Br}_{\tilde{W}}^{+}$ there exists
a greatest common (right) divisor $\mathrm{DG}(\tilde{b})$ for $\tilde{b}$
and Garside's ``fundamental braid'' $b_{w_{\circ}}$ \cite{MR248801}.
A few years later, Deligne rephrased the existence of this divisor
as a property of paths on the simplicial arrangement corresponding
to $\tilde{W}$, and used that to settle the $K(\pi,1)$-conjecture
for finite Coxeter groups \cite{MR422673}. For any braid $b=\delta\tilde{b}$
in $\mathrm{Br}^{+}$ we may construct this factor $\mathrm{DG}_{1}(b):=\mathrm{DG}(b):=\mathrm{DG}(\tilde{b})$
and then remove it from $b$; repeating this procedure on the remainder
then yields a unique factorisation 
\[
\mathrm{DGN}(b):=\delta\mathrm{DG}_{m}(b)\cdots\mathrm{DG}_{2}(b)\mathrm{DG}_{1}(b)
\]
for any element in $\mathrm{Br}^{+}$, and Thurston proved that this
normal form is computable in quadratic time complexity and linear
space complexity \cite{MR1161694,MR1157320}. The normal form is sometimes
used to prove faithfulness of representations (e.g.\ \cite{MR1888796,MR3648506}),
and in type $\mathsf{A}$ it can be used \cite{MR1361083,MR2869449}
to determine how elements lie in the Nielsen-Thurston classification
\cite{MR956596}; generically they are pseudo-Anosov (e.g.\ \cite{MR2772067}).
The conjugacy problem and normal form have been generalised beyond
Artin-Tits groups of finite type \cite{MR3362691}, and have recently
been employed for studying braid group-based post-quantum cryptography
(e.g.\ \cite{MR3950949}).
\begin{notation}
For any integer $i\geq1$ and element $b$ in $\mathrm{Br}^{+}$ we
denote the product of the first $i$ Deligne-Garside factors by $\mathrm{DG}_{i\geq}(b):=\mathrm{DG}_{i}(b)\cdots\mathrm{DG}_{2}(b)\mathrm{DG}_{1}(b)$,
often implicitly identify a reduced braid like $\mathrm{DG}_{i}(b)$
with the corresponding element of $\tilde{W}$.
\end{notation}

In order to better understand the normal form of an arbitrary braid
$b$ of $\mathrm{Br}^{+}$ (and of $b_{w}^{d}$ in particular), we
associate to it in §\ref{subsec:inversion-seq} a subset of roots
$\mathfrak{R}_{b}$ extending inversion sets of twisted Coxeter group
elements, and show in Lemma \ref{lem:bounding} that this set has
bounding properties.

\begin{propositiona} \label{prop:dg-bound} For any element $w$
in a twisted finite Coxeter group and integer $d\geq1$ we have an
inclusion of sets of roots $\mathfrak{R}_{b_{w}^{d}}\subseteq\mathfrak{R}\backslash\mathfrak{R}_{\mathrm{st}}^{w}$.
This then implies that for any\footnote{The bounds for $i=1$ will reproven in the sequel, using different
techniques \cite[Corollary 2.31]{WM-cross}.} $i\geq1$

\begin{enumerate}[\normalfont(i)]

\item there is an inclusion 
\begin{equation}
\mathfrak{R}_{\mathrm{DG}_{i}(b_{w}^{d})}\subseteq\mathfrak{R}_{+}\backslash\bigl(\mathrm{DG}_{i-1}(b_{w}^{d})\cdots\mathrm{DG}_{1}(b_{w}^{d})(\mathfrak{R}_{\mathrm{st}}^{w})\bigr),\label{eq:roots-of-dg}
\end{equation}
so in particular

\[
\mathfrak{R}_{w}\subseteq\mathfrak{R}_{\mathrm{DG}(b_{w}^{d})}\subseteq\mathfrak{R}_{+}\backslash\mathfrak{R}_{\mathrm{st}}^{w}\qquad\textrm{and}\qquad\ell\bigl(\mathrm{DG}_{i}(b_{w}^{d})\bigr)\leq|\mathfrak{R}_{+}\backslash\mathfrak{R}_{\mathrm{st}}^{w}|.
\]
\item Moreover, set
\[
d:=i\big(|\mathfrak{R}_{+}\backslash\mathfrak{R}_{\mathrm{st}}^{w}|-\ell(w)+1\big),
\]
 then $\mathrm{DG}_{i\geq}(b_{w}^{d})$ has ``stabilised'', i.e.\
for any $d'\geq d$ we have 
\[
\mathrm{DG}_{i\geq}(b_{w}^{d'})=\mathrm{DG}_{i\geq}(b_{w}^{d}).
\]
\end{enumerate} \end{propositiona}
\begin{example}
Consider $w=s_{3}s_{1}s_{2}s_{1}$ in type $\mathsf{B}_{3}$. Then
$\mathfrak{R}_{\mathrm{st}}^{w}=\mathfrak{R}^{w}=\{\pm\alpha_{23}\}$
and by induction on $d\geq1$ one finds\begin{alignat*}{2}
& \mathrm{DGN}(b_{w}^{2d}) && =b_{32}^{\,}b_{w_{\circ}s_{3}}^{d-1}b_{121321}^{\,},\\
& \mathrm{DGN}(b_{w}^{2d+1}) && =b_{321232}^{\,}b_{w_{\circ}s_{3}}^{d-1}b_{121321}^{\,}.
\end{alignat*} 
\end{example}

\begin{example}
Consider $w=s_{2}s_{4}s_{3}s_{2}s_{1}$ in type $\mathsf{D}_{4}$.
Then $\mathfrak{R}^{w}=\varnothing$ but as $w^{2}(\alpha_{3})=w(\alpha_{4})=\alpha_{3}$
it now immediately follows that 
\[
w_{\circ}s_{4}s_{3}\geq\mathrm{DG}(b_{w}^{d})
\]
for all $d\geq0$. One may compute that this is an equality for $d>1$,
conversely implying that $\mathfrak{R}_{\mathrm{st}}^{w}=\{\alpha_{3},\alpha_{4}\}$.
\end{example}

It is explained in the sequel to this paper \cite{WM-cross} that
convex elements $w$ of twisted finite Weyl groups whose braid lifts
attain the upper bound $\mathfrak{R}_{\mathrm{DG}(b_{w}^{d})}=\mathfrak{R}_{+}\backslash\mathfrak{R}_{\mathrm{st}}^{w}$
for at least one integer $d\geq0$ yield transverse slices in the
corresponding twisted reductive groups, generalising the slices of
He-Lusztig and Sevostyanov. Proposition \ref{prop:closed-complement}
explains that this upper bound is never attained when $w$ is not
convex, so equivalently $w$ should be a convex element satisfying
the braid equation
\begin{equation}
\mathrm{DG}(b_{w}^{d})=\mathrm{pb}(w)\tag{\ensuremath{*}}\label{eq:braid-equation}
\end{equation}
for some $d$; the previous proposition implies that satisfiability
can checked by picking any $d>|\mathfrak{R}_{+}\backslash\mathfrak{R}_{\mathrm{st}}^{w}|-\ell(w)$.

The original aim of this paper was to show that \eqref{braid-equation}
is easily computable and is satisfied by a large number of twisted
finite Coxeter group elements, including those that appear in the
cross section isomorphisms of He-Lusztig and Sevostyanov:
\begin{thm*}
Let $W$ be a twisted finite Coxeter group and pick an element $w$
in $W$.

\begin{enumerate}[\normalfont(i)]

\item If $w$ has an anisotropic braiding sequence of eigenspaces,
then it satisfies (\ref{eq:braid-equation}) when $d\geq\mathrm{ord}(w)$.

\item If $w$ is convex and $w\overset{-}{\sim}w'$ in $W^{\mathrm{cx}}$
for an element $w'$ satisfiying (\ref{eq:braid-equation}) for some
$d$, then $w$ satisfies (\ref{eq:braid-equation}) when 
\[
d>|\mathfrak{R}_{+}\backslash\mathfrak{R}_{\mathrm{st}}^{w}|-\ell(w).
\]
In particular, (\ref{eq:braid-equation}) holds for all maximally
and minimally dominant elements.

\item If $w$ is quasiregular, minimally dominant and lies in $\tilde{W}$
then

\[
b_{w}^{\mathrm{ord}(w)}=b_{\mathrm{pb}(w)^{-1}}b_{\mathrm{pb}(w)^{\phantom{-1}}}\!\!\!\!\!,
\]
and conversely elements in $\tilde{W}$ satisfying this equation are
quasiregular and satisfy $\ell(w)=\ell(\CMcal O_{\mathrm{min}}^{\mathrm{dom}})$.

\item If $W$ is irreducible, then a nontrivial element lying inside
of a parabolic subgroup of lower rank is never convex. In particular,
the following are then equivalent if $w$ is of minimal length in
its conjugacy class:

\begin{enumerate}[(a)]

\item $w$ is trivial or elliptic,

\item $w$ is convex,

\item $w$ is firmly convex and satisfies the braid equation (\ref{eq:braid-equation})
for some (or any)
\begin{equation}
d>|\mathfrak{R}_{+}|-\ell(w).\label{eq:bound}
\end{equation}

\end{enumerate}

\item If $w$ has maximal or minimal length or is minimally dominant,
then the Artin-Tits braid $b_{w}$ is not pseudo-Anosov.

\end{enumerate}
\end{thm*}
\begin{rem}
After replacing $\mathfrak{R}_{v^{\,}}$ by $\mathfrak{R}_{v^{-1}}$
and the operators in (\ref{eq:roots-of-dg}) by their inverses, the
exact same statements are true for the left Deligne-Garside normal
form.
\end{rem}

Proposition \ref{prop:braiding-dgn} explicitly computes the (right)
Deligne-Garside normal form of $b_{w}^{\mathrm{ord}(w)}$ for any
twisted finite Coxeter group element $w$ that has a braiding sequence
of eigenspaces. Some additional properties of this normal form (that
require different techniques) are proven in the sequel \cite{WM-cross}. 

This paper and its sequel indicate that conjugacy classes in reductive
groups appear to be more closely related to minimally dominant elements
than to minimal length elements \cite[Corollary B]{WM-cross}. This
suggests that it might be more natural to construct Lusztig's inverse
to the Kazhdan-Lusztig map \cite{MR2833465} and its refinement in
\cite{MR3495802} from minimally dominant elements instead of minimal
length elements; the end result would remain unchanged, according
to the following
\begin{conjecture*}
Suppose $W$ is the Weyl group of a maximal torus contained in a Borel
subgroup $B$ of a twisted reductive group over an algebraically closed
field. Let $w$ be a minimal length element in $W$ and let $w'$
a minimally dominant element in the conjugacy class of $w$. Then
for any conjugacy class $C$ of this reductive group, 
\[
C\cap BwB\neq\varnothing\qquad\textrm{if and only if}\qquad C\cap Bw'B\neq\varnothing.
\]
\end{conjecture*}
My approach to this problem is outlined in §\ref{subsec:Bruhat}.
\begin{rem}
Here are some other open problems: 

\begin{enumerate}[(i)]

\item Does Proposition \ref{prop:mixed-shift}(ii) extend to all
twisted Coxeter groups?

\item Generalise Proposition \ref{prop:shifts-braiding} and the
last part of Proposition \ref{prop:involutions} to standard parabolic
subgroups of arbitrary rank (as in \cite{MR3045154}).

\item Fix a classical type ($\mathsf{A}$, $\mathsf{B}$, $\mathsf{C}$
or $\mathsf{D}$) and pick a natural number $n$. Is the number of
\[
\textrm{convex/firmly convex/dominant/minimally-dominant/(}\ref{eq:braid-equation}\textrm{)-satisfying}
\]
elements of reflection length $n$ independent of $\mathrm{rank}(W)$
when $\mathrm{rank}(W)$ is sufficiently large?

\item Suppose $w$ is firmly convex and has maximal length, does
there exist a dominant element $w'$ such that $w\leftrightarrow w'$
in $W^{\mathrm{cx}}$?

\item Given two firmly convex elements $w,w'$ in a (parabolic) $\tilde{W}$-orbit
$\CMcal O$ such that $w\overset{-}{\sim}w'$ in $\CMcal O$, is it
true that $w\overset{-}{\sim}w'$ in $\CMcal O^{\mathrm{cx}}$?

\item Suppose that $w$ is firmly convex and satisfies the braid
equation (\ref{eq:braid-equation}). Does there exist a dominant element
$w'$ such that $w\overset{-}{\sim}w'$ in $W^{\mathrm{cx}}$? (If
so, this would imply that $\ell(w)\geq\ell(\CMcal O_{\mathrm{min}}^{\mathrm{dom}})$.
For the conjugacy classes appearing in Sevostyanov's case-by-case
work \cite{MR3883243}, that inequality follows from combining his
calculations, this work and the main theorem of the sequel \cite{WM-cross}.)

\end{enumerate}

\end{rem}

\paragraph{Acknowledgements}

 Some of the main results of this paper owe a large intellectual
debt to the ideas of the paper \cite{MR2999317}, and it is my pleasure
to thank the authors here for their highly original work. Furthermore,
I am grateful to Dominic Joyce, Balázs Szendr\H{o}i and the Mathematical
Institute of the University of Oxford for an excellent stay, where
some of this paper was written. This visit was supported by the Centre
for Quantum Geometry of Moduli Spaces at Aarhus University and the
Danish National Research Foundation. This work was also supported
by EPSRC grant EP/R045038/1, and some of the computations in this
paper were done with SageMath \cite{sagemath}.\footnote{The corresponding code is available upon request.}

\section{\label{sec:cyclic-shifts}Shifts and sequences of eigenspaces}

The main goal of this section is to prove Proposition \ref{prop:shifts-braiding}.
Before we commence, let us prove now that our definition of ``good
position'' (in Definition \ref{def:sequence-eigenspaces}) is equivalent
to the one provided in \cite[§5.2]{MR2999317}:
\begin{notation}
Given a set of hyperplanes $\mathfrak{H}'$ in a reflection representation
$V$, we write 
\[
V\backslash\mathfrak{H}':=V\backslash\bigcup_{\mathfrak{h}\in\mathfrak{H}'}\mathfrak{h}.
\]
\end{notation}

\begin{prop}
\label{prop:good-position} Let $w$ be an element of a twisted finite
Coxeter group with a sequence $\underline{\Theta}=(V_{m},\ldots,V_{1})$
of eigenspaces in the reflection representation $V$. Let $F_{i}:=\sum_{j=1}^{i}V_{j}$
denote corresponding filtration of $V$, as in Definition \ref{def:sequence-eigenspaces}.

Then this sequence $\underline{\Theta}$ is in good position w.r.t.\
the dominant Weyl chamber $C$, if and only if for each integer $i\in\{1,\ldots,m\}$
the closure of the connected component of $C$ in $V\backslash\mathfrak{H}_{F_{i-1}}$
contains a regular point of the subspace $V_{i}$.
\end{prop}

\begin{proof}
We induct on the length of the sequence: we assume that it is true
$<m$ and apply the induction hypothesis to the shortened sequence
$(V_{m-1},\ldots,V_{1})$, which is in good position w.r.t.\ $C$.
It then suffices to show that if the closure of the dominant Weyl
chamber $C$ contains a regular point of $F_{m-1}$, then it contains
a regular point of $F_{m}=F_{m-1}\oplus V_{m}$ if and only if the
closure of the connected component of $C$ in $V\backslash\mathfrak{H}_{F_{m-1}}$
contains a regular point of the subspace $V_{m}$. Let $v_{m-1}$
be a regular point of $F_{m-1}$ in $\overline{C}$, so $v_{m-1}=\sum_{i=1}^{\mathrm{rk}}c_{i}\omega_{i}$
with each $c_{i}\geq0$. Let $J$ denote the set of indices such that
$c_{j}=0$, or equivalently, the indices of simple roots whose hyperplanes
lie in $\mathfrak{H}_{F_{m-1}}$.

$\Leftarrow:$ By assumption, $V_{m}$ has a regular point $v_{m}=\sum_{i=1}^{\mathrm{rk}}c_{i}'\omega_{i}$
with $c_{i}'\geq0$ when $i$ lies in $J$. If a hyperplane does not
contain both $v_{m-1}$ and $v_{m}$ (and hence $F_{m}$), then it
only contains one linear combination of them (up to scalar); hence
by finiteness of the number of hyperplanes, there exists a point $v:=\varepsilon v_{m-1}+(1-\varepsilon)v_{m}$
for $\varepsilon$ positive but arbitrarily close to zero, which is
a regular point of $F_{m}$. But then
\[
\alpha_{i}(v)\begin{cases}
\approx\alpha_{i}(v_{m})>0 & \textrm{if }i\notin J,\\
=\varepsilon\alpha_{i}(v_{m-1})\geq0 & \textrm{if }i\in J,
\end{cases}
\]
 so $v$ still lies in the closure of the dominant Weyl chamber.

$\Rightarrow:$ By assumption $\overline{C}$ contains an open subset
of $F_{m}=F_{m-1}\oplus V_{m}$ consisting of regular points; we may
assume it is of the form $U_{m-1}\oplus U_{m}$ for bounded open subsets
$U_{m-1}$ of $F_{m-1}$ and $U_{m}$ of $V_{m}$. Using that $\overline{C}$
also contains an open subset of $F_{m-1}$, the next lemma (with $V':=F_{m-1}$
and $C':=F_{m-1}\cap\overline{C}$) implies that we may actually assume
that $U_{m-1}$ lies inside of $F_{m-1}\cap\overline{C}$: it translates
by a point of $F_{m-1}\cap\overline{C}\subseteq\overline{C}$, which
keeps $U_{m-1}\oplus U_{m}$ in $\overline{C}$. It now follows that
there is a regular point $v_{m-1}$ of $F_{m-1}$ lying in $\overline{C}$
such that for some regular $v_{m}$ in $V_{m}$ their sum $v=v_{m-1}+v_{m}$
is a regular point of $F_{m}$ lying in $\overline{C}$. If $v_{m}=\sum_{i=1}^{\mathrm{rk}}c_{i}'\omega_{i}$
does not satisfy the requirements, so $c_{i}'<0$ for some $i$ in
$J$, then $\alpha_{i}(v)=\alpha_{i}(v_{m})<0$ which contradicts
that $v$ lies in $\overline{C}$.
\end{proof}
\begin{lem}
Let $V'$ be a real vector space. If a convex cone $C'\subseteq V'$
contains an open subset of $V'$, then for any bounded open subset
$U$ of $V'$ there exists a vector $v$ in $C'$ such that $v+U\subset C'$.
\end{lem}

\begin{proof}
Since the convex cone contains an open subset, it contains a ball.
By multiplying with scalars, we can make the ball arbitrarily large,
so $C'$ also contains a ball the same size as a ball around the origin
containing $U$; now let $v$ be the translate of the origin to the
centre of this ball.
\end{proof}
The following lemma will be used several times throughout this paper:
\begin{defn}
If $w,x,y$ are elements of a twisted Coxeter group such that $b_{w}=b_{x}b_{y}$
(or equivalently, $w=xy$ and $\ell(w)=\ell(x)+\ell(y)$), then we
say that $w=xy$ is a \emph{reduced decomposition}. 
\end{defn}

\begin{notation}
When the conjugation action of an element $x$ permutes the simple
reflections of a standard parabolic subgroup $W'$, we write $\delta_{x}$
for the corresponding twist of $W'$.

Given two Weyl chambers $C,C'$ and a set of root hyperplanes $\mathfrak{H}'$,
we denote by $\mathfrak{H}'(C,C')$ the subset of hyperplanes in $\mathfrak{H}'$
separating $C$ from $C'$.
\end{notation}

\begin{lem}
\label{lem:factorise} Let $W$ be a twisted finite Coxeter group
and let $w$ be an element of $W$.

\begin{enumerate}[\normalfont(i)]

\item For any standard parabolic subgroup $W'\subseteq\tilde{W}$
preserved by $w$ under conjugation, the element $w$ admits a unique
reduced decomposition
\[
w=xy
\]
 into an element $x$ in $W$ and an element $y$ in $W'$, such that
$x$ is a minimal length \emph{double} coset representative with respect
to $W'$ and permutes its simple roots; thus the action of $x$ on
$V_{W'}$ agrees with $\delta_{x}$.

\item If the simple roots of $W'$ are all fixed by $w$, then $y$
is the identity element.

\end{enumerate}

Now let $\underline{\Theta}=(V_{m},\ldots,V_{1})$ be a sequence of
eigenspaces for $w$ such that the dominant Weyl chamber $C$ is in
good position.

\begin{enumerate}[\normalfont(i)]

\setcounter{enumi}{1}

\item Factorise the element $w=xy$ as in (i) for $W':=\tilde{W}_{V_{1}}$.
Then 
\[
\underline{\Theta}':=(V_{m}\cap V_{W'},\ldots,V_{1}\cap V_{W'})
\]
 is a sequence of eigenspaces (throwing away trivial ones) for the
element $\delta_{x}y$ of the twisted finite Coxeter group $\<\delta_{x}\>\ltimes W'$,
such that the dominant Weyl chamber for $\<\delta_{x}\>\ltimes W'$
is in good position.

\item Now factorise $w=xy$ as in (i) for $W'=\tilde{W}_{\Theta}$.
Then these elements have length
\[
\ell(x)=2\sum_{i=1}^{m}\theta_{i}\bigl|\mathfrak{H}_{F_{i-1}}\backslash\mathfrak{H}_{F_{i}}\bigr|,\qquad\ell(y)=\mathfrak{H}_{\Theta}\bigl(C,w(C)\bigr).
\]

\end{enumerate}
\end{lem}

\begin{proof}
(i): We follow the proof of \cite[Proposition 2.2]{MR2999317}: let
$w=y'xy''$ where $y'$ and $y''$ lie in $W'$ and $x$ is a minimal
length double coset representative for $W'$. From 
\[
(y'xy'')W'(y'xy'')^{-1}=wW'w^{-1}=W'
\]
 it follows that $x^{-1}W'x=W'$. The element $x^{-1}y'x$ and hence
also $y:=(x^{-1}y'x)y''$ then lie in $W'$. If $s_{i}$ is a simple
reflection of $W'$ then $xs_{i}=(xs_{i}x^{-1})x$ has length $\ell(x)+1$,
so as the element $xs_{i}x^{-1}$ lies in $W'$ it must have length
1.

For the final statement it suffices to note that a minimal coset representative
does not make any simple root in $V_{W'}$ negative.

(ii): By the last remark in (i), $\delta_{x}y$ acts as $w$ on $V_{W'}\subseteq V^{w}$,
so it acts as the identity. Using that the dominant chamber in $V_{W'}$
is a fundamental domain for $W'$ which is also preserved by $\delta_{x}$,
it follows that $y$ acts as the identity.

(iii): Since $w$ preserves $V_{1}$, its conjugation action preserves
the standard parabolic subgroup $W'$ so we may factorise as in (i).
As it consequently permutes the simple reflections of $W'$ it preserves
$V_{W'}$ so $V_{i}\cap V_{W'}$ is indeed an eigenspace for $w$
and $\delta_{x}y$. If $C$ contains an open subset of $F_{i}$ then
as projection maps are open, the orthogonal projection of $C$ to
$V_{W'}$ contains an open subset of $F_{i}\cap V_{W'}$ (which by
orthogonality of eigenspaces is the projection of $F_{i}$); but the
orthogonal projection of $C$ to $V_{W'}$ is precisely the dominant
Weyl chamber for $W'$.

(iv): Using (iii), this now follows by induction on the length of
$\Theta$ from the case where $\Theta$ consists of just one eigenspace,
which is \cite[Proposition 2.2]{MR2999317}.
\end{proof}
From (ii) and (iv) we now deduce the first part of the following crucial 
\begin{lem}
\label{lem:shifts} Let $W$ be a twisted finite Coxeter group, let
$\CMcal O$ be a $\tilde{W}$-orbit and let $\underline{\Theta}=(V_{m},\ldots,V_{1})$
be a sequence of eigenspaces of some element $w$ in $\CMcal O$,
satisfying $\mathfrak{H}_{\Theta}\subseteq\mathfrak{H}^{w}$.

\begin{enumerate}[\normalfont(i)]

\item Let the corresponding sequence of normalised principal arguments
be $(\theta_{m},\ldots,\theta_{1})$, then
\begin{equation}
\ell(\CMcal O^{\underline{\Theta}})=2\sum_{i=1}^{m}\theta_{i}\bigl|\mathfrak{H}_{F_{i-1}}\backslash\mathfrak{H}_{F_{i}}\bigr|.\label{eq:length-formula-braiding}
\end{equation}

\item Suppose $w,w'$ lie in $\CMcal O^{\underline{\Theta}}$. Then
the projection
\[
\mathrm{Tran}^{\underline{\Theta},\star}(w,w')\longrightarrow\mathrm{Tran}(w,w')\qquad\bigl(\textrm{resp.}\qquad\mathrm{Tran}^{\star}(w,w')\longrightarrow\mathrm{Tran}(w,w')\bigr)
\]
 surjects, with
\[
\star=\begin{cases}
\pm & \text{if }\mp1\notin\mathrm{eig}(\Theta),\\
\times & \textrm{otherwise}.
\end{cases}\qquad\bigl(\textrm{resp.}\qquad\star=\,\leftrightarrow\quad\text{if }\mathrm{eig}(\Theta)\cap\{1,-1\}=\varnothing.\bigr)
\]
When the element $y\in\tilde{W}_{F_{m-1}}$ in the splitting $w=xy$
for $F_{m-1}$ in Lemma \ref{lem:factorise} equals $\{\pm\mathrm{id}\}$
or has no eigenvalues in $\{\pm1\}$ (e.g.\ if $V_{F_{m-1}}\subseteq V_{m}$,
which happens when $V_{\Theta}=V$), we only need the eigenvalues
of $(V_{m-1},\ldots,V_{1})$ here.

\end{enumerate}
\end{lem}

\begin{proof}
(i): For any $w$ in $\CMcal O^{\underline{\Theta}}$, factorise it
into $w=xy$ as in part (iv) of the previous lemma. From $\mathfrak{H}_{\Theta}\subseteq\mathfrak{H}^{w}$
it follows that $V_{W'}\subseteq V^{w}$, so part (ii) of that lemma
implies that $y$ is the identity.
\end{proof}
Part (ii) of this lemma will be proven in §\ref{subsec:shifts}.

\subsection{Quasiregular elements and the Coxeter plane}

In this subsection we briefly recall classical statements involving
the eigenspace decomposition of regular elements and Coxeter elements,
and recast them in our framework. It yields a helpful perspective
on the relationship between Sevostyanov's and He-Lusztig's elements,
illustrates how Lemma \ref{lem:shifts}(i) can be applied and might
serve as a suitable warm-up to some of the ideas that follow.
\begin{defn}
Recall that an element of a twisted finite Coxeter group is called
\emph{regular} if it has a \emph{regular} eigenvector, i.e.\ a complex
eigenvector in the complexified reflection representation which is
not contained in any complexified root hyperplane.

We say that a complex eigenvector of an element $w$ is \emph{quasiregular}
if it is not contained in any complexified root hyperplane, unless
that hyperplane comes from $\mathfrak{H}^{w}$.
\end{defn}

\begin{prop}
\label{prop:quasiregular} Let $w$ be an element of a twisted finite
Coxeter group with complex eigenvalue $\lambda$, and let $\mathfrak{H}'$
be a subset of root hyperplanes in the real reflection representation.
Then $\mathfrak{H}_{V_{\lambda}^{w}}\subseteq\mathfrak{H}'$ if and
only if $w$ has a complex eigenvector for the eigenvalue $\lambda$
which is only contained in complexified root hyperplanes corresponding
to those in $\mathfrak{H}'$.

In particular, $w$ satisfies $\mathfrak{H}_{V_{\lambda}^{w}}=\varnothing$
(resp.\ $\mathfrak{H}_{V_{\lambda}^{w}}\subseteq\mathfrak{H}^{w}$)
for some $\lambda$ if and only if it is (quasi)regular.
\end{prop}

\begin{proof}
Recall that the complexification of $V_{\lambda}^{w}$ is the sum
of the complex eigenspaces corresponding to $\lambda$ and $\lambda^{-1}$,
so in other words
\begin{equation}
V_{\lambda}^{w}\otimes_{\mathbb{R}}\mathbb{C}=(V\otimes_{\mathbb{R}}\mathbb{C})_{\lambda}^{w}\oplus(V\otimes_{\mathbb{R}}\mathbb{C})_{\lambda^{-1}}^{w}.\label{eq:complexified-eigenspace}
\end{equation}

$\Rightarrow$: Pick a regular element in the complex eigenspace of
$\lambda$ (for $V\otimes_{\mathbb{R}}\mathbb{C}$ and the complexified
root hyperplanes). If it is contained in a complexified root hyperplane
then by regularity that complex hyperplane contains the entire complex
eigenspace. But since it came from a real subspace the complexified
root hyperplane must also contain its complex conjugate, which by
\[
w\overline{v}=\overline{wv}=\overline{\lambda v}=\lambda^{-1}\overline{v}
\]
 is the complex eigenspace corresponding to $\lambda^{-1}$. But by
(\ref{eq:complexified-eigenspace}) the sum of these complex eigenspaces
contains $V_{\lambda}^{w}$, which implies that the hyperplane came
from $\mathfrak{H}'$. 

$\Leftarrow$: The complexification of a hyperplane containing $V_{\lambda}^{w}$
contains the complexification $V_{\lambda}^{w}\otimes_{\mathbb{R}}\mathbb{C}$,
which by (\ref{eq:complexified-eigenspace}) contains the complex
eigenspace for the eigenvalue $\lambda$.
\end{proof}
\begin{example}
Reflections in type $\mathsf{B}_{2}$ are quasiregular but not regular.
\end{example}

The following slightly generalises \cite[Theorem 6.4(i)]{MR354894}:
\begin{prop}
Let $w=\delta\tilde{w}$ be an element of a twisted finite Coxeter
group and suppose that it is quasiregular, so there exists an eigenvalue
$\lambda=e^{2\pi i\theta}$ with $\theta\in(0,1/2]\cup\{1\}$ such
that $\mathfrak{H}_{V_{\lambda}^{w}}\subseteq\mathfrak{H}^{w}$. If
$\delta^{1/\theta}=\mathrm{id}$ then $\mathrm{ord}(w)=1/\theta$
and if furthermore $V_{\lambda}^{w}$ is in good position then
\[
\ell(w)=\frac{|\mathfrak{R}|-\ell_{f}(w)}{\mathrm{ord}(w)}.
\]
\end{prop}

\begin{proof}
Since $\lambda$ is an eigenvalue of $w$, we have $\mathrm{ord}(w)\geq1/\theta$.
As $w^{1/\theta}$ fixes $V_{\lambda}^{w}$ its twisted component
is trivial, Steinberg's theorem implies that it lies in $\tilde{W}_{V_{\lambda}^{w}}$,
but as before in the corresponding reflection representation its eigenspaces
are spanned by roots of $\mathfrak{R}^{w}$ and thus this element
is the identity. The final formula follows from e.g.\ Lemma \ref{lem:shifts}(i).
\end{proof}

The following statement also plays an important role in the sequel
\cite{WM-cross}:
\begin{prop}
Let $V'$ be a nontrivial subspace of the reflection representation
of an irreducible finite Coxeter group. Then the roots projecting
nontrivially to $V'$ generate the root lattice.
\end{prop}

\begin{proof}
Since the simple roots span the reflection representation $V$, there
must be some that project nontrivially to $V'$; let's denote this
subset by $I_{1}$. We now inductively add the remaining simple roots:
for $j\geq1$ we let $I_{j+1}$ be the union of $I_{j}$ and the simple
roots whose vertices in the Coxeter-Dynkin diagram of $W$ are connected
to those of simple roots in $I_{j}$. Since the diagram is connected,
there exists an integer $1\leq k\leq\mathrm{rk}(W)$ such that $I_{k}$
is the entire set of simple roots. We claim that for every simple
root $\alpha$ in $I_{j+1}\backslash I_{j}$, there exists a positive
root $\beta$ projecting nontrivially to $V'$ such that for some
scalar $c\in\mathbb{R}_{>0}$, the expression $\beta-c\alpha$ decomposes
as a positive sum of simple roots lying in $I_{j}$.

Since $\alpha$ does not lie in $I_{j}$, it is orthogonal to the
simple roots in $I_{j-1}$, and furthermore there exists a simple
root $\alpha'$ in $I_{j}$, such that $(\alpha,\alpha')<0$. The
induction hypothesis for $\alpha'$ furnishes a positive root 
\[
\beta'=c'\alpha'+\sum_{\alpha''\in I_{j-1}}c''\alpha'',\qquad c'\in\mathbb{R}_{>0},\quad c''\in\mathbb{R}_{\geq0},
\]
so from
\[
(\alpha,\beta')=(\alpha,c'\alpha')=c'(\alpha,\alpha')<0
\]
 we deduce that there exists a scalar $c\in\mathbb{R}_{>0}$ such
that $\beta:=\beta'+c\alpha$ is a positive root. As $\alpha$ projects
trivially to $V'$ but $\beta'$ does not, neither does $\beta$.

We now conclude that each simple root lies in the span of roots projecting
nontrivially to $V'$: given $\alpha$ in $I_{j+1}$, let $\beta$
and $\beta'$ be roots as we just constructed. Then $\beta'$ and
$\beta$ project nontrivially to $V'$, and $\alpha$ is a linear
combination of those two roots.
\end{proof}
\begin{cor}
Let $V_{0}$ and $V_{1}$ be two nontrivial subspaces of $V$. Then
\[
\{\textrm{roots projecting nontrivially to }V_{0}\}\cap\{\textrm{roots projecting nontrivially to }V_{1}\}\neq\varnothing.
\]
\end{cor}

\begin{proof}
Let $\beta_{1}$ be a root in the second set, then the proposition
yields a root $\beta_{0}$ in the first set such that $(\beta_{0},\beta_{1})\neq0$.
This implies that $\beta:=\beta_{0}+c\beta_{1}$ is a root for some
$c\in\mathbb{R}^{\times}$. If $\beta_{0}$ projects trivially to
$V_{1}$ and $\beta_{1}$ projects trivially to $V_{0}$, then $\beta$
projects nontrivially to both of them.
\end{proof}
\begin{prop}
\label{prop:sequence-not-minimal} Let $w$ be an element with a braiding
sequence of eigenspaces $\underline{\Theta}=(V_{m},\ldots,V_{1})$.
If the positive principal argument of the eigenvalue of $V_{1}$ is
not minimal amongst the positive principal arguments of the eigenvalues
of $w$, then $w$ does not have minimal length.
\end{prop}

\begin{proof}
For each root hyperplane not in $\mathfrak{H}_{\Theta}$, there is
an $i$ such that it lies in $\mathfrak{H}_{F_{i-1}}\backslash\mathfrak{H}_{F_{i}}$.
Denote the corresponding argument by $\theta_{\mathfrak{h}}:=\theta_{i}$,
and for $\mathfrak{h}\in\mathfrak{H}_{\Theta}$ set $\theta_{\mathfrak{h}}=0$.
Then the length formula (\ref{eq:length-formula-braiding}) rewrites
as
\begin{equation}
\ell(w)=2\sum_{i=1}^{m}\theta_{i}\bigl|\mathfrak{H}_{F_{i-1}}\backslash\mathfrak{H}_{F_{i}}\bigr|=2\sum_{\mathfrak{h}\in\mathfrak{H}}\theta_{\mathfrak{h}}.\label{eq:length-by-hyperplane}
\end{equation}

Now let $w'$ be a conjugate of $w$ such that $\underline{\Theta}'_{-}=(V_{m'}^{w'},\ldots,V_{1}^{w'})$
is in good position, where $\Theta'$ is the set of all eigenspaces
(or eigenvalues) of $w'$, and let $\theta_{\mathfrak{h}}'$ denote
the corresponding arguments for root hyperplanes $\mathfrak{h}$.
Then there is an inequality $\theta_{\mathfrak{h}}\geq\theta_{\mathfrak{h}}'$,
for all root hyperplanes: for the roots projecting nontrivially to
$V_{i}^{w'}$ but trivially to $F_{i-1}':=\sum_{j=1}^{i-1}V_{j}^{w'}$,
the argument is $\theta_{i}'$ which is minimal amongst all possible
arguments.

According to the previous corollary, there exists a root hyperplane
$\mathfrak{h}'$ in $\mathfrak{H}\backslash\mathfrak{H}_{V_{1}}=\mathfrak{H}_{F_{0}}\backslash\mathfrak{H}_{F_{1}}$
which also lies in $\mathfrak{H}_{F_{0}'}\backslash\mathfrak{H}_{F_{1}'}$.
As by assumption the argument of $V_{1}$ is not minimal, we have
a strict inequality $\theta_{\mathfrak{h}'}>\theta_{\mathfrak{h}'}'$.
From (\ref{eq:length-by-hyperplane}) we hence obtain 
\[
\ell(w)=2\sum_{\mathfrak{h}\in\mathfrak{H}}\theta_{\mathfrak{h}}>2\sum_{\mathfrak{h}\in\mathfrak{H}}\theta_{\mathfrak{h}}'=\ell(w').\qedhere
\]
\end{proof}
Much of the following statement was observed by Coxeter \cite{MR0027148}
and was proven uniformly by Steinberg \cite{MR106428}. At the time,
it was mainly used to relate the number of roots to the Coxeter number
and to the height of the highest root.
\begin{thm}
Fix an irreducible finite Coxeter group. 

\begin{enumerate}[\normalfont(i)]

\item Its Coxeter elements are elliptic.

\item We may uniquely decompose the set of vertices of its Coxeter-Dynkin
diagram (which is a tree) into two disconnected subsets. For each
of these subsets, the corresponding simple reflections then commute,
so their products yield two involutions $\iota_{1}$ and $\iota_{2}$,
and then their product $w:=\iota_{1}\iota_{2}$ is a Coxeter element
of minimal length.

\item There exists a plane in the reflection representation on which
$\iota_{1}$ and $\iota_{2}$ act as reflections, so in particular
this plane contains two lines on which they act as $-1$. Both of
these lines intersect nontrivially with the closure of the dominant
Weyl chamber.

\item Every root projects nontrivially to this plane, so this plane
yields a complete sequence of eigenspaces of length one for their
product $w$.

\item Hence the order of $w$ on the reflection representation agrees
with the order of $w$ on this plane, and the corresponding eigenvalue
thus has minimal argument amongst the eigenvalues of $w$. Moreover,
this eigenvalue occurs with multiplicity one.

\end{enumerate}
\end{thm}

\begin{defn}
For any Coxeter element, the plane corresponding to this eigenvalue
is called its \emph{Coxeter plane}. Its order is called the \emph{Coxeter
number}, and is typically denoted by $h$. The two Coxeter elements
just constructed are sometimes called \emph{bipartite}.
\end{defn}

\begin{prop}
\label{prop:Coxeter-braiding} The only Coxeter elements of minimal
length that have a braiding sequences of eigenspaces, are the two
whose Coxeter planes go through the dominant Weyl chamber.
\end{prop}

\begin{proof}
There are $h$ lines in the Coxeter plane on which the reflections
in the dihedral subgroup generated by $\iota_{1}$ and $\iota_{2}$
act as $-1$, yielding a decomposition of this plane into $2h$ fundamental
domains. Since the lines on which $\iota_{1}$ and $\iota_{2}$ act
as $-1$ both lie in the closure of the dominant Weyl chamber, and
the segment between them lies in its interior, there are there are
precisely $2h$ Weyl chambers containing an open subset of the Coxeter
plane. The corresponding $2h$ conjugates of $w$ in the Coxeter group
are the $2h$ conjugates of $w=\iota_{1}\iota_{2}$ inside this dihedral
subgroup, but these are all either $w$ or $w^{-1}=\iota_{2}\iota_{1}$.
The claim now follows from Proposition \ref{prop:sequence-not-minimal}.
\end{proof}
This implies that Sevostyanov's set of elements only includes two
Coxeter elements of minimal length; hence it does not include all
of the elements that are used in the statement of He-Lusztig. However,
He-Lusztig's proof uses a trick (which we've replaced with a simpler
braid monoid argument in part (ii) of the main Theorem) to reduce
the cross section statement for elliptic elements of minimal length
to Geck-Michel's ``good'' elements \cite[§3.1-§3.5]{MR2904572}.
As already mentioned, the case-free construction of such elements
in \cite[§5.2]{MR2999317} is through decreasing, complete sequences
of eigenspaces, which does yield a subset of Sevostyanov's elliptic
elements.

\subsection{Gradient flows}

The main novelty of this subsection is Lemma \ref{lem:normal-vector-cyclic},
which is pivotal in proving the new claims on cyclic shifts and strong
conjugations. In the end we deduce Proposition \ref{prop:shifts-braiding}(i).
\begin{prop}
\label{prop:normal-vector} Let $w$ be an element of a twisted finite
Coxeter group, let $v$ be an element in the wall of the dominant
Weyl chamber $C$ corresponding to the hyperplane $\mathfrak{h}_{i}$
of a simple reflection $s_{i}$, and let $n$ be a normal vector to
$\mathfrak{h}_{i}$ at one of its elements, pointing towards $C$
(e.g.,\ $n=\alpha_{i}$).

Then there are implications
\[
\begin{array}{ccc}
\pm\bigl(n,w(v)\bigr)>0 & \qquad\Longrightarrow\qquad & \ell(ws_{i})\gtrless\ell(w),\\
\pm\bigl(w(n),v\bigr)>0 & \qquad\Longrightarrow\qquad & \ell(s_{i}w)\gtrless\ell(w).
\end{array}
\]
\end{prop}

\begin{proof}
The claims are equivalent to the implications
\[
\begin{array}{ccc}
\ell(ws_{i})\gtrless\ell(w) & \qquad\Longrightarrow\qquad & \pm\bigl(n,w(v)\bigr)\geq0,\\
\ell(s_{i}w)\gtrless\ell(w) & \qquad\Longrightarrow\qquad & \pm\bigl(w(n),v\bigr)\geq0.
\end{array}
\]
We shall prove the first one; the second one then follows by taking
inverses. The length of $ws_{i}$ is larger/smaller than the length
of $w$ if and only if the hyperplane $\mathfrak{h}_{i}$ lies between
$w(C)$ and $\mp C$, if and only if $w(C)$ and $\pm C$ lie on the
same side w.r.t.\ $\mathfrak{h}_{i}$, if and only if $\pm\bigl(n,w(v)\bigr)\geq0.$
\end{proof}
\begin{cor}
Hence if the inequality
\[
\bigl(n,w(v)\bigr)+\bigl(w(n),v\bigr)\gtrless0
\]
 holds, then at least one of
\[
\ell(ws_{i})\gtrless\ell(w)\qquad\textrm{or}\qquad\ell(s_{i}w)\gtrless\ell(w)
\]
holds, and if both of those hold then the converse is true.
\end{cor}

\begin{cor}
[{\cite[Lemma 1.1]{MR2999317}}] \label{cor:gradient-function} Moreover,
consider the function
\[
f_{w}:V\longrightarrow\mathbb{R},\qquad v\longmapsto\bigl|\bigl|(\mathrm{id}_{V}-w)(v)\bigr|\bigr|^{2}
\]
and set $w'=s_{i}ws_{i}$.

\begin{enumerate}[\normalfont(i)]

\item The gradient $\nabla f_{w}$ at $v$ satisfies
\[
\bigl(\nabla f_{w}(v),n\bigr)=2\bigl(n-w(n),v-w(v)\bigr)=-2\bigl(n,w(v)\bigr)-2\bigl(w(n),v\bigr),
\]

\item and yields an implication
\[
\bigl(\nabla f_{w}(v),n\bigr)\lessgtr0\qquad\Longrightarrow\qquad\pm\ell(w')\geq\pm\ell(w).
\]

\end{enumerate}
\end{cor}

\begin{proof}
(i): The first equality follows from expanding
\[
\bigl(\nabla f_{w}(v),n\bigr)=\lim_{t\rightarrow0}\frac{f_{w}(v+tn)-f_{w}(v)}{t},
\]
 and the second equality from using $\bigl(w(n),w(v)\bigr)=(n,v)=0$.

(ii): This is equivalent to the implication
\[
\ell(w')=\ell(w)\pm2\qquad\Longrightarrow\qquad\bigl(\nabla f_{w}(v),n\bigr)\lessgtr0,
\]
which follows from (i) and the previous corollary.
\end{proof}
One of the earliest observations that led to this paper was noticing
that the corresponding gradient flow can be effectively restricted
to sums of eigenspaces:
\begin{notation}
Let $\Theta$ be a set of eigenspaces. We denote by $V_{\Theta}^{\pm}$
the eigenspace in $\Theta$ corresponding to the eigenvalue with maximal/minimal
positive principal argument, as in Definition \ref{def:sequence-eigenspaces}.
Given an element $v$ in $V_{\Theta}$, we denote its projection to
$V_{\Theta}^{\pm}$ by $v_{\pm}$.
\end{notation}

\begin{prop}
[{\cite[§1.5-§1.6]{MR2999317}}] \label{prop:gradient-flow} Let the
eigenvalues of $w$ be $\{e^{\pm2\pi i\theta_{m}},\ldots,e^{\pm2\pi i\theta_{1}}\}$
with each $0\leq\theta_{j}\leq1/2$, and denote the projection of
an element $v$ of $V$ to the real eigenspace corresponding to $e^{2\pi i\theta_{j}}$
by $v_{j}$. Then this function $f_{w}$ has a global gradient flow
\[
\Phi_{w}:V\times\mathbb{R}\longrightarrow V,\qquad(v,t)\longmapsto\exp\bigl(2t(2-w-w^{-1})\bigr)v=\sum_{j=1}^{m}\exp\bigl(4t(1-\cos2\pi\theta_{j})\bigr)v_{j},
\]
and hence it satisfies
\[
\underset{t\rightarrow\pm\infty}{\lim}\frac{\Phi_{w}(v,t)}{||\Phi_{w}(v,t)||}=\frac{v_{\pm}}{||v_{\pm}||}.
\]

In other words, if we flow the half-line through $v$ in positive/negative
direction, then it converges to the half-line through $v_{\pm}$.
Thus if $\Theta$ is a set of eigenspaces, then this flow yields a
fiber bundle from $V_{\Theta}\backslash(V_{\Theta}^{\pm})^{\bot}$
to the unit sphere in $V_{\Theta}^{\pm}$.
\end{prop}

\begin{prop}
\label{prop:generic-flow} For a generic point in $V_{\Theta}$, the
corresponding flow generically consists of regular points of $V_{\Theta}$,
and each of the nonregular points are regular in a hyperplane of $V_{\Theta}$.
\end{prop}

\begin{proof}
This follows from dimension considerations, as in the proof of \cite[Proposition 1.2]{MR2999317}.
\end{proof}
\begin{notation}
Let $w$ be an element of a twisted finite Coxeter group $W$ and
let $C$ denote the dominant Weyl chamber. Given another Weyl chamber
$C'$ there is a unique element $\tau$ such that $\tau(C)=C'$; we
write $w_{C'}:=\tau w\tau^{-1}$.
\end{notation}

\begin{lem}
\label{lem:normal-vector-cyclic} Let $w$ be an element of a twisted
finite Coxeter group $W$ and $\Theta$ be a subset of its eigenspaces
such that the closure of the dominant Weyl chamber $C$ contains regular
points of $V_{\Theta}$. Pick a generic regular point as in the previous
statement, flow it in negative/positive direction until it meets a
root hyperplane inside $V_{\Theta}$. Call the corresponding nonregular
point $\hat{v}$, and flow slightly further to a regular point $v'$.

Then there exists a Weyl chamber $C'$ whose closure contains $\hat{v}$
and $v'$, satisfying
\[
w\overset{\pm}{\leadsto}w_{C'}.
\]
\end{lem}

\begin{proof}
The assumption on nonregular points can be rephrased as saying that
the hyperplanes containing $\hat{v}$ but not $V_{\Theta}$ all intersect
$V_{\Theta}$ in the same hyperplane of $V_{\Theta}$. The claims
on the flow follow from dimension considerations; let $v$ be such
a regular point satisfying the requirements. We now slightly perturb
$v$ to a point $\tilde{v}$ in the interior of $C$, so that its
flow goes to a point $\tilde{v}'$ nearby $v'$; generically, it will
only meet the hyperplanes containing $\hat{v}$ in regular points,
and those intersection points can be made to lie arbitrarily close
to $\hat{v}$. Thus we obtain a path of Weyl chambers 
\[
C,\qquad s_{i_{1}}(C),\qquad\ldots,\qquad s_{i_{k}}\cdots s_{i_{1}}(C),
\]
 where the final Weyl chamber contains $\tilde{v}'$ and its closure
thus contains $v'$. This is a ``minimal path'' of Weyl chambers:
if we sufficiently ``zoom in'' on a gradient flow then it is approximately
linear, so this path is derived from an approximately straight line
between these two chambers.

Corollary \ref{cor:gradient-function}(i) yields that
\[
\bigl(w(n),\hat{v}\bigr)+\bigl(n,w(\hat{v})\bigr)\gtrless0
\]
for each normal vector $n$ of the hyperplanes pointing towards $C$.
Thus at least one of 
\[
\bigl(w(n),\hat{v}-w(\hat{v})\bigr)=\bigl(w(n),\hat{v}\bigr)\gtrless0\qquad\textrm{or}\qquad\bigl(n,w(\hat{v})-\hat{v}\bigr)=\bigl(n,w(\hat{v})\bigr)\gtrless0
\]
 holds. In fact, one of these inequalities holds for each $n$: since
$\hat{v}$ lies in $V_{\Theta}$ and this subspace is preserved by
$\mathrm{id}_{V}-w$, the element $\hat{v}-w(\hat{v})$ still lies
in $V_{\Theta}$. By the assumption on hyperplanes containing nonregular
points, $\hat{v}-w(\hat{v})$ lies on one side of such a hyperplane
if and only if it lies on the same side of all of them. And then the
same inequality is true with $\hat{v}$ replaced by one of the new
intersection points $\hat{v}'$ since they lie arbitrarily close. 

As $\hat{v}'$ lies arbitrarily close to $\hat{v}$ and $\hat{v}$
lies in these hyperplanes, we have
\[
s_{i_{1}}\cdots s_{i_{j}}(\hat{v}')\approx s_{i_{1}}\cdots s_{i_{j}}(\hat{v})=\hat{v}.
\]
The claim now follows inductively by applying Proposition \ref{prop:normal-vector}
consecutively: if say the first equation holds, then for each $w_{j}:=s_{i_{j}}\cdots s_{i_{1}}w$
we have
\[
\bigl(w_{j}(n),\hat{v}'\bigr)=\bigl(w(n),s_{i_{1}}\cdots s_{i_{j}}(\hat{v}')\bigr)\approx\bigl(w(n),\hat{v}\bigr)\gtrless0,
\]
and hence
\[
\ell(s_{i_{j+1}}w_{j})=\ell(w_{j})\pm1=\ell(w)\pm(j+1).\qedhere
\]
\end{proof}

\begin{proof}
[Proof of Proposition \ref{prop:shifts-braiding}(i)] We slightly
generalise the reasoning of \cite[Proposition 1.2]{MR2999317}, inducting
on the length of $\underline{\Theta}_{\pm}$ (or the rank of $W$).
The closure $\overline{C}$ of $C$ contains an open subset of $V_{\Theta}$,
and for a generic point $v'$ in this subset its projection $v_{\pm}'$
to $V_{\Theta}^{\pm}$ is nonzero and regular. The limit of a generic
point $v$ inside of the chamber $C$ under the positive/negative
flow might not lie in $V_{\Theta}^{\pm}$ (when the corresponding
eigenvalue is not maximal among the eigenvalues of $w$), but by picking
$v$ sufficiently close to $v'$ its half-line will pass by the half-line
through $v_{\pm}$ with arbitrarily little distance between them,
and thus the flow will pass through a chamber $C'$ whose closure
contains an open neighbourhood of $v_{\pm}'$ in $V_{\Theta}^{\pm}$.
Furthermore we can ensure that $v$ satisfies the statements in Proposition
\ref{prop:generic-flow}; when it meets the first hyperplane there
is then a strict inequality
\[
\bigl(\nabla f_{w}(v),n\bigr)\lessgtr0
\]
where $n$ is the normal vector to $v$, pointing towards the chamber
the flow came from. Since the flow converges there is, for each chamber
it goes through, a final time it is there before getting within the
required distance of $v_{\pm}$, so if we only use those wall crossings
then it is evident that the corresponding sequence of Weyl chambers
from $C$ to $C'$ is finite. Corollary \ref{cor:gradient-function}(ii)
now implies that $w\overset{\pm}{\rightarrow}w_{C'}$ in $W$, whilst
Lemma \ref{lem:normal-vector-cyclic} furnishes that $w\overset{\pm}{\leadsto}w_{C'}$
in $\CMcal O^{\Theta}$.

Thus by replacing $C$ with $C'$, we may now assume that the closure
of $C$ contains a regular point $v_{\pm}'$ of $V_{\Theta}^{\pm}$,
which is the first eigenspace of $\underline{\Theta}_{\pm}$. By Proposition
\ref{lem:factorise} there is a factorisation $w=xy$ where $y$ lies
in the proper standard parabolic subgroup $\tilde{W}_{v_{\pm}'}$
of $\tilde{W}$ and $x$ permutes the indices of $\tilde{W}_{v_{\pm}'}$.
Applying the induction hypothesis to the element $\delta_{x}y$ in
$\<\delta_{x}\>\ltimes\tilde{W}_{v_{\pm}'}$ then implies that we
can further shift the chamber so that it is in good position with
respect to $\underline{\Theta}_{\pm}$, yielding the first two claims.
By starting at this chamber and following the flow in reverse direction,
we get the final claim.
\end{proof}

\subsection{\label{subsec:shifts}Shifts}

The main aim of this subsection is to generalise the proof of \cite[§3.3]{MR2999317}
to yield Lemma \ref{lem:shifts}(ii), and subsequently deduce the
remainder of Proposition \ref{prop:shifts-braiding}. For the new
claims on cyclic shifts and strong conjugations, we will need the
following lemma:
\begin{defn}
[{\cite[§2.4]{MR2999317}}] Let $V'$ be a subspace of the reflection
representation $V$. We will say that two Weyl chambers $C,C'$ are
\emph{$V'$-adjacent} if $\overline{C}\cap\overline{C}'\cap V'$ spans
a codimension 1 subset of $V'$.
\end{defn}

\begin{lem}
\label{lem:eigenspace-shifts} Let $w$ be an element of a twisted
finite Coxeter group. Let $V'$ be a real eigenspace in the reflection
representation corresponding to an eigenvalue of $w$. Let $C$ and
$C'$ be two Weyl chambers whose closures contain open subsets of
$V'$, which lie in the same connected component of $V\backslash\mathfrak{H}_{V'}$
and which are $V'$-adjacent. 

If the eigenvalue is not $1$ or $-1$ then either 
\[
w\overset{-}{\leadsto}w_{C'}\quad\textrm{and}\quad w_{C'}\overset{+}{\leadsto}w\qquad\textrm{or}\qquad w\overset{+}{\leadsto}w_{C'}\quad\textrm{and}\quad w_{C'}\overset{-}{\leadsto}w
\]
 holds. If the eigenvalue is $\pm1$ then 
\[
w\overset{\pm}{\leadsto}w_{C'}\quad\textrm{and}\quad w_{C'}\overset{\pm}{\leadsto}w
\]
holds.
\end{lem}

\begin{proof}
Let $\tau\in\tilde{W}$ be the unique element such that $\tau(C)=C'$.
Pick a point in the interior of $\overline{C}\cap V'$. Assume that
the eigenvalue is not $\pm1$; since this point lies in $V'$, its
action is described by a two-dimensional plane where $\overline{C}\cap\overline{C}'\cap V'$
is a hyperplane and where $\overline{C}\cap V'$ and $\overline{C}'\cap V'$
are chambers touching this hyperplane. Either $w$ rotates in the
direction of the hyperplane (with angle strictly between $0$ and
$\pi$) or it rotates away from it.

Let's assume the former. Since we may choose the point arbitrarily
close to the hyperplane and the eigenvalue is not 1, it follows that
$w(C)$ crosses this hyperplane, and that $w^{-1}(C)$ does not as
the eigenvalue is not $-1$. The chamber $w(C)$ might not be equal
to $C'$, but its projection to $C'$ is the same or lies further
along this rotation, so $\mathfrak{H}\bigl(C,C'\bigr)\backslash\mathfrak{H}_{V'}\subseteq\mathfrak{H}\bigl(C,w(C)\bigr)\backslash\mathfrak{H}_{V'}$.
But by assumption on $C'$ we have $\mathfrak{H}\bigl(C,C'\bigr)\cap\mathfrak{H}_{V'}=\varnothing$,
so $\mathfrak{H}\bigl(C,C'\bigr)\subseteq\mathfrak{H}\bigl(C,w(C)\bigr)$.
Hence
\[
\mathfrak{H}\bigl(C,\tau(C)\bigr)\sqcup\mathfrak{H}\bigl(C',w\tau^{-1}(C')\bigr)=\mathfrak{H}(C,C')\sqcup\mathfrak{H}\bigl(C',w(C)\bigr)=\mathfrak{H}\bigl(C,w(C')\bigr),
\]
which is equivalent to $\ell(\tau)+\ell(w\tau^{-1})=\ell(w)$, and
hence $\ell(w\tau^{-1})=\ell(w)-\ell(\tau)$. Similarly,
\[
\mathfrak{H}\bigl(C,\tau(C)\bigr)\sqcup\mathfrak{H}\bigl(C',w_{C'}(C')\bigr)=\mathfrak{H}\bigl(C,C'\bigr)\sqcup\mathfrak{H}\bigl(C',w_{C'}(C')\bigr)=\mathfrak{H}\bigl(C,w_{C'}\tau(C)\bigr)
\]
 yields $\ell(\tau)+\ell(w_{C'})=\ell(w_{C'}\tau)$. This shows $w\overset{-}{\leadsto}w_{C'}$
and $w_{C'}\overset{+}{\leadsto}w$, the other case is proven analogously,
and so is the case where the eigenvalue is $\pm1$.
\end{proof}
\begin{prop}
\label{prop:chamber-sequence} Let $\underline{\Theta}=(\ldots,V_{1})$
be a sequence of eigenspaces of an element of a twisted finite Coxeter
group, such that the dominant Weyl chamber $C$ is in good position.
If $C'$ is a Weyl chamber lying in the same connected component as
$C$ in $V\backslash\mathfrak{H}_{V_{1}}$, then $C'$ is in good
position w.r.t.\ $\underline{\Theta}$ if and only if its closure
contains a regular point of $V_{1}$.

In particular, for such $C'$ there exists a sequence of Weyl chambers
\[
C=C_{0},\ldots,C_{m}=C'
\]
 such that each Weyl chamber in this sequence is in good position
with respect to $\underline{\Theta}$, and each pair $C_{i},C_{i+1}$
is $V_{1}$-adjacent.
\end{prop}

\begin{proof}
Since $C$ and $C'$ lie in the same connected component of $V\backslash\mathfrak{H}_{V_{1}}$,
they definitely lie in the same connected component of $V\backslash\mathfrak{H}_{F_{i}}$
for $i\geq1$. The first claim then follows from Proposition \ref{prop:good-position}.

The second claim follows from the first, using \cite[Lemma 2.4]{MR2999317}.
\end{proof}
\begin{prop}
\label{prop:wall-intersection} Let $C,C'$ be two Weyl chambers.
The unique (untwisted) Coxeter group element $\tau$ sending $C$
to $C'$ fixes all points in $\overline{C}\cap\overline{C}'$, i.e.\
\[
\overline{C}\cap\overline{\tau(C)}=\overline{C}\cap V^{\tau}.
\]
\end{prop}

\begin{proof}
We may assume that $C$ is the dominant Weyl chamber, and that $C'$
is not. The intersection $\overline{C}\cap\overline{C}'$ is then
precisely the intersection of the hyperplanes separating $C$ from
$C'$, and at least one of them $\mathfrak{h}_{i}$ corresponds to
a simple reflection $s_{i}$. Thus we can decrease the distance from
$C$ to $C'$ by crossing one such plane to $C'':=s_{i}(C)$; as the
intersection $\overline{C}\cap\overline{C}'$ lies in the wall $\mathfrak{h}_{i}$
it is fixed by the reflection $s_{i}$, and thus the new intersection
$\overline{C}''\cap\overline{C}'$ contains $\overline{C}\cap\overline{C}'$.
By induction we thus obtain a path from $C$ to $C'$ through simple
reflections, each of which fixes $\overline{C}\cap\overline{C}'$,
and hence so does their product.
\end{proof}
We now finish the proof of Lemma \ref{lem:shifts}, following \cite[§3.3]{MR2999317}
\emph{very} closely:
\begin{proof}
(ii): Let $\tau\in\mathrm{Tran}(w,w')$ and consider $C':=\tau(C)$,
so that $w'=w_{C'}$. We induct on $\dim V_{\Theta}$; the base case
is that $\Theta$ is an empty sequence, but then $\mathfrak{H}=\mathfrak{H}_{\Theta}\subseteq\mathfrak{H}^{w}$
implies that $w$ is the identity element.

By Lemma \ref{lem:factorise}(i) and (iii), we may factorise $w=xy$
and (the connected component in $V\backslash\mathfrak{H}_{V_{1}}$
or the orthogonal projection to $V_{\tilde{W}_{V_{1}}}$ of) $C$
is in good position for the sequence $\underline{\Theta}'$ for $\delta_{x}y$
in $\<\delta_{x}\>\ltimes\tilde{W}_{V_{1}}$, and similarly for $C'$
and any other Weyl chamber in its connected component for $V\backslash\mathfrak{H}_{V_{1}}$.
Let $\tau$ be an element in $\tilde{W}_{V_{1}}$ such that $\tau(C)$
lies in the same component as $C'$ in $V\backslash\mathfrak{H}_{V_{1}}$.
From the induction hypothesis for the sequence $\underline{\Theta}'$
for the element $\delta_{x}y$ in $\<\delta_{x}\>\ltimes\tilde{W}_{V_{1}}$,
we then obtain a preimage for $\tau$ in 
\[
\mathrm{Tran}^{\underline{\Theta},\star}(\delta_{x}y,\tau\delta_{x}y\tau^{-1})\longrightarrow\mathrm{Tran}(\delta_{x}y,\tau\delta_{x}y\tau^{-1}).
\]
Since this preimage is a sequence of elements in $\tilde{W}_{V_{1}}$
(thus fixing $V_{1}$) and each of them commute with $\delta_{x}$
in the same way as with $x$, it also yields a preimage for $\tau$
in
\[
\mathrm{Tran}^{\underline{\Theta},\star}(w,\tau w\tau^{-1})\longrightarrow\mathrm{Tran}(w,\tau w\tau^{-1}),
\]
and similarly for $\mathrm{Tran}^{\star}(w,\tau w\tau^{-1})$.

Hence we may assume that $C$ and $C'$ lie in the same connected
component of $V\backslash\mathfrak{H}_{V_{1}}$. By Proposition \ref{prop:chamber-sequence}
there exists a sequence of Weyl chambers $C=C_{0},C_{1},\ldots,C_{n}=C'$
in the same component of $V\backslash\mathfrak{H}_{V_{1}}$, such
that $w_{C_{i}}$ lies in $\CMcal O{}^{\underline{\Theta}}$ and that
each pair $C_{i},C_{i+1}$ is $V_{1}$-adjacent. By induction on the
length of this sequence, we may assume that $n=1$. Consider the hyperplane
$\mathfrak{h}':=\mathfrak{h}\cap V_{1}$ in $V_{1}$, where $\mathfrak{h}\in\mathfrak{H}\backslash\mathfrak{H}_{V_{1}}$
is a root hyperplane separating $C$ and $C'$. Since $\overline{C}\cap\overline{C}'$
spans $\mathfrak{h}'$, this intersection $\overline{C}\cap\overline{C}'$
contains a regular point $v$ of $\mathfrak{h}'$.

Case 1: $w(\mathfrak{h}')\neq\mathfrak{h}'$. From Proposition \ref{prop:wall-intersection}
it follows that the ``minimal path'' from $C$ to $C'$ yields a
sequence of Weyl chambers $C=\tilde{C}_{0},\ldots,\tilde{C}_{m}=C'$
in the same component of $V\backslash\mathfrak{H}_{V_{1}}$, where
each consecutive pair of Weyl chambers is adjacent and $V_{1}$-adjacent.
By \cite[Proposition 2.5]{MR2999317} the lengths of the corresponding
elements is the same, so that the simple reflections corresponding
to this sequence of Weyl chambers yields a sequence in $\mathrm{Tran}^{\leftrightarrow}(w,w')$
mapping to $\tau$. A suitable element in $\mathrm{Tran}^{\underline{\Theta},\star}(w,w')$
is obtained from Lemma \ref{lem:eigenspace-shifts}, using Lemma \ref{lem:shifts}(i).

Case 2: $w(\mathfrak{h}')=\mathfrak{h}'$ and $\dim V_{1}\geq2$,
so $\dim\mathfrak{h}'\geq1$. We use Lemma \ref{lem:factorise} to
split $w=xy=\delta_{x}y$ for $\<\delta_{x}\>\ltimes\tilde{W}_{\mathfrak{h}'}$,
then since $\mathfrak{h}'\neq0$ the sequence $\underline{\Theta}'$
is ``smaller'' again. By Proposition \ref{prop:wall-intersection}
the translation $\tau\in\tilde{W}$ fixes $\mathfrak{h}'$, so it
lies in $\tilde{W}_{\mathfrak{h}'}$. As before, we use the induction
hypothesis to obtain an element in $\mathrm{Tran}^{\star}(w,w')$
or $\mathrm{Tran}^{\underline{\Theta},\star}(w,w')$ mapping to $\tau$.

Case 3: $w(\mathfrak{h}')=\mathfrak{h}'$ and $\dim V_{1}=1$. Thus
the eigenvalue $\lambda_{1}$ corresponding to $V_{1}$ is either
$1$ or $-1$, and the claim follows from Lemma \ref{lem:eigenspace-shifts}
again.
\end{proof}
Unless $w$ is an involution, we cannot (as is well-known) always
find a sequence of Weyl chambers such that the length of the corresponding
element strictly increases or decreases:
\begin{example}
Consider $w=s_{2}s_{3}s_{4}s_{1}s_{2}s_{3}$ in type $\mathsf{A}_{4}$,
which is not of minimal length in its conjugacy class and has eigenvalues
$e^{\pm2\pi i/5}$ and $e^{\pm4\pi i/5}$. Let $\theta$ be all of
its eigenvalues, so $V_{\Theta}=V$ and $V_{\Theta}^{-}=V_{2/5}^{w}$.
The latter is spanned by the vectors
\[
\varphi\omega_{1}-\varphi\omega_{2}+\omega_{3}\qquad\textrm{and}\qquad\omega_{1}-\omega_{3}+\varphi\omega_{4},
\]
where $\varphi:=(1+\sqrt{5})/2$ denotes the golden ratio. Its intersection
with the closure of the dominant Weyl chamber is the zero vector,
but one can calculate that
\[
\ell(s_{i}ws_{i})\geq\ell(w)
\]
for any simple reflection $s_{i}$.
\end{example}

Nor does there always exist a sequence of Weyl chambers containing
a regular point of $V_{\Theta}$:
\begin{example}
Consider $w=w_{\circ}s_{1}s_{2}s_{4}$ in type $\mathsf{B}_{4}$,
which has eigenvalues $\pm1$ and $e^{\pm2\pi i/3}$. Let $\theta$
be all of them except 1, so $V_{\Theta}=V_{w}$. Then $V_{\Theta}$
is the hyperplane of $s_{4}$, so $s_{1}ws_{1}$ yields another element
whose corresponding chamber $s_{1}(C)$ touches $V_{\Theta}$, and
this is the only other element reachable from $w$ by such simple
shifts. As the eigenspace for $-1$ is the line through the fundamental
weight $\omega_{3}$, neither the dominant Weyl chamber $C$ nor $s_{1}(C)$
contains a regular point of $V_{\Theta}^{-}$, which is the span of
$\alpha_{1}$ and $\alpha_{2}$.
\end{example}

\begin{proof}
[Proof of Proposition \ref{prop:shifts-braiding}]

(ii): Let $w$ be an element of $\CMcal O{}_{\mathrm{max}/\mathrm{min}}^{\Theta}$.
From part (i) we obtain an element $w'$ in $\CMcal O_{\pm}^{\Theta}\subseteq\CMcal O^{\Theta}$
such that $w\overset{\pm}{\rightarrow}w'$, so as $w$ has maximal/minimal
length it follows that $w'$ also has maximal/minimal length. It then
follows from Lemma \ref{lem:shifts}(i) that
\[
\ell(\CMcal O{}_{\mathrm{max}/\mathrm{min}}^{\Theta})=\ell(w)=\ell(w')=\ell(\CMcal O_{\pm}^{\Theta})=2\sum_{i=1}^{m}\theta_{i}\bigl|\mathfrak{H}_{F_{i-1}}\backslash\mathfrak{H}_{F_{i}}\bigr|
\]
 as claimed.

(iii): By using the second or third statement in (i) of the proposition
and maximality/minimality, we may use either strong conjugations or
cyclic shifts to move both $w$ and $w'$ to $\CMcal O_{\pm}^{\Theta}$.
The claim then follows from \ref{lem:shifts}(ii).
\end{proof}

\section{Minimally dominant elements}

In this section we study standard parabolic subgroups of twisted finite
Coxeter groups, and prove the main Lemma and Proposition \ref{prop:involutions}.
We will need
\begin{lem}
\label{lem:regular-vector-dominant} Let $w$ be an element of a twisted
finite Coxeter group and let $v$ be a vector in the intersection
of $V_{w}$ with the closure of the dominant Weyl chamber. Then $v$
is a regular point for $V_{w}$ if and only if
\begin{equation}
\textrm{for all roots }\beta:\quad(\beta,v)\begin{cases}
>0 & \textrm{if }\beta\textrm{ is a positive root and is not fixed by }w,\\
=0 & \textrm{if }\beta\textrm{ is fixed by }w,\\
<0 & \textrm{if }\beta\textrm{ is a negative root and is not fixed by }w.
\end{cases}\label{eq:regular-vector-dominant}
\end{equation}
\end{lem}

\begin{proof}
$\Rightarrow:$ As $v$ lies in the closure of the dominant Weyl chamber,
equation (\ref{eq:regular-vector-dominant}) follows when it is shown
that $(\beta,v)=0$ if and only if $\beta$ is fixed by $w$. But
due to the orthogonal decomposition $V=V_{w}\oplus V^{w}$, a regular
point $v$ satisfies 
\[
(\beta,v)=0\quad\textrm{if and only if}\quad(\beta,V_{w})=0\quad\textrm{if and only if}\quad\beta\in V^{w}\quad\textrm{if and only if}\quad\beta\textrm{ is fixed by }w.
\]

$\Leftarrow:$ If $v$ is not regular, then there exists a root $\beta$
not in $(V_{w})^{\bot}=V^{w}$ such that $(\beta,v)=0$.
\end{proof}

\subsection{Isotropy subgroups}

We begin this subsection with a description of isotropy subgroups,
then analyse intersections of the form $V^{w}\cap\overline{C}$ and
finish by proving the main Lemma.
\begin{prop}
\label{prop:parabolic} Let $W=\Omega\ltimes\tilde{W}$ be a twisted
finite Coxeter group, let $V$ be its reflection representation with
dominant Weyl chamber $C$ and let $V'$ be a subset of $V$.

\begin{enumerate}[\normalfont(i)]

\item The isotropy subgroup $W_{V'}$ is a parabolic subgroup of
$W$. More precisely, it is conjugate to a standard parabolic subgroup
of the form $\Omega'\ltimes\tilde{W}'$, where $\tilde{W}'$ is a
standard parabolic subgroup of $\tilde{W}$ and $\Omega'$ is the
subgroup of $\Omega$ preserving a certain partition of the simple
reflections not in $\tilde{W}'$.

\item If the closure $\overline{C}$ contains a regular point of
the span of $V'$, then $W_{V'}$ is a standard parabolic subgroup.

\item If it does not, then there exists a simple reflection fixing
$V'\cap\overline{C}$ but not $V'$. 

\end{enumerate}
\end{prop}

\begin{proof}
(i), (ii): We replace $V'$ by its span to make it convex, and subsequently
pick a regular point $v$ inside of it. After conjugating $W_{V'}$,
we may assume that this point lies in the closure of the dominant
Weyl chamber. Now let $\delta\tilde{w}$ be an element in $W_{V'}\subseteq W=\Omega\ltimes\tilde{W}$
fixing $v$. As $\delta$ preserves $\overline{C}$ and $\overline{C}$
is a fundamental domain for the action of $\tilde{W}$, any point
in $\overline{C}$ fixed by $w$ must be fixed by both $\delta$ and
$\tilde{w}$. Thus $\tilde{w}$ lies in the subgroup of $\tilde{W}$
fixing $v$; it is well-known that this is a standard parabolic subgroup
whose indices correspond to the zero coordinates of $v$ in the basis
of fundamental weights. Since $v$ is fixed by $\delta$, that set
of coordinates must be preserved by $\delta$. Since it also acts
as a permutation on the remaining coordinates, the claim follows.

(iii): Again replace $V'$ by its span, and note that $\tilde{W}_{V'}^{\,}\subseteq\tilde{W}_{V'\cap\overline{C}}$.
If the span of $V'\cap\overline{C}$ does not contain a regular point
of $V'$ then it there must be some reflection in $\tilde{W}_{V'}^{\,}$
that is not in $\tilde{W}_{V'\cap\overline{C}}$, so this inclusion
is proper. Since by (ii) the subgroup $\tilde{W}_{V'\cap\overline{C}}$
is a standard parabolic, it is generated by the simple reflections
it contains, so the claim follows.
\end{proof}
\begin{cor}
\label{cor:dom} Let $w$ be an element of a twisted finite Coxeter
group.

\begin{enumerate}[\normalfont(i)]

\item If $w$ has an anisotropic braiding sequence of eigenspaces,
then it is dominant.

\item If $w$ is dominant, then it is firmly convex.

\end{enumerate}
\end{cor}

\begin{proof}
(i): Let $\underline{\Theta}$ be this sequence, so by assumption
the closure of the dominant Weyl chamber contains an open subset $U$
of $V_{\Theta}$. The subgroup fixing it is a standard parabolic,
or in other words, the intersection of the hyperplanes in $\mathfrak{H}_{\Theta}$,
is some facet of the dominant Weyl chamber. Then we can find an open
neighbourhood $U'$ of $U$ inside of this facet, and by shrinking
it we may assume that it lies in the closure of the dominant Weyl
chamber. From $\mathfrak{H}_{V_{w}}\subseteq\mathfrak{H}_{\Theta}\subseteq\mathfrak{H}^{w}=\mathfrak{H}_{V_{w}}$
it follows that $V_{w}$ is also contained in this facet, so as $V_{\Theta}\subseteq V_{w}$
it follows that $V_{w}\cap U'$ is an open subset of $V_{w}$.

(ii): The subspace $V_{w}$ is spanned by its intersection with the
closure of the dominant Weyl chamber $C$. By say the previous lemma
the reflections whose hyperplanes contain $V_{w}$ are the reflections
of a standard parabolic subgroup, so equivalently the roots that are
perpendicular to $V_{w}$ (which as in the proof of Lemma \ref{lem:regular-vector-dominant}
are precisely the set of roots $\mathfrak{R}^{w}$ fixed by $w$)
form a standard parabolic subsystem.

Now pick a regular vector $v$ in $V_{w}\cap\overline{C}$ and let
$\beta$ be a positive root in $\mathfrak{R}_{\mathrm{st}}^{w}$ which
is not fixed by $w$. Then $w^{i}(\beta)$ is also a positive root
not fixed by $w$ for any integer $i$, so $(w^{i}(\beta),v)>0$ by
Lemma \ref{lem:regular-vector-dominant}. By summing these equations
for $i$ in $\{1,\ldots,\mathrm{ord}(w)\}$ it follows that
\[
\bigl(\sum_{i=1}^{\mathrm{ord}(w)}w^{i}(\beta),v\bigr)=\sum_{i=1}^{\mathrm{ord}(w)}\bigl(w^{i}(\beta),v\bigr)>0,
\]
but as the vector $\sum_{i=1}^{\mathrm{ord}(w)}w^{i}(\beta)$ is fixed
by $w$ it should be orthogonal to $V_{w}$ and hence to $v$.
\end{proof}
\begin{notation}
For any element $w$ of a twisted finite Coxeter group and subset
$V'$ of its reflection representation, we write $(V')^{w}:=V^{w}\cap V'$
for the fixed points of $w$ that lie in $V'$. In particular, $\overline{C}^{w}:=V^{w}\cap\overline{C}$
denotes the intersection of $V^{w}$ with the closure of the dominant
Weyl chamber $C$.
\end{notation}

\begin{cor}
\label{cor:parabolic} Let $W$ be a twisted finite Coxeter group
and pick an element $w=\delta\tilde{w}$ in $W$. Consider the subgroup
$W_{w}$ of $W$ consisting of elements fixing $\overline{C}^{w}$.

\begin{enumerate}[\normalfont(i)]

\item This is a standard parabolic subgroup containing $\delta$
and $\tilde{w}$.

\item Moreover, $w$ is an elliptic element of this subgroup if and
only if $\overline{C}^{w}$ contains a regular point of $V^{w}$.

\end{enumerate}
\end{cor}

\begin{proof}
(i): Immediately follows from part (ii) of the previous proposition.

(ii): By definition (and convexity) the cone $\overline{C}^{w}$ contains
a regular point of $V^{w}$ if and only if the set of reflections
fixing $\overline{C}^{w}$ coincides with the set of reflections fixing
$V^{w}$, if and only if $\tilde{W}_{w}$ is the subgroup of $\tilde{W}$
of elements fixing $V^{w}$ (by part (iii)). But since $w$ lies in
$W_{w}$, that is true if and only if $w$ is elliptic in $W_{w}$.
\end{proof}
\begin{prop}
Let $W=\Omega\ltimes\tilde{W}$ be a twisted finite Coxeter group
and pick an element $w=\delta\tilde{w}$. 

\begin{enumerate}[\normalfont(i)]

\item There is an orthogonal decomposition $V^{w}=V_{W_{w}}^{w}\oplus(V_{W_{w}}^{\bot})^{\delta}$.

\item $(V_{W_{w}}^{\bot})^{\delta}$ is spanned by the vectors
\begin{equation}
\{\sum_{k=1}^{\mathrm{ord}(\delta)}\delta^{k}(\omega_{i}):s_{i}\notin W_{w}\}.\label{eq:basis-complement}
\end{equation}

\item The intersection $\overline{C}^{w}$ spans $(V_{W_{w}}^{\bot})^{\delta}$.

\end{enumerate}
\end{prop}

\begin{proof}
(i): As $w$ lies in $W_{w}$, the action of $w$ preserves $V_{W_{w}}$
and thus its orthogonal complement $V_{W_{w}}^{\bot}$, so that $V^{w}=V_{W_{w}}^{w}\oplus(V_{W_{w}}^{\bot})^{w}$.
As $\tilde{w}$ lies in $\tilde{W}_{w}$ we have $V_{\tilde{w}}\subseteq V_{W_{w}}$,
so $V_{W_{w}}^{\bot}\subseteq V^{\tilde{w}}$ and then $(V_{W_{w}}^{\bot})^{w}=(V_{W_{w}}^{\bot})^{\delta}$.

(ii): This follows because $V_{W_{w}}^{\bot}$ is spanned by $\{\omega_{i}:s_{i}\notin W_{w}\}$.

(iii): By construction we have $\overline{C}^{w}\subseteq V^{w}\cap V_{W_{w}}^{\bot}=(V_{W_{w}}^{\bot})^{\delta}$,
and then the claim follows from the explicit description of $(V_{W_{w}}^{\bot})^{\delta}$
in (ii); those basis vectors lie inside of $\overline{C}$.
\end{proof}
\begin{lem}
Let $w=\delta\tilde{w}$ be an element of a twisted finite Coxeter
group $W$, let $s_{i}$ be a simple reflection and consider the conjugate
$w':=s_{i}ws_{i}$. Then
\end{lem}

\[
\overline{C}^{w}\textrm{ }\begin{cases}
\subseteq\overline{C}^{w'} & \textrm{if }s_{i}\in W_{w},\\
=\textrm{cone on }\overline{C}^{w'}\textrm{ and }\sum_{j=1}^{\mathrm{ord}(\delta)}\delta^{j}(\omega_{i}) & \textrm{if }s_{i}\notin W_{w}.
\end{cases}
\]

\begin{proof}
Case 1: If $s_{i}$ lies inside of $W_{w}$ then it fixes $\overline{C}^{w}$,
and then so does $w'=s_{i}ws_{i}$.

Case 2: Since $\delta$ lies in $W_{w}$ and $s_{i}$ does not, neither
does $\delta^{k}(s_{i})$ for any integer $k$. It follows from the
previous proposition that $\sum_{k=1}^{\mathrm{ord}(\delta)}\delta^{k}(\omega_{i})$
lies in $(V_{W_{w}}^{\bot})^{\delta}\subseteq V^{w}$. Since $(V_{W_{w}}^{\bot})^{\delta}\cap V_{W_{w}}^{w}\subseteq V_{W_{w}}^{\bot}\cap V_{W_{w}}=\varnothing$,
we may shift each element of $V_{W_{w}}^{w}\subseteq V^{w}$ by this
one to obtain an $s_{i}$-invariant element, yielding a linear (but
not orthogonal) decomposition 
\[
V^{w}=V_{W_{w}}^{w,i}\oplus(V_{W_{w}}^{\bot})^{\delta}.
\]
Now consider an element $v=(v_{0},v_{1})$ of
\[
s_{i}(V^{w})=V_{W_{w}}^{w,i}\oplus s_{i}\bigl((V_{W_{w}}^{\bot})^{\delta}\bigr).
\]
By construction, the expansion of $v_{0}=\sum_{j=1}^{\mathrm{rk}}c_{j}\omega_{j}$
in the basis of fundamental weights satisfies $c_{i}=0$. If $c_{j}=0$
for all $s_{j}$ in $W_{w}$ then it lies in $V_{W_{w}}^{\bot}$ which
implies $v_{0}=0$. So if $v_{0}\neq0$ then $c_{j'}\neq0$ for some
$s_{j'}$ in $W_{w}$; in fact, we we claim that $c_{j'}>0$ and $c_{j''}<0$
for some other $s_{j''}$ in $W_{w}$: if not, say all of these coordinates
corresponding to $W_{w}$ are negative then by multiplying with $-1$
we can make them all positive, and then by adding to $v$ suitable
multiples of the basis elements of $(V_{W_{w}}^{\bot})^{\delta}$
that are given in (\ref{eq:basis-complement}) we obtain an element
in $V^{w}$ all of whose coefficients $c_{j}$ are positive, so it
lies in $V^{w}\cap\overline{C}$ but since $s_{j'}\in W_{w}$ this
is a contradiction. 

Now remove $\sum_{k=1}^{\mathrm{ord}(\delta)}\delta^{k}(\omega_{i})$
from the set of vectors (\ref{eq:basis-complement}) and call the
resulting span $(V_{W_{w}}^{\bot,i})^{\delta}$, so
\[
(V_{W_{w}}^{\bot})^{\delta}=(V_{W_{w}}^{\bot,i})^{\delta}\oplus\mathbb{R}\bigl(\sum_{k=1}^{\mathrm{ord}(\delta)}\delta^{k}(\omega_{i})\bigr)
\]
and thus
\[
s_{i}\bigl((V_{W_{w}}^{\bot})^{\delta}\bigr)=(V_{W_{w}}^{\bot,i})^{\delta}\oplus\mathbb{R}\bigl(s_{i}\sum_{k=1}^{\mathrm{ord}(\delta)}\delta^{k}(\omega_{i})\bigr).
\]
If $\omega_{i}$ occurs $n>0$ times in $\sum_{k=1}^{\mathrm{ord}(\delta)}\delta^{k}(\omega_{i})$
then
\[
s_{i}\bigl(\sum_{k=1}^{\mathrm{ord}(\delta)}\delta^{k}(\omega_{i})\bigr)=\sum_{k=1}^{\mathrm{ord}(\delta)}\delta^{k}(\omega_{i})-n\alpha_{i}
\]
 has a negative $\omega_{i}$-coefficient, and the rest is positive.
No other basis elements has an nontrivial $\omega_{i}$-coefficient
however; thus if $v_{0}\neq0$ then $v$ does not lie in $\overline{C}$,
and hence

\[
V^{w'}\cap\overline{C}=s_{i}(V^{w})\cap\overline{C}=s_{i}\bigl((V_{W_{w}}^{\bot})^{\delta}\bigr)\cap\overline{C}.
\]
And as $s_{i}$ does not fix the vector $\sum_{k=1}^{\mathrm{ord}(\delta)}\delta^{k}(\omega_{i})$
it now follows that 
\[
V^{w'}\cap\overline{C}=s_{i}\bigl((V_{W_{w}}^{\bot})^{\delta}\bigr)\cap\overline{C}=(V_{W_{w}}^{\bot,i})^{\delta}\cap\overline{C},
\]
and the claim follows.
\end{proof}
\begin{lem}
\label{lem:standard-parab} Let $w$ be an element of a twisted finite
Coxeter group.

\begin{enumerate}[\normalfont(i)]

\item Set $w':=s_{i}ws_{i}$ for some simple reflection $s_{i}$.
If $V^{w'}\cap\overline{C}$ is not contained in $V^{w}\cap\overline{C}$,
then
\[
\ell(w)>\ell(w').
\]
In particular, if $w$ has minimal length, then the closure of the
dominant Weyl chamber contains a regular point of its fixed space
$V^{w}$.

This implies that an element of minimal length in its conjugacy class
lies in a standard parabolic subgroup of lower rank if and only if
it is not elliptic.\footnote{The untwisted version is the main lemma of \cite{MR1425324} which
is deduced there from extensive case-by-case computations, whilst
Howlett also gave a proof of the untwisted statement using the parabolic
Burnside ring \cite[Proposition 3.1.12]{MR1778802}; his proof extends,
ours is entirely different.}

\item If the closure of the dominant Weyl chamber contains a nonzero
point of $V_{w}$ (e.g.\ $w$ is a nontrivial dominant element),
then $w$ does not lie inside a standard parabolic subgroup of $W$
which has lower rank than $W$.

In particular, if furthermore $w$ lies in the untwisted part $\tilde{W}$
then its length can be recovered from its Alexander polynomial (see
\cite{WM-Alex} or \cite[Theorem 6]{trinh2021hecke}) and
\[
\ell(w)\geq\mathrm{rank}(W),
\]
 with equality if and only if $w$ is a Coxeter element of minimal
length.\footnote{This statement can be used to simplify Michel's proof \cite{wilson}
of Kamgarpour's inequality \cite{kamgarpour2015stabilisers}.}

\end{enumerate}
\end{lem}

\begin{proof}
(i): The previous lemma implies that $s_{i}\notin W_{w}$. Thus
\[
\ell(s_{i}w)=\ell(w)+1=\ell(ws_{i}).
\]
If $s_{i}w=ws_{i}$ then we obtain $w=s_{i}ws_{i}=w'$ which contradicts
the assumptions. Hence $s_{i}w\neq ws_{i}$, so from \cite[Exercice VI.2.23]{MR0240238}
we conclude that $\ell(w)>\ell(w')$. 

As we've shown in Lemma \ref{lem:shifts}(i), we can always find a
conjugate $w'$ of $w$ with $w\overset{-}{\rightarrow}w'$, such
that $\overline{C}$ contains a regular point of $V^{w'}$. Thus if
$\overline{C}$ does not contain a regular point of $V^{w}$, the
length must drop at least once in the corresponding sequence. By say
Corollary \ref{cor:parabolic}(i) it now follows that an element of
minimal length lies in a proper parabolic subgroup if and only if
it is not elliptic.

(ii): Suppose that $w$ does lie in a standard parabolic subgroup
of lower rank, then $V_{w}$ lies in the span of the corresponding
simple roots, so any nontrivial element $v$ in the intersection of
$V_{w}$ with the closure of the dominant Weyl chamber $C$ decomposes
as a sum of strictly positive scalar multiples of a proper subset
of the simple roots. (Here we're using that the dominant Weyl chamber
lies inside of its dual cone, i.e.\ that each fundamental weight
can be expressed as a positive (rational) sum of simple roots.) Then
as the Coxeter-Dynkin diagram of $W$ is connected, there exists a
root $\alpha_{j}$ in this proper subset and a root $\alpha_{k}$
outside of it, such that $(\alpha_{j},\alpha_{k})<0$. Then for this
vector $v=\sum_{i=1}^{\mathrm{rk}}c_{i}\alpha_{i}$ we have
\[
(v,\alpha_{k})=\sum_{i=1}^{\mathrm{rk}}c_{i}(\alpha_{i},\alpha_{k})\leq c_{j}(\alpha_{j},\alpha_{k})<0
\]
 but then $v$ does not lie in $\overline{C}$, which is a contradiction.

If $w$ lies in the untwisted part $\tilde{W}$ it now follows that
each simple reflection must occur at least once in a reduced decomposition
of $w$, and if they do each occur once then $w$ is a Coxeter element.
\end{proof}
\begin{example}
Consider the reflections in the roots $\alpha_{12}$ and $\alpha_{1112}$
in type $\mathsf{G}_{2}$. For both elements we have $V_{w}\cap\overline{C}=\{0\}$,
but neither of them lies in a standard parabolic subgroup.
\end{example}

\begin{example}
Consider the elements $w=s_{1}us_{1}$ and $w'=s_{3}ws_{3}$ in the
conjugacy class of $u=s_{2}s_{3}s_{4}$ in the Weyl group $W$ of
type $\mathsf{A}_{4}$. Then $V_{u}$ is spanned by $\alpha_{2},\alpha_{3},\alpha_{4}$,
and then
\begin{align*}
V_{w} & =s_{1}(V_{u})=\<\alpha_{12},\alpha_{3},\alpha_{4}\>=\<\omega_{2}-\omega_{3},-\omega_{2}+2\omega_{3}-\omega_{4},-\omega_{3}+2\omega_{4}\>=\<\omega_{2},\omega_{3},\omega_{4}\>\\
V_{w'} & =s_{3}(V_{w})=\<\omega_{2},\omega_{3}-\alpha_{3},\omega_{4}\>=\<\omega_{2},\omega_{2}-\omega_{3}+\omega_{4},\omega_{4}\>=\<\omega_{2},\omega_{3},\omega_{4}\>
\end{align*}
implies that both $w$ and $w'$ are dominant. Since 
\[
\ell(w')>\ell(w)=5=\mathrm{rank}(W)+1
\]
 only $w$ is minimally dominant, but the conjugate $s_{1}w's_{1}=s_{3}us_{3}$
lies in a standard parabolic subgroup and is hence not dominant.
\end{example}

The analogous statement for (ii) is false for elliptic elements when
the subspace of fixed points is repaced by the eigenspace whose eigenvalue
has minimal argument:
\begin{example}
Consider the Coxeter element $s_{1}s_{2}s_{3}$ in type $\mathsf{A}_{3}$.
This subspace is its Coxeter plane, which does not pass through the
dominant Weyl chamber.
\end{example}

\begin{proof}
[Proof of the main Lemma] (i): For the first equality, note that
$\CMcal O_{+}^{\Theta}\subseteq\CMcal O^{\mathrm{dom}}$, where $\Theta$
is the set of all eigenspaces (simply remove the eigenspaces with
eigenvalue 1 at the end of the sequence), but also $\CMcal O_{+}^{\Theta}\subseteq\CMcal O_{\mathrm{max}}^{\Theta}=\CMcal O_{\mathrm{max}}$
by Proposition \ref{prop:shifts-braiding}(ii). The rest follows from
part (i) of the previous lemma.

(ii): Follows from part (ii) and (iii) of Proposition \ref{prop:shifts-braiding}.

(iii): Follows from equation (\ref{eq:length-formula}) and from part
(ii) of the previous lemma.
\end{proof}

\subsection{\label{subsec:Bruhat}Bruhat cells}

In this subsection we study the relationship between minimally dominant
elements and Bruhat cells, outlining a possible approach to the main
Conjecture.

Lusztig used a character formula for Iwahori-Hecke algebra basis elements
(which is useful computationally, see also \cite[§6-§7]{MR3343221})
to show that his constructions do not depend on the conjugacy class
of the braid $b_{w}$, when working over an algebraically closed field
\cite[§1.2]{MR2833465}. According to Proposition \ref{prop:mixed-shift}(ii),
this conjugacy class can be obtained from strong conjugations and
cyclic shifts; for the latter, it is not necessary to make any assumptions
on the base:
\begin{prop}
\label{prop:conjugacy-cyclic} Let $C$ be a conjugacy class of a
twisted reductive group $G$ with Borel subgroup $B$ containing a
maximal torus. Let $w$ and $w'$ be elements of its Weyl group lying
in the same cyclic shift class. Then for every conjugacy class $C$
in $G$ we have
\[
C\cap BwB\neq\varnothing\qquad\textrm{if and only if}\qquad C\cap Bw'B\neq\varnothing.
\]
\end{prop}

\begin{proof}
Let $N_{+}:=[B,B]$ be the unipotent radical of $B$, let $N_{-}$
be the unipotent radical of the opposite Borel subgroup and then set
$N_{z}:=N_{+}\cap z^{-1}N_{-}z$ for any element $z$ in this (twisted)
Weyl group. By induction, it suffices to consider the case where $w=xy$
and $w'=yx$ are both reduced decompositions. If an element $bxyb'$
of $BxyB$ lies in the conjugacy class $C$, then so does the element
$yb'bx$ which furthermore lies inside of
\[
ByBxB=ByBxN_{w}=ByN_{y}xN_{x}=Bw'N_{w'}=Bw'B.\qedhere
\]
\end{proof}
The following statement was partially inspired by constructions in
\cite[p.\,336-337]{MR3883243}.
\begin{lem}
Let $w$ be a minimally dominant element of a twisted finite Coxeter
group with dominant Weyl chamber $C$. Then there exists a Weyl chamber
$C'$ such that

\begin{enumerate}[\normalfont(i)]

\item $w$ has minimal length with respect to $C'$.

\item $V_{w}$ is generated by the roots of a standard parabolic
subroot system in $C'$.

\item The positive roots for $C$ inside $V_{w}$ coincide with the
positive roots for $C'$ in $V_{w}$.

\end{enumerate}
\end{lem}

We've already seen in Lemma \ref{lem:standard-parab}(i) that (ii)
follows from (i), but we will explicitly use it here in proving (i):
\begin{proof}
Given an arbitrary Weyl chamber $\tilde{C}$, we denote by $\ell_{\tilde{C}}(\cdot)$
the corresponding length function. Let $v_{0}$ be a regular element
in $V^{w}$ and $v_{1}$ be a regular element in $V_{w}\cap\overline{C}$.
Since $w$ is dominant, $v_{1}$ is also a regular element of $V_{w}$.
In particular, $v_{0}+\varepsilon v_{1}$ is regular for $\varepsilon>0$
sufficiently small, defining a Weyl chamber $C'$ satisfying the condition
of (iii). By construction, the closure $\overline{C}'$ contains a
regular point of $V^{w}$, so the isotropy subgroup fixing it is a
standard parabolic subgroup $W'$, satisfying $V_{W'}\subseteq V_{w}=(V^{w})^{\bot}$.
As this subgroup also contains $w$, this inclusion must be an equality.
Thus (ii) holds, which implies that $w$ lies inside of this standard
parabolic subgroup. Then $\ell_{C'}(w)=|\mathfrak{R}_{w}\cap V_{w}|$,
for the inversion set $\mathfrak{R}_{w}$ defined with respect to
either $C$ or $C'$. 

Let $\CMcal O$ denote the $\tilde{W}$-orbit of $w$. According to
Proposition \ref{prop:shifts-braiding}(ii) there is an inclusion
$\CMcal O_{-}^{\mathrm{dom}}\subseteq\CMcal O_{\mathrm{min}}^{\mathrm{dom}}$,
so by (iii) of this proposition we can cyclic shift $w$ to an element
$w'$ in $\CMcal O_{-}^{\mathrm{dom}}$; applying the same reasoning,
we obtain a Weyl chamber $C''$ satisfying (ii) and (iii) for this
element. But with respect to $C''$ the element $w'$ lies in $\CMcal O_{-}$
and hence in $\CMcal O_{\mathrm{min}}$ by Proposition \ref{prop:shifts-braiding}(ii)
again. Then from Corollary \ref{cor:roots-made-negative-mixed} we
deduce that
\[
\ell_{C'}(w)=|\mathfrak{R}_{w}\cap V_{w}|=|\mathfrak{R}_{w'}\cap V_{w'}|=\ell_{C''}(w')
\]
 so (i) follows for $w$ as well.
\end{proof}
\begin{example}
Consider the conjugate of $s_{1}s_{2}s_{3}s_{4}s_{3}s_{1}s_{2}$ of
$s_{1}s_{2}s_{3}$ in type $\mathsf{A}_{4}$. It is dominant but not
minimally dominant, and there is no Weyl chamber such that the conditions
of the lemma hold.
\end{example}

\begin{cor}
\label{cor:minimally-dominant-to-minimal} Let $w$ be a minimally
dominant element in a twisted finite Coxeter group. Then there exists
a sequence of simple reflections $s_{i_{0}},\ldots,s_{i_{n-1}}$ such
that upon setting $w_{j+1}:=s_{i_{j}}w_{j}s_{i_{j}}$ for $0\leq j\leq n-1$,
the properties
\[
\ell(w_{j})\geq\ell(w_{j+1})\qquad\textrm{and}\qquad\alpha_{i_{j}}\notin V_{w_{j}}
\]
 hold and $w_{n}$ has minimal length.
\end{cor}

\begin{proof}
Take any regular point $v_{1}$ of $V_{w}$ lying in $\overline{C}$
and regular point $v_{0}$ of $V^{w}$ such that $\varepsilon v_{0}+v_{1}$
lies in $C$ and is regular for $\varepsilon>0$ sufficiently small.
The gradient flow explicitly given in Proposition \ref{prop:gradient-flow}
now defines a straight line from the half-line through $\varepsilon v_{0}+v_{1}$
to the half-line through $v_{0}+\varepsilon v_{1}$, and hence from
$C$ to the chamber $C'$ constructed in the previous proof. The vectors
along this path then yield a continuous family of hyperplanes where
the positive roots of $V_{w}$ remain positive, i.e.\ condition (iii)
must be true at each of the chambers along this straight path. But
this implies that along this path, $\alpha_{i_{j}}\notin V_{w_{j}}$
is always satisfied. The condition $\ell(w_{j})\geq\ell(w_{j+1})$
follows from Corollary \ref{cor:gradient-function}(ii).
\end{proof}
\begin{conjecture}
\label{conj:intersect-bruhat} Let $G$ be a twisted reductive group
over an algebraically closed field, with Borel subgroup $B$. Let
$w$ be an element of the corresponding Weyl group and $s_{i}$ a
simple reflection such that $\alpha_{i}\notin V_{w}$ and set $w':=s_{i}ws_{i}$.
Then for every conjugacy class $C$ in $G$ we have

\begin{equation}
C\cap BwB\neq\varnothing\qquad\textrm{if and only if}\qquad C\cap Bw'B\neq\varnothing.\label{eq:intersect-conjugacy}
\end{equation}
\end{conjecture}

As shown in \cite[Example 3.6]{MR2042932}, this fails when the underlying
field is not algebraically closed. Note that by \cite[Proposition 3.4]{MR2042932},
it suffices to prove $\Rightarrow$ with $\ell(w)>\ell(w')$. Assuming
this conjecture, one obtains
\begin{cor}
Let $w$ be minimally dominant element and $w\overset{-}{\rightarrow}w'$
for some element $w'$. Then (\ref{eq:intersect-conjugacy}) holds.
\end{cor}

\begin{proof}
By say Proposition \ref{prop:shifts-braiding}(i) there exists a minimal
length element $w''$ such that $w'\overset{-}{\rightarrow}w''$,
and thus $w''\overset{+}{\rightarrow}w'\overset{+}{\rightarrow}w$.
From \cite[Proposition 3.4]{MR2042932} it then follows that
\[
C\cap Bw''B\neq\varnothing\quad\implies\quad C\cap Bw'B\neq\varnothing\quad\implies\quad C\cap BwB\neq\varnothing.
\]
Now let $w'''$ denote the minimal length element constructed from
$w$ in Corollary \ref{cor:minimally-dominant-to-minimal}. Then the
previous conjecture coupled with strong conjugation invariance yields
\[
C\cap BwB\neq\varnothing\quad\implies\quad C\cap Bw'''B\neq\varnothing\quad\implies\quad C\cap Bw''B\neq\varnothing.\qedhere
\]
\end{proof}
And that would yield the main Conjecture:
\begin{proof}
By Proposition \ref{prop:shifts-braiding}(i) there exists a minimal
length element $w''$ such that $w'\overset{-}{\rightarrow}w''$.
The previous corollary and strong conjugation invariance then yield
\[
C\cap BwB\neq\varnothing\quad\Longleftrightarrow\quad C\cap Bw''B\neq\varnothing\quad\Longleftrightarrow\quad C\cap Bw'B\neq\varnothing.\qedhere
\]
\end{proof}

\subsection{Involutions}

In this subsection we prove Proposition \ref{prop:involutions}. My
primary contribution here is generalising known results about conjugacy
classes to orbits of standard parabolic groups, and in the new characterisation
given in part (ii)(d).
\begin{notation}
Given an element $w$ of a twisted finite Coxeter group, we let $\mathfrak{R}_{-1}^{w}$
denote the set of roots on which it acts as $-1$. Given a subset
$J$ of simple reflections, we denote by $\tilde{W}_{J}$ the standard
parabolic subgroup generated by $J$ and by $_{J}\mathfrak{R}_{+}$
the corresponding set of positive roots.
\end{notation}

The explicit description for part (ii)(c) of Proposition \ref{prop:involutions}
is:\begin{equation}   \label{eq:explicit-description-involution}   \parbox{\dimexpr\linewidth-4em}{Let $J$ denote the set of simple reflections such that $W'=\tilde{W}_{J}$ and consider the subset
\[ J':=\{s_{j}\in J:s_{j}ws_{j}=w\textrm{ and }\ell(w)\lessgtr\ell(ws_{j})\} \subseteq J. \]
 Decompose $w$  into $w=w^{J'}w_{J'}$, where $w_{J'}$ denotes the identity/longest element of $\tilde{W}_{J'}$ and $w^{J'}$ is the  maximal/minimal left coset representative for $w$ in $\tilde{W}_{J'}$. Then
\[ _{J}\mathfrak{R}_{+}\cap{}_{J'}\mathfrak{R}_{+}={}_{J}\mathfrak{R}_{+}\cap\mathfrak{R}^{w}\qquad(\textrm{resp.}\quad{}_{J}\mathfrak{R}_{+}\cap{}_{J'}\mathfrak{R}_{+}={}_{J}\mathfrak{R}_{+}\cap\mathfrak{R}_{-1}^{w}), \]
the element $w^{J'}$ is an involution, a maximal/minimal length \emph{double} coset representative in $\tilde{W}_{J}\times\tilde{W}_{J'}$ and permutes the simple roots of $_{J'}\mathfrak{R}_{+}$.
} \end{equation}

For part (ii)(d) we will need
\begin{prop}
Let $V$ be a vector space over a field with a half-space $\mathbb{H}\subseteq V$
and a nontrivial subspace $V'\subseteq V$. Let $X$ be a subset of
the dual vector space $V^{*}$. Then either there exists in $X$ a
functional $\alpha$ which is nonpositive and nontrivial on some orthogonal
basis for $V'$ lying in $\mathbb{H}$, or there exists a nontrivial
vector in $V'\cap\mathbb{H}$ on which all elements of $X$ evaluate
nonnegatively.
\end{prop}

\begin{proof}
Let $\mathfrak{h}$ be the hyperplane bordering $\mathbb{H}$. Since
$\mathfrak{h}$ has codimension 1 the intersection $V'\cap\mathfrak{h}$
has dimension $\dim(V')-1$ or $\dim(V')$. Thus there exists an orthogonal
basis $\{v_{i}'\}$ for $V'$ consisting of $\dim(V')-1$ vectors
lying in $\mathfrak{h}$ and another vector $v'$ lying in $\mathbb{H}$.
If the values $\alpha(v')$ are nonnegative for each $\alpha$ in
$X$, we are done. Otherwise, at least one of them is negative. Multiplying
the other basis vectors of $V'$ (which lie in the linear space $\mathfrak{h}$)
by $-1$ if this $\alpha$ evaluates positively on them, we obtain
an orthogonal basis for $V'$ lying in $\mathbb{H}$ on which $\alpha$
always evaluates nonpositively.
\end{proof}

The main proof of Hart-Rowley in \cite{MR2040170} can be rephrased
as
\begin{lem}
[{\cite{MR2040170}}] Let $w$ be an involution of a twisted Coxeter
group and let $s_{j}$ be a simple reflection. Suppose that $\alpha_{j}\notin\mathfrak{R}_{w}\sqcup\mathfrak{R}^{w}$
(resp.\ $\alpha_{j}\in\mathfrak{R}_{w}\backslash\mathfrak{R}_{-1}^{w}$).
Then
\[
\ell(s_{j}ws_{j})\gtrless\ell(w).
\]
\end{lem}

\begin{proof}
Using that $w$ is an involution, we deduce from $\alpha_{j}\notin\mathfrak{R}_{w}$
(resp.\ $\alpha_{j}\in\mathfrak{R}_{w}$) that 
\[
\ell(s_{j}w)=\ell(ws_{j})\gtrless\ell(w).
\]
From $\alpha_{j}\notin\mathfrak{R}_{w}\sqcup\mathfrak{R}^{w}$ (resp.\
$\alpha_{j}\in\mathfrak{R}_{w}\backslash\mathfrak{R}_{-1}^{w}$) it
follows that $\alpha_{j}\notin\mathfrak{R}_{s_{j}w}$ (resp. $\alpha_{j}\in\mathfrak{R}_{s_{j}w}$),
and hence
\[
\ell(s_{j}ws_{j})\gtrless\ell(s_{j}w)\gtrless\ell(w).\qedhere
\]
\end{proof}

For the implications $(b)\Rightarrow(c)\Rightarrow(a)$ we build upon
a proof of Howlett \cite[Proposition 3.2.10]{MR1778802}, just as
\cite[Lemma 3.6]{MR2355597} did for the twisted case.\footnote{The assumption $\delta^{2}=\mathrm{id}$ is missing there; see Example
\ref{exa:fake-involution}.} The latter however only considered the standard parabolic subgroup
$\tilde{W}$. Richardson originally obtained his result through proving
the minimal case (for $W'=\tilde{W}$) of $(d)\Rightarrow(c)$ \cite{MR679916};
we will not follow his approach.
\begin{proof}
[Proof of Proposition \ref{prop:involutions}]

(ii), $(a)\Rightarrow(b)$: If the statement in $(b)$ is false, then
the previous lemma implies that $\ell(s_{j}ws_{j})\gtrless\ell(w)$
for some $s_{j}$ in $W'$. 

(ii), $(b)\Rightarrow(c)$: From $w_{J'}^{2}=\mathrm{id}$ and the
properties of $J'$ it follows that
\[
\mathrm{id}=w^{2}=w^{J'}w_{J'}w=w^{J'}ww_{J'}=w^{J'}w^{J'}w_{J'}w_{J'}=(w^{J'})^{2},
\]
so that we may also write $w=w_{J'}w^{J'}$. Given a simple reflection
$s_{j''}\in J'$, then from the definition of $w_{J'}$ it follows
that know that $w_{J'}s_{j''}w_{J'}=s_{j'}$ for another simple reflection
$s_{j'}\in J'$. Thus
\[
w^{J'}s_{j'}w^{J'}=w^{J'}w_{J'}s_{j''}w_{J'}w^{J'}=ws_{j''}w=s_{j''}
\]
 which proves the claim on permuting simple roots.

Suppose now that $w^{J'}$ is not a maximal/minimal double coset representative
for $\tilde{W}_{J}\times\tilde{W}_{J'}$, so there exists a simple
reflection $s_{j}$ in $J$ such that $\ell(s_{j}w^{J'})\gtrless\ell(w^{J'})$.
As $w^{J'}$ is a maximal/minimal left coset representative for $\tilde{W}_{J'}$,
this inequality implies that $s_{j}w^{J'}$ is as well. Then 
\[
\ell(s_{j}w)=\ell(s_{j}w^{J'}w_{J'})=\ell(s_{j}w^{J'})+\ell(w_{J'})\gtrless\ell(w^{J'})+\ell(w_{J'})=\ell(w)
\]
 and as $w$ is an involution it also follows that $\ell(ws_{j})\gtrless\ell(w)$.
The element $s_{j}w$ lies in the coset $s_{j}w^{J'}\tilde{W}_{J'}$,
and if $s_{j}\in J'$ then $s_{j}w=ws_{j}$ which would imply that
it also lies in the coset $w^{J'}\tilde{W}_{J'}$. But as both $s_{j}w^{J'}$
and $w^{J'}$ are maximal/minimal coset representatives, this is a
contradiction. Hence $s_{j}w\neq ws_{j}$, which by \cite[Exercice VI.2.23]{MR0240238}
implies that $\ell(s_{j}ws_{j})\gtrless\ell(w)$. We conclude that
$\alpha_{j}$ does not lie in $\mathfrak{R}_{w}\sqcup\mathfrak{R}^{w}$
(resp.\ lies in $\mathfrak{R}_{w}\backslash\mathfrak{R}_{-1}^{w}$),
which contradicts the assumption of $(b)$. 

(ii), $(c)\Rightarrow(a)$: Pick any element $x$ in $\tilde{W}_{J}$
and factorise it as $x^{J'}x_{J'}$ as before, so $x_{J'}$ lies in
$\tilde{W}_{J'}$ and $x^{J'}$ is a maximal/minimal left coset representative.
By definition of $J'$, we have
\[
xwx^{-1}=x^{J'}(x_{J'}^{\,}wx_{J'}^{-1})(x^{J'})^{-1}=x^{J'}w(x^{J'})^{-1}.
\]
As $x^{J'}$ lies in $\tilde{W}_{J}$ it then follows that
\[
\ell(xwx^{-1})=\ell\bigl(x^{J'}w(x^{J'})^{-1}\bigr){\leq\atop \geq}\ell(x^{J'}w^{J'}w_{J'})\pm\ell(x^{J'})=\ell(x^{J'})\mp\ell(w^{J'}w_{J'})\pm\ell(x^{J'})=\ell(w),
\]
and hence $w$ is of maximal/minimal length in its $\tilde{W}_{J}$-orbit.

(ii), $(b)\Leftrightarrow(d)$: By (a twisted analogue of) \cite[Exercice V.4.2(d)]{MR0240238},
we may assume that $W$ is finite. Let $v$ be a regular point of
the dominant Weyl chamber $C$, so that the halfplane
\[
\mathbb{H}_{v}:=\{x\in V:(x,v)\geq0\}
\]
 separates the positive roots from the negative roots. Let $v_{w}$
denote a regular point of $V_{\mp1}^{w}\cap\overline{C}$.

$(d)\Rightarrow(b)$: $v_{w}$ is also a regular point of $V_{\mp1}^{w}$.
For any root $\beta$ we have
\[
\bigl(w(\beta),v_{w}\bigr)=\bigl(\beta,w^{-1}(v_{w})\bigr)=\mp(\beta,v_{w}).
\]
In the $-$-case, as the nontrivial eigenvalues of an involution are
all $-1$ and $V_{w}$ thus coincides with $V_{-1}^{w}$, we deduce
from Lemma \ref{lem:regular-vector-dominant} that for any positive
root $\beta$, the root $w(\beta)$ is negative if and only if $\beta$
is not fixed by $w$. In the $+$-case, since $v_{w}$ lies in the
dominant Weyl chamber this equation implies that $w(\beta)$ is positive
unless $\bigl(w(\beta),v_{w}\bigr)=0$. But by regularity this implies
that $\beta$ is orthogonal to $V_{1}^{w}$, so as $V=V_{1}^{w}\oplus V_{-1}^{w}$
this means that $\beta$ lies in $V_{-1}^{w}$.

$(b)\Rightarrow(d)$: On the other hand, suppose that $v_{w}$ is
not a regular point of $V_{\mp1}^{w}$. Let $I_{w}$ denote the set
of indices corresponding to hyperplanes $\mathfrak{h}_{i}$ of simple
reflections $s_{i}$ not containing $V_{\mp1}^{w}$. We split it $I_{w}=J_{w}\sqcup J_{w}'$
into the subset of indices $J_{w}$ such that $\mathfrak{h}_{i}$
doesn't contain $\overline{C}\cap V_{\mp1}^{w}$, and the subset $J_{w}'$
of indices for which it does; according to Proposition \ref{prop:parabolic}(iii)
the subset $J_{w}'$ is now nonempty. Consider the basis of $V$ of
fundamental weights $\{\omega_{i}\}_{i=1}^{\mathrm{rk}}$, then evidently
for each $j$ in $J_{w}$ there exists an vector in $\overline{C}\cap V_{\mp1}^{w}$
that has a nontrivial (and thus positive) $\omega_{j}$-component.
Taking for each $j$ such an element and summing them, we obtain an
element $v_{w}^{J}=\sum_{j\in J_{w}}c_{j}\omega_{j}$ lying in $\overline{C}\cap V_{\mp1}^{w}$.
We let $V'$ denote the orthogonal complement inside $V_{\mp1}^{w}$
to the span of $\overline{C}\cap V_{\mp1}^{w}$.

Suppose that there exists a nontrivial vector $v'$ in $V'$ on which
all the $\alpha_{j}$ for $j\in J_{w}'$ evaluate nonnegatively. By
adding a sufficiently large positive multiple of $v_{w}^{J}$ to $v'$,
we obtain an element whose remaining coordinates in this basis are
positive. But by construction this element lies in the span of $\overline{C}\cap V_{\mp1}^{w}$,
and hence so does $v'$ which is a contradiction.

The previous proposition then yields an orthogonal basis $\{v_{i}'\}_{i=1}^{\dim V'}$
for $V'$ lying in $\mathbb{H}_{v}$ and a particular $j\in J_{w}'$
such that $\alpha_{j}$ evaluates and nonpositively on all of these
basis elements, and negatively on at least one of them. Denote the
projection of $\alpha_{j}$ to $V_{\pm1}^{w}$ by $\alpha_{j}^{\pm}$.
As $V=V_{1}^{w}\oplus V_{-1}^{w}$ and $\alpha_{j}$ projects to zero
in the orthogonal complement of $V'$ in $V_{\mp1}^{w}$ (which is
spanned by $V_{\mp1}^{w}\cap\overline{C}$), we have
\[
w(\alpha_{j})=\pm\alpha_{j}^{\pm}+\sum_{i=1}^{\mathrm{dim}V'}w\bigl((\alpha_{j},v_{i}')v_{i}'\bigr)=\pm\alpha_{j}^{\pm}\mp\sum_{i=1}^{\mathrm{dim}V'}(\alpha_{j},v_{i}')v_{i}'=\pm\alpha_{j}+2\sum_{i=1}^{\mathrm{dim}V'}(\mp\alpha_{j},v_{i}')v_{i}'.
\]
Since $\alpha_{j}$ and $v_{i}'$ all lie in $\mathbb{H}_{v}$ and
$(-\alpha_{j},v_{i}')\geq0$, so does $\pm w(\alpha_{j})$, so $w(\alpha_{j})$
is a positive/negative root. As the inequality is strict for at least
one $i$, $\alpha_{j}$ does not lie in $V_{\pm1}^{w}$, so the assumption
of $(b)$ does not hold.

(i): If $w$ does not have maximal/minimal length, then by (ii) it
does not satisfy equation (\ref{eq:max-min-involution}). From the
lemma we deduce that $w\overset{\pm}{\twoheadrightarrow}s_{j}ws_{j}$
for some $s_{j}$ in $W'$, and then the claim follows from induction.
\end{proof}
\begin{defn}
Consider the element $s_{1}w_{\circ}s_{1}=w_{\circ}s_{1}s_{3}$ in
type $\mathsf{A}_{3}$. This is an involution of the form $w_{\circ}w_{J}$,
but as it is not of maximal length (and does not fix any roots) it
is not dominant.
\end{defn}

\begin{example}
\label{exa:fake-involution} Consider the element $w:=\delta w_{\circ}$
in type $\mathsf{D}_{4}$, where $\delta$ is the diagram automorphism
\[
1\mapsto3\mapsto4\mapsto1.
\]
As $\delta^{-1}w_{\circ}\delta=w_{\circ}=w_{\circ}^{-1}$ (for any
twist) it follows that $w^{2}=\delta^{2}\neq e$, but $\ell(s_{1}ws_{1})<\ell(w)$
so $w$ is not of minimal length in its conjugacy class (another example
is $\delta s_{4}s_{3}s_{1}=s_{1}\delta s_{3}s_{1}$). Further conjugating
$s_{1}ws_{1}$ with $s_{2}$ and then $s_{3}$, we find
\[
w\overset{-}{\twoheadrightarrow}\delta s_{3}s_{2}s_{4}s_{1}s_{2}s_{1}
\]
which has minimal length, but is not of the form $\delta w_{J'}$
(which is what the last requirement in (\ref{eq:explicit-description-involution})
specialises to when $W'=\tilde{W}$).
\end{example}

\begin{prop}
Orbits for standard parabolic subgroups of involutions in a twisted
finite Coxeter group can be algorithmically classified in terms of
its Coxeter-Dynkin diagram.
\end{prop}

\begin{proof}
Let $J$ be such that the standard parabolic subgroup $W'$ equals
$\tilde{W}_{J}$. By Proposition \ref{prop:involutions}, we need
to understand when elements of the minimal case in (\ref{eq:explicit-description-involution})
are conjugate under $\tilde{W}_{J}$. So we may assume that they are
the form $w'=w^{J'}w_{J'}$ and $w''=w^{J''}w_{J''}$, where $w_{J'}$
and $w_{J''}$ are the longest elements of standard parabolic subgroups
for certain $J',J''\subseteq J$. By (\ref{eq:explicit-description-involution})
the element $w^{J'}$ (resp.\ $w^{J''}$) is a double coset representative;
in particular, they are minimal right coset representatives for $\tilde{W}_{J}$.
If $xw'x^{-1}=w''$ for some $x$ in $\tilde{W}_{J}$, then as
\[
xw^{J'}=w''xw_{J'}^{-1}=\bigl(w''xw_{J'}^{-1}(w^{J''})^{-1}\bigr)w^{J''}
\]
 and $w''xw_{J'}^{-1}(w^{J''})^{-1}$ lies in $\tilde{W}_{J}$, it
follows by uniqueness of minimal coset representatives that $w^{J'}=w^{J''}$. 

By assumption the projection of the $-1$ eigenspace to $V_{W'}$
is given by $V_{\tilde{W}_{J'}}$ and $V_{\tilde{W}_{J''}}$ respectively.
As $x$ only acts on $V_{W'}$, it thus yields a bijection between
the roots of $\tilde{W}_{J'}$ to the roots of $\tilde{W}_{J''}$.
It has to map the simple roots of $\tilde{W}_{J'}$ to a fundamental
system in $\tilde{W}_{J''}$, so after modifying $x$ by an element
of $\tilde{W}_{J''}$ (which commutes with $w^{J'}=w^{J''}$) we may
assume that it maps the simple roots of $\tilde{W}_{J'}$ to the simple
roots of $\tilde{W}_{J''}$. But then $x$ is in particular a ``$W$-equivalence'',
and those are described by Richardson's algorithm \cite{MR679916}.
\end{proof}

\section{Powers of reduced braids}

Finally, in this section we prove Proposition \ref{prop:mixed-shift},
Proposition \ref{prop:dg-bound} and the main Theorem. Throughout
we will employ the language of Garside categories wherever possible,
but I have tried to set up the notation for readers who are only familiar
with braid monoids:
\begin{defn}
\label{def:locally-garside-category} Let $\tilde{\mathrm{Br}}^{+}$
be a right-cancellative category with Garside family $\tilde{W}$
\cite[Definition III.1.31]{MR3362691}. Given a monoid $\Omega$ of
endomorphisms of $\tilde{\mathrm{Br}}^{+}$ preserving $\tilde{W}$,
we may interpret the semidirect product $\mathrm{Br}^{+}:=\Omega\ltimes\tilde{\mathrm{Br}}^{+}$
as a category: the objects are those of $\tilde{\mathrm{Br}}^{+}$,
whereas the source of an arbitrary morphism $b=\delta\tilde{b}:=(\delta,\tilde{b})$
is given by that of $\tilde{b}$, and its target is the image of the
target of $\tilde{b}$ under $\delta$; we will call this a \emph{twisted
locally Garside category}. We will say that a morphism $b$ of $\mathrm{Br}^{+}$
is \emph{reduced} if it lies in $\Omega\ltimes\tilde{W}$.

When $\tilde{\mathrm{Br}}^{+}$ is small and right-Noetherian, then
so is $\mathrm{Br}^{+}$; then there exists a right-length (or right-height)
function $\ell:=\ell_{\mathrm{Br}^{+}}$ on the morphisms of $\mathrm{Br}^{+}$
\cite[Proposition II.2.47]{MR3362691}, which extends the right-length
function of $\tilde{\mathrm{Br}}^{+}$ via $\ell_{\mathrm{Br}^{+}}(b)=\ell_{\mathrm{Br}^{+}}(\tilde{b})=\ell_{\tilde{\mathrm{Br}}^{+}}(\tilde{b})$.
\end{defn}

\begin{rem}
One can reinterpret $\mathrm{Br}^{+}$ as a locally Garside category
with Garside family $\Omega\ltimes\tilde{W}$ (or just $\tilde{W}$
if the endomorphisms in $\Omega$ are all invertible); since the semidirect
product $\Omega\ltimes\tilde{W}$ does not strictly speaking yield
a Coxeter group when $\tilde{W}$ is itself a Coxeter group, we will
not do so.
\end{rem}

As the relations of Coxeter groups are homogeneous, twisted braid
monoids are examples of small, Noetherian twisted locally Garside
categories with one object and with finite lengths \cite[Proposition II.2.32]{MR3362691};
many of the statements that we need for their Deligne-Garside normal
forms generalise immediately:

\subsection{Stability of the Deligne-Garside normal form}

Roughly speaking, the right Deligne-Garside normal form of a morphism
in such categories can be obtained by starting with any decomposition
into reduced morphisms, and subsequently trying to make the rightmost
factors as large as possible:

\begin{defthm} \label{def:dgn} Let $\mathrm{Br}^{+}:=\Omega\ltimes\tilde{\mathrm{Br}}^{+}$
be a twisted locally Garside category. Then every morphism $b$ in
$\mathrm{Br}^{+}$ admits an (essentially unique \cite[Proposition III.1.25]{MR3362691})
\emph{(right) Deligne-Garside }(or \emph{right greedy})\emph{ normal
form }
\[
b=\delta b_{m}\cdots b_{1},
\]
where $\delta\in\Omega$, and $b_{m},\ldots,b_{1}\in\tilde{\mathrm{Br}}^{+}$
are reduced morphisms such that if some morphism $b'$ right-divides
$b_{i+1}b_{i}$ for some $1\leq i\leq m-1$, then it right-divides
$b_{i}$ \cite[Proposition IV.1.20]{MR3362691}. We denote the left-complement
to $\mathrm{DG}_{i\geq}(b)$ in $\mathrm{DGN}(b)$ by
\[
\mathrm{DG}_{>i}(b):=\delta\mathrm{DG}_{m}(b)\cdots\mathrm{DG}_{i+1}(b)\in B^{+}.
\]
\end{defthm}

In the language of twisted Coxeter groups, this means that we can
uniquely decompose any element $b$ of its twisted braid monoid as
\[
b=\delta b_{w_{m}}\cdots b_{w_{1}},
\]
where $\delta\in\Omega$ and $w_{m},\ldots,w_{1}\in\tilde{W}$, such
that if $\ell(w_{i+1}s_{j})=\ell(w_{i+1})-1$ for some simple reflection
$s_{j}$ and $1\leq i\leq m-1$ then also $\ell(s_{j}w_{i})=\ell(w_{i})-1$.
In this context, by $\mathrm{DG}_{i}(b)$ we will either mean the
element $w_{i}$ of the group $\tilde{W}$ or the reduced braid $b_{w_{i}}$
in $\mathrm{Br}_{\tilde{W}}^{+}$.
\begin{notation}
We write $b\geq b'$ for $b,b'\in\mathrm{Br}^{+}$ whenever we can
decompose $b=b''b':=b''\circ b'$ for some $b''\in\mathrm{Br}^{+}$.
If this morphism $b''$ is not invertible, then we may also write
$b>b'$. 
\end{notation}

\begin{prop}
\label{prop:dg-properties} Let $\mathrm{Br}^{+}:=\Omega\ltimes\tilde{\mathrm{Br}}^{+}$
be a twisted locally Garside category.

\begin{enumerate}[\normalfont(i)]

\item If $b_{k},\ldots,b_{1}$ are elements of $\tilde{\mathrm{Br}}^{+}$
such that each consecutive pair $b_{i+1}b_{i}$ is in Deligne-Garside
normal form, then so is their product $b_{k}\cdots b_{1}$.

\item Assume that $\mathrm{Br}^{+}$ is small and right-Noetherian
and let $b,b'$ be two morphisms such that $b\geq b'$. Then $\ell(b)\geq\ell(b')$,
with inequality if and only if $b>b'$.

\item If $\tilde{W}$ is a Coxeter group and $v,w\in\tilde{W}$ are
elements such that $w\geq v$, then the product $b_{v}b_{w^{-1}}$
is in Deligne-Garside normal form.

\end{enumerate}
\end{prop}

For convenience, the head function $\mathrm{DG}(\cdot):=\mathrm{DG}_{1}(\cdot)$
is always assumed to be sharp \cite[Definition IV.1.42]{MR3362691};
this means that
\begin{lem}
[{\cite[Proposition IV.1.50]{MR3362691}}] For any pair of composable
morphisms $b,b'$ in a twisted locally Garside category, we have
\[
\mathrm{DG}(bb')=\mathrm{DG}\big(\mathrm{DG}(b)b'\big).
\]
\end{lem}

It can also be used to prove certain stability properties:
\begin{cor}
\label{cor:dg-prop} Moreover, for any $i\geq1$ we have

\begin{enumerate}[\normalfont(i)]

\item $\mathrm{DG}_{i\geq}(bb')=\mathrm{DG}_{i\geq}\bigl(\mathrm{DG}_{i\geq}(b)b'\bigr)$,
so in particular

\item if $b$ is an endomorphism then $\mathrm{DG}_{i\geq}(b^{d-1})b\geq\mathrm{DG}_{i\geq}(b^{d})$
for any $d\geq1$.

\end{enumerate}
\end{cor}

\begin{proof}
(i): We induct on $i$. If it's true $<i$, then from the case $i-1$
we deduce
\begin{align*}
\mathrm{DG}_{>i-1}(bb')\mathrm{DG}_{i-1\geq}(bb') & =bb'\\
 & =\mathrm{DG}_{>i-1}(b)\mathrm{DG}_{i-1\geq}(b)b'\\
 & =\mathrm{DG}_{>i-1}(b)\mathrm{DG}_{>i-1}\bigl(\mathrm{DG}_{i-1\geq}(b)b'\bigr)\mathrm{DG}_{i-1\geq}\bigl(\mathrm{DG}_{i-1\geq}(b)b'\bigr)\\
 & =\mathrm{DG}_{>i-1}(b)\mathrm{DG}_{>i-1}\bigl(\mathrm{DG}_{i-1\geq}(b)b'\bigr)\mathrm{DG}_{i-1\geq}(bb')
\end{align*}
 which implies that
\[
\mathrm{DG}_{>i-1}(bb')=\mathrm{DG}_{>i-1}(b)\mathrm{DG}_{>i-1}\bigl(\mathrm{DG}_{i-1\geq}(b)b'\bigr).
\]
Similarly, the case $i-1$ yields
\begin{align*}
\mathrm{DG}_{>i-1}\bigl(\mathrm{DG}_{i\geq}(b)b'\bigr)\mathrm{DG}_{i-1\geq}\bigl(\mathrm{DG}_{i\geq}(b)b'\bigr) & =\mathrm{DG}_{i\geq}(b)b'\\
 & =\mathrm{DG}_{i}(b)\mathrm{DG}_{i-1\geq}(b)b'\\
 & =\mathrm{DG}_{i}(b)\mathrm{DG}_{>i-1}\bigl(\mathrm{DG}_{i-1\geq}(b)b'\bigr)\mathrm{DG}_{i-1\geq}\bigl(\mathrm{DG}_{i-1\geq}(b)b'\bigr)\\
 & =\mathrm{DG}_{i}(b)\mathrm{DG}_{>i-1}\bigl(\mathrm{DG}_{i-1\geq}(b)b'\bigr)\mathrm{DG}_{i-1\geq}\bigl(\mathrm{DG}_{i\geq}(b)b'\bigr)
\end{align*}
 and thus
\[
\mathrm{DG}_{>i-1}\bigl(\mathrm{DG}_{i\geq}(b)b'\bigr)=\mathrm{DG}_{i}(b)\mathrm{DG}_{>i-1}\bigl(\mathrm{DG}_{i-1\geq}(b)b'\bigr).
\]
 As $\mathrm{DG}_{i}(b)=\mathrm{DG}\bigl(\mathrm{DG}_{>i-1}(b)\bigr)$,
we combine these two identies with the lemma to obtain
\begin{align*}
\mathrm{DG}_{i}(bb') & =\mathrm{DG}(\mathrm{DG}_{>i-1}(bb'))\\
 & =\mathrm{DG}\Bigl(\mathrm{DG}_{>i-1}(b)\mathrm{DG}_{>i-1}\bigl(\mathrm{DG}_{i-1\geq}(b)b'\bigr)\Bigr)\\
 & =\mathrm{DG}\Bigl(\mathrm{DG}\bigl(\mathrm{DG}_{>i-1}(b)\bigr)\mathrm{DG}_{>i-1}\bigl(\mathrm{DG}_{i-1\geq}(b)b'\bigr)\Bigr)\\
 & =\mathrm{DG}\Bigl(\mathrm{DG}_{i}(b)\mathrm{DG}_{>i-1}\bigl(\mathrm{DG}_{i-1\geq}(b)b'\bigr)\Bigr)\\
 & =\mathrm{DG}\Bigl(\mathrm{DG}_{>i-1}\bigl(\mathrm{DG}_{i\geq}(b)b'\bigr)\Bigr)=\mathrm{DG}_{i}\bigl(\mathrm{DG}_{i\geq}(b)b'\bigr).
\end{align*}
We now combine this identity with the induction hypothesis to find
\begin{align*}
\mathrm{DG}_{i\geq}(bb') & =\mathrm{DG}_{i}(bb')\mathrm{DG}_{i-1\geq}(bb')\\
 & =\mathrm{DG}_{i}(bb')\mathrm{DG}_{i-1\geq}\bigl(\mathrm{DG}_{i-1\geq}(b)b'\bigr)\\
 & =\mathrm{DG}_{i}(bb')\mathrm{DG}_{i-1\geq}\Bigl(\mathrm{DG}_{i-1\geq}\big(\mathrm{DG}_{i\geq}(b)\bigr)b'\Bigr)\\
 & =\mathrm{DG}_{i}\bigl(\mathrm{DG}_{i\geq}(b)b'\bigr)\mathrm{DG}_{i-1\geq}\bigl(\mathrm{DG}_{i\geq}(b)b'\bigr)=\mathrm{DG}_{i\geq}\bigl(\mathrm{DG}_{i\geq}(b)b'\bigr).\qedhere
\end{align*}
(ii): From (i) we obtain 
\[
\mathrm{DG}_{i\geq}(b^{d-1})b\geq\mathrm{DG}_{i\geq}\bigl(\mathrm{DG}_{i\geq}(b^{d-1})b\bigr)=\mathrm{DG}_{i\geq}(b^{d}).\qedhere
\]
\end{proof}
\begin{cor}
\label{cor:stabilise} Let $b$ be an endomorphism of a twisted locally
Garside category and let $i,d\geq1$ be integers such that
\[
\mathrm{DG}_{i\geq}(b^{d+1})=\mathrm{DG}_{i\geq}(b^{d}),
\]
then for any integer $d'\geq d$ we have
\[
\mathrm{DG}_{i\geq}(b^{d'})=\mathrm{DG}_{i\geq}(b^{d}).
\]
\end{cor}

\begin{proof}
By induction on $d'>d+1$. If the claim holds for all integers between
$d$ and $d'$, then (i) yields
\[
\mathrm{DG}_{i\geq}(b^{d'})=\mathrm{DG}_{i\geq}\bigl(\mathrm{DG}_{i\geq}(b^{d'-1})b\bigr)=\mathrm{DG}_{i\geq}\bigl(\mathrm{DG}_{i\geq}(b^{d'-2})b\bigr)=\mathrm{DG}_{i\geq}(b^{d'-1})=\mathrm{DG}_{i\geq}(b^{d}).\qedhere
\]
\end{proof}
\begin{example}
Consider $w=s_{2}s_{1}s_{4}s_{3}s_{2}s_{1}$ in type $\mathsf{A}_{4}$.
Then
\[
\ell\big(\mathrm{DG}(b_{w})\big)=\ell(w)=6,\qquad\ell\big(\mathrm{DG}(b_{w}^{2})\big)=\ell(s_{1}w)=7,\qquad\ell\big(\mathrm{DG}(b_{w}^{\geq3})\big)=\ell(w_{\circ})=10.
\]
\end{example}

\begin{example}
Consider $w=w_{\circ}s_{1}$ in type $\mathsf{A}_{3}$. By induction
on $i\geq0$ one finds\begin{alignat*}{3}
& \mathrm{DGN}(b_{w}^{4i}) && =  b_{2312}^{2i}b_{w_{\circ}}^{2i}, \\ 
& \mathrm{DGN}(b_{w}^{4i+1}) && =  b_{2312}^{2i}b_{w}^{\,}b_{w_{\circ}}^{2i}, \\ 
& \mathrm{DGN}(b_{w}^{4i+2}) && =  b_{2312}^{2i+1}b_{w_{\circ}}^{2i+1}, \\ 
& \mathrm{DGN}(b_{w}^{4i+3}) && = b_{2312}^{2i+1}b_{w^{-1}}^{\,}b_{w_{\circ}}^{2i+1}.
\end{alignat*} 
\end{example}

\begin{cor}
\label{cor:dg-involution} Let $\mathrm{Br}^{+}:=\Omega\ltimes\tilde{\mathrm{Br}}^{+}$
be a twisted locally Garside category and let $b=\delta\tilde{b}$
be an endomorphism in $\mathrm{Br}^{+}$ such that $\mathrm{DG}(b^{2})=\tilde{b}$;
in other words, $b$ is reduced and $\delta^{-1}(\tilde{b})\tilde{b}$
is in Deligne-Garside normal form.

\begin{enumerate}[\normalfont(i)]

\item For any integer $d\geq0$ we have 
\[
\mathrm{DGN}(b^{d})=\delta^{-d}\delta^{1-d}(\tilde{b})\cdots\delta^{-1}(\tilde{b})\tilde{b},
\]

\item and for any other composable element $b'$ in $B^{+}$ we have
\[
b^{i}b'\geq\mathrm{DG}_{i\geq}(b^{d}b').
\]

\end{enumerate}
\end{cor}

\begin{proof}
(i): Since $\delta^{-1}(\tilde{b})\tilde{b}$ is in normal form, so
is $\delta^{-j}(\tilde{b})\delta^{1-j}\tilde{b}$ for any integer
$j$.

(ii): From (i) and Corollary \ref{cor:dg-prop}(i) we find
\[
b^{i}b'\geq\mathrm{DG}_{i\geq}(b^{i}b')=\mathrm{DG}_{i\geq}\big(\mathrm{DG}_{i\geq}(b^{i})b'\big)=\mathrm{DG}_{i\geq}\big(\mathrm{DG}_{i\geq}(b^{d})b'\big)=\mathrm{DG}_{i\geq}\big(b^{d}b'\big).\qedhere
\]
\end{proof}
\begin{example}
If $w=\delta\tilde{w}$ is an element of a twisted Coxeter group $\Omega\ltimes\tilde{W}$
such that $w^{2}\in\Omega$ (e.g., $w$ is an involution), then $\delta^{-1}(\tilde{w})=\tilde{w}^{-1}$
so part (iii) of Proposition \ref{prop:dg-properties} implies that
the condition of the previous corollary holds for $b_{w}$. Hence
we obtain
\[
\mathrm{DGN}(b_{w}^{d})=\delta^{-d}b_{\delta^{1-d}(\tilde{w})}^{\,}\cdots b_{\delta^{-1}(\tilde{w})}^{\,}b_{\tilde{w}}^{\,}.
\]
\end{example}

\subsection{\label{subsec:inversion-seq}Root inversion sets and sequences}

In this subsection we extend the notion of root inversion sets and
sequences to elements of twisted braid monoids, and subsequently prove
Proposition \ref{prop:dg-bound}. At the end we also prove part (i)
of Proposition \ref{prop:mixed-shift}.
\begin{notation}
We will always denote a sequence of roots with an underscore, and
its underlying set is then denoted in the same way but omitting the
underscore.
\end{notation}

\begin{defn}
Let $W$ be a twisted Coxeter group and let $b$ be an element of
its braid monoid $\mathrm{Br}^{+}$. For any sequence of roots $\underline{\mathfrak{N}}=(\dots,\beta_{2},\beta_{1})$
we obtain new one by setting
\[
b(\underline{\mathfrak{N}}):=w_{b}(\underline{\mathfrak{N}}):=\bigl(\ldots,w_{b}(\beta_{2}),w_{b}(\beta_{1})\bigr),\qquad b^{-1}(\underline{\mathfrak{N}}):=w_{b}^{-1}(\underline{\mathfrak{N}}),
\]
 and similarly for a set $\mathfrak{N}$. In particular, for each
element $w$ in $W$ we have $b_{w}(\mathfrak{N})=w(\mathfrak{N})$
and $b_{w}^{-1}(\mathfrak{N})=w^{-1}(\mathfrak{N})$.

We define a \emph{simple decomposition of $b$} to be a decomposition
of the form
\[
b=\delta b_{i_{l}}\cdots b_{i_{1}},
\]
and given one we consider the \emph{(right) (root) inversion sequence}
\begin{equation}
\underline{\mathfrak{R}}_{b}:=\bigl(s_{i_{1}}\cdots s_{i_{l-1}}(\alpha_{i_{l}}),\ldots,s_{i_{1}}(\alpha_{i_{2}}),\alpha_{i_{1}}\bigr),\label{eq:inversion-sequence}
\end{equation}
and call the underlying set $\mathfrak{R}_{b}$ its \emph{(right)
(root) inversion set}.\footnote{Independently, a similar construction appears in \cite[§2.9.3]{stump2015cataland}.}
\end{defn}

Unlike $\underline{\mathfrak{R}}_{b}$, the roots in $\underline{\mathfrak{R}}_{b_{w}}$
are distinct and positive:
\begin{prop}
\label{prop:grp-rw} Let $W$ be a twisted Coxeter group and pick
an element $w$ in $W$.

\begin{enumerate}[\normalfont(i)]

\item There is an identity of sets of roots
\[
\mathfrak{R}_{b_{w}}=\mathfrak{R}_{w}:=w^{-1}(-\mathfrak{R}_{+})\cap\mathfrak{R}_{+}.
\]

\item For a pair of elements $x,y$ in $W$ we have
\begin{equation}
\mathfrak{R}_{xy}=\{\beta\in y^{-1}(\mathfrak{R}_{x})\cup\mathfrak{R}_{y}:-\beta\notin y^{-1}(\mathfrak{R}_{x})\cup\mathfrak{R}_{y}\}.\label{eq:inversion-set-of-product}
\end{equation}
In particular,
\begin{align}
 & \ell(xy)=\ell(x)-\ell(y)\quad\textrm{ if and only if }\quad\mathfrak{R}_{y}\subseteq\mathfrak{R}_{x^{-1}},\label{eq:left-strong-conjugation-inversion-set}\\
 & \ell(xy)=\ell(y)-\ell(x)\quad\textrm{ if and only if }\quad\mathfrak{R}_{x}\subseteq\mathfrak{R}_{y^{-1}},\label{eq:right-strong-conjugation-inversion-set}\\
 & \ell(xy)=\ell(x)+\ell(y)\quad\textrm{ if and only if }\quad\mathfrak{R}_{y}\subseteq\mathfrak{R}_{xy}.\label{eq:inversion-set-additive}
\end{align}

\item We have $w(\mathfrak{R}_{w})=-\mathfrak{R}_{w}$ if and only
if $w^{2}\in\Omega$.

\end{enumerate}
\end{prop}

\begin{proof}
(i): Since twists do not affect these sets (after moving them to the
left), this follows from \cite[Corollaire VI.17.2]{MR0240238}.

(iii): The implication $\Rightarrow$ follows from (\ref{eq:inversion-set-of-product}),
as that formula yields $\mathfrak{R}_{w^{2}}=\varnothing$ but then
$w^{2}\in\Omega$. For $\Leftarrow$, since the cardinalities of these
sets agree, it suffices to show $\subseteq:$ if $\beta$ lies in
$\mathfrak{R}_{w}$, then $-w(\beta)$ is positive and $-ww(\beta)=-\delta(\beta)$
is negative, so $-w(\beta)$ lies in $\mathfrak{R}_{w}$.
\end{proof}
\begin{defn}
\label{def:convex-ordering} Let $\mathfrak{R}$ be a root system
and pick a subset of roots $\mathfrak{N}=\{\beta_{l},\ldots,\beta_{1}\}$.
We will say that it is \emph{convex} if for any pair of roots $\beta_{i},\beta_{j}$
in $\mathfrak{N}$ such that $c_{i}\beta_{i}+c_{j}\beta_{j}=\beta_{k}$
for some $c_{i},c_{j}\in\mathbb{R}_{>0}$ is also a root, then $\beta_{k}$
lies in $\mathfrak{N}$. 

We will say that an ordering
\[
\beta_{l}>\cdots>\beta_{1}
\]
of the roots in $\mathfrak{N}$ is \emph{convex} if for such a triple
of positive roots $\beta_{i},\beta_{j},\beta_{k}$ with $i>j$ we
have $\beta_{i}>\beta_{k}>\beta_{j}$.
\end{defn}

\begin{prop}
\begin{enumerate}[\normalfont(i)]

\item Let $w$ be an element of a twisted Coxeter group $W$. Then
there is a map
\[
\{\textrm{reduced decompositions of }w\}\longrightarrow\{\textrm{convex orderings on its inversion set }\mathfrak{R}_{w}\}
\]
sending 
\[
\delta s_{i_{m}}\cdots s_{i_{1}}\longmapsto\bigl(s_{i_{1}}\cdots s_{i_{m-1}}(\alpha_{i_{m}}),\ldots,s_{i_{1}}(\alpha_{i_{2}}),\alpha_{i_{1}}\bigr).
\]
When $W$ is finite and $w$ has maximal length, this is a bijection.

\item Let $\mathfrak{R}$ be a rank 2 root system. Then there are
only two convex orderings on the set of positive roots, and they differ
from each other by inverting; these correspond to the two reduced
decompositions of the corresponding longest element.

\end{enumerate}
\end{prop}

\begin{proof}
(i): This is essentially due (independently) to Zhelobenko \cite[§3]{MR911771},
Dyer \cite[Proposition 2.13]{MR1248893} and Papi \cite[p.\,663]{MR1169886}.

(ii): This was observed in \cite[§4]{MR911771} and \cite[p.\,664]{MR1169886}
in case-by-case calculations for rank 2 Weyl groups. For a case-free
proof: any convex ordering needs to start and end with a simple root,
since all of the other positive roots lie inside of their convex cone.
There are only two such roots in the rank 2 case, and by convexity
each choice yields a unique ordering, which are the same after inverting
one of them.
\end{proof}
\begin{lem}
\label{lem:identification-rw} Let $W$ be a twisted Coxeter group
and let $b$ be an element of its braid monoid $\mathrm{Br}^{+}$. 

\begin{enumerate}[\normalfont(i)]

\item Let $b'$ be another element of the braid monoid $\mathrm{Br}^{+}$.
By choosing simple decompositions of $b'$ and $b$, their product
yields one for $b'b$ (after moving twists) and thus a concatenation
of sequences
\[
\underline{\mathfrak{R}}_{b'b}=\bigl(b^{-1}(\underline{\mathfrak{R}}_{b'}),\underline{\mathfrak{R}}_{b}\bigr).
\]

\item Given a simple decomposition of $b$ with corresponding root
ordering $\underline{\mathfrak{R}}_{b}$, applying a braid move to
the decomposition corresponds to inverting the corresponding section
of roots. Thus whilst the ordering of the sequence $\underline{\mathfrak{R}}_{b}$
depends on the choice of simple decomposition of the braid $b$, the
underlying set $\mathfrak{R}_{b}$ does not.

In particular, for any other pair of elements $b',b''$ in $\mathrm{Br}^{+}$
we have an inclusion of sets
\[
\mathfrak{R}_{b'}\subseteq b(\mathfrak{R}_{b''b'b}).
\]

\item Furthermore, we have $b=b_{w}$ for some $w$ in $W$ if and
only if the set $\mathfrak{R}_{b}$ consists entirely of positive
roots.

\item Suppose we can decompose $b=b_{x}b_{y}$ for some $x$ and
$y$ in $W$. Then for any choice of decomposition of $b$, all roots
in the sequence $\underline{\mathfrak{R}}_{b}$ are distinct.

In particular, if we can decompose this same braid as $b=b'b_{x}$
for some $x$ in $W$ and $b'$ in $\mathrm{Br}^{+}$ then
\[
\mathfrak{R}_{b'}=x(\mathfrak{R}_{b'}\backslash\mathfrak{R}_{x}).
\]

\end{enumerate}
\end{lem}

\begin{proof}
(i): This follows from the explicit construction (\ref{eq:inversion-sequence}).

(ii): Let $\underline{b}'',\underline{b}',\underline{b}$ be sequences
of simple elements of $\mathrm{Br}^{+}$, where $\underline{b}'=(\cdots,b_{j},b_{i})$
is one ``side'' of a braid relation, so by part (i) there is the
concatenation
\begin{align*}
\underline{\mathfrak{R}}_{\underline{b}''\underline{b}'\underline{b}} & =\bigl((b'b)^{-1}(\underline{\mathfrak{R}}_{\underline{b}''}),b^{-1}(\underline{\mathfrak{R}}_{\underline{b}'}),\underline{\mathfrak{R}}_{b}\bigr).
\end{align*}
The product of the elements in $\underline{b}'$ equals the lift of
the Coxeter group element $w=\cdots s_{j}s_{i}$, which is the longest
element of a rank 2 root system. Proposition \ref{prop:grp-rw}(i)
and part (ii) of the previous proposition then imply that $\underline{\mathfrak{R}}_{\cdots b_{j}b_{i}}$
obtained from the other side of the braid relation is the inversion
of $\underline{\mathfrak{R}}_{\cdots b_{i}b_{j}}$, and then the same
holds for $b^{-1}(\underline{\mathfrak{R}}_{\cdots b_{j}b_{i}})$
and $b^{-1}(\underline{\mathfrak{R}}_{\cdots b_{i}b_{j}})$.

(iii): The implication $\Rightarrow$ follows from Proposition \ref{prop:grp-rw}(i).
For $\Leftarrow$, choose a simple decomposition $b=\delta b_{i_{l}}\cdots b_{i_{1}}$.
Inducting on $l$ and using $\mathfrak{R}_{b}\subseteq\mathfrak{R}_{+}$,
one can deduce from say (\ref{eq:inversion-set-additive}) that the
expression
\[
w:=\delta s_{i_{l}}\cdots s_{i_{1}}
\]
 is reduced, so that
\[
b_{w}=\delta b_{i_{l}}\cdots b_{i_{1}}=b.
\]

(iv): Choosing decompositions for $x$ and $y$ yields the root sequence
\[
\underline{\mathfrak{R}}_{b}=\bigl(y^{-1}(\underline{\mathfrak{R}}_{x}),\underline{\mathfrak{R}}_{y}\bigr),
\]
but if a root $\beta$ lies in the intersection $y^{-1}(\mathfrak{R}_{x})\cap\mathfrak{R}_{y}$
then $y(\beta)$ lies in $\mathfrak{R}_{x}\cap-\mathfrak{R}_{+}=\varnothing$.
\end{proof}
\begin{lem}
\label{lem:bounding} Let $\mathrm{Br}^{+}$ be a small, right-Noetherian
locally Garside category and pick an endomorphism $b$ in $\mathrm{Br}^{+}$.

\begin{enumerate}[\normalfont(i)]

\item Suppose that for some $i\geq1$ and each $1\leq j\leq i$,
certain integers $d_{j}$ are given the property that $\ell\bigl(\mathrm{DG}_{j}(b^{d'})\bigr)\leq d_{j}$
(e.g.\ $d_{j}$ is the length of a Garside map), and set
\[
d:=i-\ell\bigl(\mathrm{DG}_{i\geq}(b^{i})\bigr)+\sum_{j=1}^{i}d_{j}
\]
 then $\mathrm{DG}_{i\geq}(b^{d})$ has ``stabilised'', i.e.\ for
any $d'\geq d$ we have
\[
\mathrm{DG}_{i\geq}(b^{d'})=\mathrm{DG}_{i\geq}(b^{d}).
\]

\item If $\mathrm{Br}^{+}$ is the braid monoid of a twisted Coxeter
group, then there is an inclusion
\[
\mathfrak{R}_{\mathrm{DG}_{i}(b)}\subseteq\mathrm{DG}_{i-1}(b)\cdots\mathrm{DG}_{1}(b)(\mathfrak{R}_{b})\cap\mathfrak{R}_{+}.
\]

\end{enumerate} 
\end{lem}

\begin{proof}
(i): By the pigeonhole principle, for some $0\leq k\leq d$ we must
have 
\[
\ell\bigl(\mathrm{DG}_{i\geq}(b^{i+k+1})\bigr)=\ell\bigl(\mathrm{DG}_{i\geq}(b^{i+k})\bigr),
\]
 and then the claim follows from Proposition \ref{prop:dg-properties}(ii)
and Corollary \ref{cor:stabilise}.

(ii): Follows by induction from Lemma \ref{lem:identification-rw}(iii).
\end{proof}

\begin{proof}
[Proof of Proposition \ref{prop:dg-bound}] Pick a reduced decomposition
for $w$. By induction on $d\geq1$, Lemma \ref{lem:identification-rw}(iv)
furnishes an identity of concatenated sequences
\[
\underline{\mathfrak{R}}_{b_{w}^{d}}=\big(w^{-1}(\underline{\mathfrak{R}}_{b_{w}^{d-1}}),\underline{\mathfrak{R}}_{w}\big)=\big(w^{1-d}(\underline{\mathfrak{R}}_{w}),\ldots,w^{-2}(\underline{\mathfrak{R}}_{w}),w^{-1}(\underline{\mathfrak{R}}_{w}),\underline{\mathfrak{R}}_{w}\big).
\]
As we have an inclusion of sets
\[
w^{-k}(\mathfrak{R}_{w})\subseteq w^{-k}(\mathfrak{R}\backslash\mathfrak{R}_{\mathrm{st}}^{w})=\mathfrak{R}\backslash\mathfrak{R}_{\mathrm{st}}^{w}
\]
 for any integer $k\in\mathbb{Z}$, we now also have an inclusion
$\mathfrak{R}_{b_{w}^{d}}\subseteq\mathfrak{R}\backslash\mathfrak{R}_{\mathrm{st}}^{w}$.

The final claim in (i) then also follows from Lemma \ref{lem:identification-rw}(iv),
as linearity yields an equality of cardinalities
\[
|x(\mathfrak{R}\backslash\mathfrak{R}_{\mathrm{st}}^{w})\cap\mathfrak{R}_{+}|=|\mathfrak{R}_{+}\backslash\mathfrak{R}_{\mathrm{st}}^{w}|
\]
 for any element $x$ in $W$, and the rest of this proposition follows
from the previous lemma as 
\[
\ell\bigl(\mathrm{DG}_{i\geq}(b_{w}^{i})\bigr)=\ell(b_{w}^{i})=i\,\ell(w).\qedhere
\]
\end{proof}
\begin{lem}
Let $\beta$ be a positive root in a (not necesessarily spherical
or finite) root system. Then there exists a sequence $(c_{1}\alpha_{i_{1}},c_{2}\alpha_{i_{2}},\ldots,c_{m}\alpha_{i_{m}})$
with each scalar $c_{i}\in\mathbb{R}_{>0}$ so that the partial sums
\begin{equation}
c_{i_{1}}\alpha_{i_{1}},\qquad c_{i_{1}}\alpha_{i_{1}}+c_{i_{2}}\alpha_{i_{2}},\qquad\ldots,\qquad\sum_{j=1}^{m}c_{j}\alpha_{i_{j}}=\beta\label{eq:partial-sums}
\end{equation}
are all positive roots.
\end{lem}

\begin{proof}
We may assume that $\beta$ is not simple and follow the proof of
\cite[Proposition VI.1.19]{MR0240238}: If $(\beta,\alpha_{i})\leq0$
for each $\alpha_{i}$, then by linearity also $(\beta,\beta)\leq0$
which is a contradiction. Thus we have $(\beta,\alpha_{i})>0$ for
some $i$, which implies that $\beta-c\alpha_{i}$ is a root for some
$c_{i}\in\mathbb{R}_{>0}$. As $\beta$ is not simple this root still
has some positive coefficients when decomposed into simple roots,
and hence it must be positive.
\end{proof}
Papi showed that a subset of positive roots $\mathfrak{N}\subseteq\mathfrak{R}_{+}$
is of the form $\mathfrak{R}_{w}$ if and only if both $\mathfrak{N}$
and $\mathfrak{R}_{+}\backslash\mathfrak{N}$ are convex \cite{MR1169886}.
\begin{prop}
\label{prop:closed-complement} Let $\mathfrak{R}_{+}$ be a system
of positive roots of a root system $\mathfrak{R}$. A subset of roots
$\mathfrak{L}\subseteq\mathfrak{R}$ is a standard parabolic subsystem
if and only if it is invariant under multiplication by $-1$ and both
$\mathfrak{L}$ and the positive complement $\mathfrak{R}_{+}\backslash\mathfrak{L}$
are convex.

Hence, for any automorpism $w$ of $\mathfrak{R}$ we have
\[
w\textrm{ is convex}\qquad\text{iff}\qquad\mathfrak{R}_{+}\backslash\mathfrak{R}_{\mathrm{st}}^{w}\textrm{ is convex}\qquad\textrm{iff}\qquad\mathfrak{R}_{+}\backslash\mathfrak{R}_{\mathrm{st}}^{w}=\mathfrak{R}_{v}\textrm{ for some }v\textrm{ in }W,
\]
where $W$ denotes the associated Coxeter group. In particular, the
equation
\[
\mathfrak{R}_{\mathrm{DG}(b_{w}^{d})}=\mathfrak{R}_{+}\backslash\mathfrak{R}_{\mathrm{st}}^{w}\qquad(\textrm{resp.}\quad\mathfrak{R}_{\mathrm{DG}(b_{w}^{d})}=\mathfrak{R}_{+}\backslash\mathfrak{R}^{w})
\]
 generalising the braid equation (\ref{eq:braid-equation}) to arbitrary
elements, always fail when $w$ is not convex (resp.\ firmly convex).
\end{prop}

\begin{proof}
Suppose that $\mathfrak{L}$ is not a parabolic subsystem, then as
it is convex under multiplication by $-1$ it follows that for some
positive root $\beta$ in $\mathfrak{L}$ there exists a simple root
$\alpha_{n}$ in its support which is not in $\mathfrak{L}$. The
previous lemma yields a sequence $(c_{1}\alpha_{i_{1}},c_{2}\alpha_{i_{2}},\ldots,c_{m}\alpha_{i_{m}})$
so that the partial sums (\ref{eq:partial-sums}) are all roots. Consider
those $k$ such that at least one of $\sum_{j=1}^{k-1}c_{j}\alpha_{i_{j}}$
or $\alpha_{i_{k}}$ is not in $\mathfrak{L}$; as $\alpha_{n}$ occurs
in the sequence, this set is nonempty. As the final root $\beta$
itself lies in $\mathfrak{L}$, downward induction and finiteness
yields that for at least one such pair of roots $(\beta_{t-1},\beta_{t})$
their sum $\beta_{t-1}+c_{t}\beta_{t}$ does lie in $\mathfrak{L}$.
But as $\mathfrak{R}_{+}\backslash\mathfrak{L}$ is convex it is not
possible that both roots of this pair lie in $\mathfrak{R}_{+}\backslash\mathfrak{L}$,
so at least one of them lies in $\mathfrak{L}$, say $\beta_{1}$.
But as $\mathfrak{L}$ is convex and convex under multiplication by
$-1$, the equation
\[
c_{t}\beta_{t}=(\beta_{t-1}+c_{t}\beta_{t})+(-\beta_{t-1})
\]
 implies that $\beta_{t}$ also lies in $\mathfrak{L}$, yielding
a contradiction.

For the final claim, note that if $\mathfrak{R}^{w}\neq\mathfrak{R}_{\mathrm{st}}^{w}$,
then the previous proposition implies that the second equation fails.
\end{proof}
The first sentence of part (i) of Proposition \ref{prop:mixed-shift}
follows from
\begin{lem}
\label{lem:stable-roots-under-shifts} Let $x$ and $w$ be elements
of a twisted Coxeter group and consider the conjugate $w':=xwx^{-1}$.
If this conjugation

\begin{enumerate}[\normalfont(i)]

\item is a cyclic shift then $x(\mathfrak{R}_{\mathrm{st}}^{w}\cap\mathfrak{R}_{+})=\mathfrak{R}_{\mathrm{st}}^{w'}\cap\mathfrak{R}_{+}$,
and if

\item it is a strong conjugation then $\mathfrak{R}_{\mathrm{st}}^{w}\cap\mathfrak{R}_{x}$
is stable under the action of $w$.

\end{enumerate}

Hence in either case we have $x(\mathfrak{R}_{\mathrm{st}}^{w})=\mathfrak{R}_{\mathrm{st}}^{w'}$.
\end{lem}

\begin{proof}
Let $\beta$ be a root in $\mathfrak{R}_{\mathrm{st}}^{w}\cap\mathfrak{R}_{+}$
and assume the intersection of its $w$-orbit with $\mathfrak{R}_{x}$
is empty (resp. is the entire orbit). Then 
\[
(w')^{i}\bigl(x(\beta)\bigr)=xw^{i}(\beta)
\]
 is a positive (resp.\ negative) root for all integers $i$ in $\mathbb{Z}$,
so $\beta$ lies in $\mathfrak{R}_{\mathrm{st}}^{w'}\cap\mathfrak{R}_{+}$
(resp.\ $\mathfrak{R}_{\mathrm{st}}^{w'}\cap\mathfrak{R}_{-}$).

(i): The cyclic shift implies reduced decompositions $w=yx$ and $w'=xy$,
where we set $y:=wx^{-1}$, or similarly $w=x^{-1}y$ and $w'=yx^{-1}$
with $y=xw$. By taking inverses and using $\mathfrak{R}_{\mathrm{st}}^{w^{-1}}=\mathfrak{R}_{\mathrm{st}}^{w}$,
the second case reduces to the first case. Hence it suffices to show
the inclusion $x(\mathfrak{R}_{\mathrm{st}}^{w}\cap\mathfrak{R}_{+})\subseteq\mathfrak{R}_{\mathrm{st}}^{w'}\cap\mathfrak{R}_{+}$.
The reduced decomposition yields 
\[
\mathfrak{R}_{x}\subseteq\mathfrak{R}_{w}\subseteq\mathfrak{R}_{+}\backslash\mathfrak{R}_{\mathrm{st}}^{w},
\]
so the previous paragraph implies the claim.

(ii): By assumption either $xw$ or $wx^{-1}$ is reduced. In the
former case, we have $\ell(xwx^{-1})=\ell(xw)-\ell(x^{-1})$ which
by (\ref{eq:left-strong-conjugation-inversion-set}) implies that
$\mathfrak{R}_{x}\subseteq\mathfrak{R}_{xw}$. In the latter case,
$\ell(xwx^{-1})=\ell(x)-\ell(wx^{-1})$ which by (\ref{eq:right-strong-conjugation-inversion-set})
implies $\mathfrak{R}_{x}\subseteq\mathfrak{R}_{xw^{-1}}$. Now suppose
that $\beta$ lies in $\mathfrak{R}_{\mathrm{st}}^{w}\cap\mathfrak{R}_{x}=\mathfrak{R}_{\mathrm{st}}^{w^{-1}}\cap\mathfrak{R}_{x}$,
so $w(\beta)$ and $w^{-1}(\beta)$ are still positive. In the first
case, we find $xw(\beta)<0$ so $w(\beta)\in\mathfrak{R}_{x}$ and
in the second case we find $xw^{-1}(\beta)<0$ so $w^{-1}(\beta)\in\mathfrak{R}_{x}$.
Since $w$ has finite order either yields an inclusion $w(\mathfrak{R}_{\mathrm{st}}^{w}\cap\mathfrak{R}_{x})\subseteq\mathfrak{R}_{\mathrm{st}}^{w}\cap\mathfrak{R}_{x}$
which must be an equality.

Hence in either case, the conditions of the first paragraph are met,
and the final claim follows.
\end{proof}
For the comment following Proposition \ref{prop:mixed-shift}(i),
recall that the number of fixed roots in a conjugacy class is constant.
\begin{cor}
\label{cor:roots-made-negative-mixed} Let $w$ and $w'$ be two elements
of a twisted finite Coxeter group such that $w\overset{\times}{\sim}w'$.
Then
\[
|\mathfrak{R}_{w}\cap V_{w}|=|\mathfrak{R}_{w'}\cap V_{w'}|.
\]
\end{cor}

\begin{proof}
By assumption there exists an element $\tau$ in the Coxeter group
such that $\tau w\tau^{-1}=w'$, which yields $\tau(V_{w})=V_{w'}$.
By the lemma we have $\tau(\mathfrak{R}_{\mathrm{st}}^{w})=\mathfrak{R}_{\mathrm{st}}^{w'}$,
so this means that $\tau$ bijectively maps the nonstable orbits of
$w$ in $V_{w}$ to nonstable orbits of $w'$ in $V_{w'}$. As each
non-stable orbit contributes 1 to the cardinality $|\mathfrak{R}_{w}|$,
the claim follows.
\end{proof}
For the proof of the main Theorem we will use
\begin{lem}
\label{lem:move-power-bound} Let $w$ be an element of a twisted
finite Coxeter group $W$ and let $w'$ be a cyclic shift, so there
are reduced decompositions $w=yx$ and $w'=xy$ for some elements
$x$ and $y$ in $W$. Assume that $w$ and $w'$ are convex, and
set $x':=\mathrm{pb}(w')\,x\,\mathrm{pb}(w)^{-1}$. Then there are
braid identities

\begin{enumerate}[\normalfont(i)]

\item $b_{\mathrm{pb}(w')}b_{x}=b_{x'}b_{\mathrm{pb}(w)}$ and

\item $b_{\mathrm{pb}(w')^{-1}}b_{x'}=b_{x}b_{\mathrm{pb}(w)^{-1}}$.

\end{enumerate}
\end{lem}

\begin{proof}
(i): From Lemma \ref{lem:stable-roots-under-shifts}(i) it follows
that
\[
\mathrm{pb}(w)\,x^{-1}(\mathfrak{R}_{\mathrm{st}}^{w'}\cap\mathfrak{R}_{+})=\mathrm{pb}(w)(\mathfrak{R}_{\mathrm{st}}^{w}\cap\mathfrak{R}_{+})\subseteq\mathfrak{R}_{+}.
\]
This implies that $\mathfrak{R}_{\mathrm{pb}(w)x^{-1}}\subseteq\mathfrak{R}_{+}\backslash\mathfrak{R}_{\mathrm{st}}^{w'}$,
which means that $\mathrm{pb}(w')\geq\mathrm{pb}(w)\,x^{-1}$ in the
weak left Bruhat-Chevalley order. As $\mathrm{pb}(w)\geq w\geq x$,
that is equivalent to $b_{\mathrm{pb}(w')}b_{x}\geq b_{\mathrm{pb}(w)}$. 

(ii): Note that
\[
\mathfrak{R}_{\mathrm{pb}(w)^{-1}}=\mathfrak{R}_{+}\backslash w_{\circ}(\mathfrak{R}_{\mathrm{st}}^{w}\cap-\mathfrak{R}_{+})
\]
 and similarly for $\mathrm{pb}(w')^{-1}$. Using Lemma \ref{lem:stable-roots-under-shifts}(i)
we deduce that
\[
x'\bigl(w_{\circ}(\mathfrak{R}_{\mathrm{st}}^{w}\cap-\mathfrak{R}_{+})\bigr)=w_{\circ}w'_{\mathrm{st}}x(\mathfrak{R}_{\mathrm{st}}^{w}\cap\mathfrak{R}_{+})=w_{\circ}w'_{\mathrm{st}}(\mathfrak{R}_{\mathrm{st}}^{w'}\cap\mathfrak{R}_{+})=w_{\circ}(\mathfrak{R}_{\mathrm{st}}^{w'}\cap-\mathfrak{R}_{+})\subseteq\mathfrak{R}_{+},
\]
which combines with the previous identity to $\mathrm{pb}(w)^{-1}\geq x'$.
Similarly,
\[
\mathrm{pb}(w)^{-1}(x')^{-1}(w_{\circ}(\mathfrak{R}_{\mathrm{st}}^{w'}\cap-\mathfrak{R}_{+}))=\mathrm{pb}(w)^{-1}w_{\circ}(\mathfrak{R}_{\mathrm{st}}^{w}\cap-\mathfrak{R}_{+})=\mathfrak{R}_{\mathrm{st}}^{w}\cap\mathfrak{R}_{+}\subseteq\mathfrak{R}_{+}
\]
yields $\mathrm{pb}(w')^{-1}\geq\mathrm{pb}(w)^{-1}(x')^{-1}$, and
thus $b_{\mathrm{pb}(w')^{-1}}b_{x'}=b_{x}b_{\mathrm{pb}(w)^{-1}}$. 
\end{proof}

\subsection{Mixed shifts and braid conjugation}

In this subsection we prove part (ii) of Proposition \ref{prop:mixed-shift};
it won't be used elsewhere in this paper.
\begin{defn}
Recall the notions of Definition \ref{def:locally-garside-category}.
When we omit the adjective \emph{locally}, the category is also left-cancellative,
and the Garside family is bounded by a Garside map $b_{w_{\circ}}$
\cite[Proposition V.1.20]{MR3362691}. For every reduced morphism
$b_{w}$ there exists a reduced morphism $\partial(b_{w})$ such that
$b_{w_{\circ}}=\partial(b_{w})b_{w}$ \cite[Proposition IV.1.23]{MR3362691},
which is unique by cancellativity. The square of $\partial(\cdot)$
extends to a unique endofunctor $\partial^{2}(\cdot)$ of $\mathrm{Br}$,
satisfying $b_{w_{\circ}}b=\partial^{2}(b)b_{w_{\circ}}$ for any
morphism $b$ \cite[Corollary V.1.34]{MR3362691}; we will assume
that it is invertible, and then so is $\partial(\cdot)$ \cite[Proposition V.2.17]{MR3362691}.
One then inductively defines for $m\geq0$ the iterates $b_{w_{\circ}}^{m}:\mathrm{ob}(\mathrm{Br})\rightarrow\mathrm{Br}$
by
\[
b_{w_{\circ}}^{0}(\cdot):=\mathrm{id}(\cdot),\qquad b_{w_{\circ}}^{m}(\cdot):=\partial^{2}\bigl(b_{w_{\circ}}^{m-1}(\cdot)\bigr)b_{w_{\circ}}(\cdot),\qquad b_{w_{\circ}}^{-m}(\cdot):=b_{w_{\circ}}^{m}\bigl(\partial^{-2m}(\cdot)\bigr)^{-1},
\]
so that $b_{w_{\circ}}^{m}bb_{w_{\circ}}^{-m}=\partial^{2m}(b)$ for
any $m\in\mathbb{Z}$ \cite[Lemma V.1.41]{MR3362691}.

The enveloping groupoid of a twisted locally Garside category $\mathrm{Br}^{+}$
is isomorphic to the semidirect product $\mathrm{Br}:=\Omega\ltimes\tilde{\mathrm{Br}}$,
where $\tilde{\mathrm{Br}}$ is the enveloping groupoid of $\tilde{\mathrm{Br}}^{+}$;
we call this the \emph{twisted locally Garside groupoid} corresponding
to $\mathrm{Br}^{+}$. A Garside category is Ore \cite[Proposition V.2.32]{MR3362691},
from which it follows that $\mathrm{Br}^{+}$ embeds into $\mathrm{Br}$.
This is a groupoid of left and of right fractions \cite[Proposition II.3.11]{MR3362691}
and it admits right-lcms \cite[Proposition V.2.35]{MR3362691}, which
implies that we can remove common factors in a left fraction. Then
decomposing their complements to this lcm into normal forms, one obtains
the \emph{symmetric normal form} \cite[Corollary III.2.21]{MR3362691}
\begin{equation}
b=\delta b_{i}'\cdots b_{1}'(b_{1}'')^{-1}\cdots(b_{j}'')^{-1},\label{eq:symmetric-normal-form}
\end{equation}
which is essentially unique \cite[Proposition III.2.16]{MR3362691}.
Rewriting each of the negative terms as a positive term times a power
of $b_{w_{\circ}}$ and then moving all of these powers to the right
yields
\begin{equation}
b=b'(b'')^{-1}=\delta b_{i}'\cdots b_{1}'\partial^{-1}(b_{1}'')\cdots\partial^{1-2j}(b_{j}'')b_{w_{\circ}}^{-j},\label{eq:delta-normal-form}
\end{equation}
 which after removing the final factor is the usual normal form of
$bb_{w_{\circ}}^{j}\in\mathrm{Br}^{+}$ \cite[Proposition V.3.35]{MR3362691}.
\end{defn}

\begin{lem}
\label{lem:mixed-shift} Let $\mathrm{Br}^{+}$ be a twisted locally
Garside category, and assume that it embeds into its groupoid $\mathrm{Br}$.

\begin{enumerate}[\normalfont(i)]

\item If $b,b',b''$ are composable morphisms of $\mathrm{Br}^{+}$
such that the morphism $bb'(b'')^{-1}$ of $\mathrm{Br}$ still lies
in $\mathrm{Br}^{+}$, then $\mathrm{DG}(bb')\geq\mathrm{DG}(b'')$.
In particular, if $bb'$ is in right Deligne-Garside normal form and
$b''$ is reduced, then $b'\geq b''$.

\item Now suppose that $b_{w}$ is a reduced endomorphism of $\mathrm{Br}^{+}$,
and $b$ is any composable morphism of $\mathrm{Br}^{+}$ such that
the conjugate $bb_{w}b^{-1}$ also lies in $\mathrm{Br}^{+}$, and
consider the reduced morphism $b_{u}^{\,}:=\mathrm{DG}(bb_{w}^{\,})b_{w}^{-1}$.
Then $b_{u}^{\,}b_{w}^{\,}b_{u}^{-1}$ is also reduced.

\end{enumerate}
\end{lem}

\begin{proof}
(i): Setting $b''':=bb'(b'')^{-1}\in\mathrm{Br}^{+}$, we have an
identity $bb'=b'''b''$ inside $\mathrm{Br}^{+}$, which implies 
\[
\mathrm{DG}(bb')=\mathrm{DG}(b'''b'')\geq\mathrm{DG}(b'').
\]
Under the additional assumptions, we now deduce
\[
b'=\mathrm{DG}(bb')\geq\mathrm{DG}(b'')=b''.
\]

(ii): Since $b=bb_{w}b_{w}^{-1}\geq\mathrm{DG}(bb_{w})b_{w}^{-1}=b_{u}$
we may write $b=b'b_{u}$ for some morphism $b'$ in $\mathrm{Br}^{+}$.
Then from the assumption it follows that 
\[
bb_{w}^{\,}b_{u}^{-1}=bb_{w}^{\,}b^{-1}b'
\]
also lies in $\mathrm{Br}^{+}$, so by (i) we have $b_{u}b_{w}=\mathrm{DG}(bb_{w}^{\,})\geq b_{u}$.
Thus we may split $b_{u}b_{w}=b''b_{u}$ with $b''$ in $\mathrm{Br}^{+}$,
and since $b_{u}b_{w}$ is reduced so is $b''=b_{u}^{\,}b_{w}^{\,}b_{u}^{-1}$.
\end{proof}
\begin{example}
Consider a simple reflection $v:=s_{1}$ and the longest element $w:=w_{\circ}$
in type $\mathsf{A}_{2}$. Then $b_{v}^{\,}b_{w}^{\,}b_{v}^{-1}=b_{1}b_{12}$
lies in $\mathrm{Br}^{+}$, but it is not reduced.
\end{example}

The following ``convexity'' property goes back to Garside's solution
of the conjugacy problem via summit sets \cite[Lemma 12]{MR248801}
and their refinement by super summit sets \cite[Corollary 4.2]{MR1315459};
a similar statement for Garside groups appears in \cite[Lemma 3.2]{MR2747147}:
\begin{lem}
Let $\mathrm{Br}^{+}$ be a twisted Garside category. Suppose $b',b''$
are endomorphisms in its Garside groupoid $\mathrm{Br}$, such that
$b_{w_{\circ}}^{m}\geq b',b''\geq b_{w_{\circ}}^{n}$ for some integers
$m$ and $n$ and $bb'b^{-1}=b''$ for some morphism $b$ in $\mathrm{Br}$.
Let $\mathrm{DG}(b)$ be the right-most term appearing in its symmetric
normal form (\ref{eq:symmetric-normal-form}). Then also
\[
b_{w_{\circ}}^{m}\geq\mathrm{DG}(b)b'\mathrm{DG}(b)^{-1}\geq b_{w_{\circ}}^{n}.
\]
\end{lem}

\begin{proof}
Considering the negative factors in (\ref{eq:symmetric-normal-form}),
set $\tilde{b}:=bb_{w_{\circ}}^{j}\in\mathrm{Br}^{+}$ and $\tilde{b}':=b_{w_{\circ}}^{-j}b'b_{w_{\circ}}^{j}$.
Then $\tilde{b}\in\mathrm{Br}^{+}$ and $\tilde{b}\tilde{b}'(\tilde{b})^{-1}=b''$,
and from $\partial^{2}(b_{w_{\circ}})=\partial(\mathrm{id})=b_{w_{\circ}}$
it follows that $b_{w_{\circ}}^{m}\geq\tilde{b}'\geq b_{w_{\circ}}^{n}$
still holds. By assumption we may write $b''=\tilde{b}''b_{w_{\circ}}^{n}$
with $\tilde{b}''$ in $\mathrm{Br}^{+}$. Then
\begin{align*}
\mathrm{DG}(\tilde{b})\tilde{b}' & \geq\mathrm{DG}(\tilde{b}\tilde{b}')\\
 & =\mathrm{DG}(b''\tilde{b})\\
 & =\mathrm{DG}\bigl(\tilde{b}''\partial^{2n}(\tilde{b})\bigr)b_{w_{\circ}}^{n}\\
 & \geq\mathrm{DG}\bigl(\partial^{2n}(\tilde{b})\bigr)b_{w_{\circ}}^{n}\\
 & =\partial^{2n}\bigl(\mathrm{DG}(\tilde{b})\bigr)b_{w_{\circ}}^{n}=b_{w_{\circ}}^{n}\mathrm{DG}(\tilde{b}),
\end{align*}
which yields the second inequality, for $\tilde{b}$ and $\tilde{b}'$.

Taking inverses, we have $\tilde{b}(\tilde{b}')^{-1}\tilde{b}^{-1}=(b'')^{-1}$
with $(\tilde{b}')^{-1},(b'')^{-1}\geq b_{w_{\circ}}^{-m}$. The previous
paragraph therefore furnishes that
\[
\mathrm{DG}(\tilde{b})(\tilde{b}')^{-1}\mathrm{DG}(\tilde{b})^{-1}\geq b_{w_{\circ}}^{-m}
\]
 so taking inverses again now yields the first inequality for $\tilde{b}$
and $\tilde{b}'$.

But from (\ref{eq:delta-normal-form}) we obtain $b_{w_{\circ}}^{j-1}\mathrm{DG}(\tilde{b})b_{w_{\circ}}^{-j}=\mathrm{DG}(b)$,
so
\[
b_{w_{\circ}}^{j-1}\bigl(\mathrm{DG}(\tilde{b})\tilde{b}'\mathrm{DG}(\tilde{b})^{-1}\bigr)b_{w_{\circ}}^{1-j}=\bigl(b_{w_{\circ}}^{j-1}\mathrm{DG}(\tilde{b})b_{w_{\circ}}^{-j}\bigr)b'\bigl(b_{w_{\circ}}^{j}\mathrm{DG}(\tilde{b})^{-1}b_{w_{\circ}}^{1-j}\bigr)=\mathrm{DG}(b)b'\mathrm{DG}(b)^{-1}
\]
 and then using $\partial^{2}(b_{w_{\circ}})=b_{w_{\circ}}$ again
the claim for $\mathrm{DG}(b)b'\mathrm{DG}(b)^{-1}$ also follows.
\end{proof}
\begin{cor}
If the endomorphisms $b'=b_{w}$ and $b''=bb_{w}b^{-1}$ are both
reduced, then so is
\[
\mathrm{DG}(b)b_{w}\mathrm{DG}(b)^{-1}.
\]
\end{cor}

The following statement was originally employed in a different approach
to proving part (ii) of Proposition \ref{prop:mixed-shift}; it won't
be used but might be of interest in itself.
\begin{prop}
Let $b_{v}=\delta\tilde{b}_{v}$ and $b_{w}$ be reduced endomorphisms
of a twisted Garside category satisfying $\mathrm{DG}(b_{v}^{2})=\tilde{b}_{v}^{\,}$.
If the element 
\[
b_{v}^{k}b_{w}^{\,}b_{v}^{-k}
\]
 is reduced (and thus equal to $b_{v^{k}wv^{-k}}$ in the case of
twisted Coxeter groups) for some integer $k\geq1$, then it is also
reduced for $k=1$.
\end{prop}

\begin{proof}
Follows immediately from the previous corollary and Corollary \ref{cor:dg-involution}(i).
\end{proof}
In the case where $v$ is an involution of a twisted Coxeter group
(which is locally Garside but not in general Garside), we can give
an entirely different proof:
\begin{proof}
[Proof 2] As $v$ is an involution and $b_{v}^{k}b_{w}^{\,}=b_{w'}^{\,}b_{v}^{k}$
for some $w'$ in $W$, Corollary \ref{cor:dg-involution}(ii) implies
that
\[
b_{v}b_{w}\geq\mathrm{DG}(b_{v}^{k}b_{w}^{\,})=\mathrm{DG}(b_{w'}^{\,}b_{v}^{k})\geq b_{v},
\]
so $b_{v}b_{w}=bb_{v}$ for some $b$ in $\mathrm{Br}^{+}$. According
to Lemma \ref{lem:identification-rw}(iv) we have
\[
\mathfrak{R}_{b}=v\Bigl(\bigl(w^{-1}(\mathfrak{R}_{v})\cup\mathfrak{R}_{w}\bigr)\backslash\mathfrak{R}_{v}\Bigr),
\]
and according to Lemma \ref{lem:identification-rw}(ii) it suffices
to show that $\mathfrak{R}_{b}$ consists of positive roots.

We may assume that $k>1$; again using that $v$ is an involution,
Proposition \ref{prop:grp-rw}(ii) and Lemma \ref{lem:identification-rw}(i)
yield
\[
\mathfrak{R}_{b_{v}^{k}}=\mathfrak{R}_{v}\sqcup-\mathfrak{R}_{v}
\]
and therefore an inclusion of sets of roots
\[
w^{-1}(\mathfrak{R}_{v})\cup\mathfrak{R}_{w}\subseteq w^{-1}(\mathfrak{R}_{v})\cup w^{-1}(-\mathfrak{R}_{v})\cup\mathfrak{R}_{w}=\mathfrak{R}_{b_{v}^{k}b_{w}^{\,}}=\mathfrak{R}_{b_{w'}b_{v}^{k}}=v^{k}(\mathfrak{R}_{w'})\cup\mathfrak{R}_{v}\cup-\mathfrak{R}_{v},
\]
and hence
\[
\mathfrak{R}_{b}\subseteq v^{k+1}(\mathfrak{R}_{w'})\cup-v(\mathfrak{R}_{v})=v^{k+1}(\mathfrak{R}_{w'})\cup\mathfrak{R}_{v}.
\]
Thus if a root of $\mathfrak{R}_{b}$ is negative, it lies in $v^{k+1}(\mathfrak{R}_{w'})$.
If $k$ is odd then $v^{k+1}(\mathfrak{R}_{w'})=\mathfrak{R}_{w'}$
consists of positive roots, and if it is even then it lies in $v^{k+1}(\mathfrak{R}_{w'})=v(\mathfrak{R}_{w'})$.
But then
\[
\varnothing\neq v(\mathfrak{R}_{b})\cap\mathfrak{R}_{w'}\subseteq v(\mathfrak{R}_{b})\cap\mathfrak{R}_{v}=\Bigl(\bigl(w^{-1}(\mathfrak{R}_{v})\cup\mathfrak{R}_{w}\bigr)\backslash\mathfrak{R}_{v}\Bigr)\cap\mathfrak{R}_{v}=\varnothing,
\]
which is a contradiction.
\end{proof}
The converse is false:
\begin{example}
Consider $b=b_{2}$ and $b'=b_{12}$ in type $\mathsf{A}_{2}$. Then
$bb'b^{-1}=b_{21}^{\,}$, but $b^{2}b'b^{-2}=b_{2}^{\,}b_{21}^{\,}b_{2}^{-1}$.
\end{example}

And it may fail when $v$ is not an involution:
\begin{example}
Consider $b=b_{231}$ and $b'=b_{2321}$ in type $\mathsf{A}_{3}$.
Then $bb'b^{-1}=b_{23}^{\,}b_{13}^{\,}$, but $b^{2}b'b^{-2}=b_{1213}$.
\end{example}

For twisted finite Coxeter groups, the second proof is slightly weaker:
\begin{example}
Consider $v=s_{3}s_{4}s_{1}s_{2}s_{3}$ in type $\mathsf{A}_{4}$.
Then $\mathrm{DG}(b_{v}^{2})=b_{v}^{\,}$, but $v$ is not an involution.
\end{example}

We now generalise Definition \ref{def:shifts} to the Garside setting:
\begin{notation}
Mirroring $\geq$, we write $b'\leq b$ if there exists a morphism
$b''$ such that $b'\circ b''=b$.
\end{notation}

\begin{defn}
Let $\mathrm{Br}^{+}:=\Omega\ltimes\tilde{\mathrm{Br}}^{+}$ be a
cancellative twisted locally Garside category and pick two reduced
endomorphisms $b_{w}$ and $b_{w'}$ in $\mathrm{Br}^{+}$.

\begin{enumerate}[\normalfont(i)]

\item If there exists a sequence of reduced morphisms $b_{w'}=b_{w_{n+1}},\ldots,b_{w_{0}}=b_{w}$
in $\mathrm{Br}^{+}$ and reduced morphisms $b_{\tau_{n}},\ldots,b_{\tau_{0}}$
in $\tilde{\mathrm{Br}}^{+}$ such that
\[
b_{\tau_{i}}^{\,}b_{w_{i}}^{\,}=b_{w_{i+1}}^{\,}b_{\tau_{i}}^{\,}\textrm{ and }b_{\tau_{i}}^{\,}b_{w_{i}}^{\,}\textrm{ is reduced}
\]
or
\[
b_{w_{i}}^{\,}b_{\tau_{i}}^{\,}=b_{\tau_{i}}^{\,}b_{w_{i+1}}^{\,}\textrm{ and }b_{w_{i}}^{\,}b_{\tau_{i}}^{\,}\textrm{ is reduced}
\]
for each $0\leq i\leq n$, then we denote these \emph{strong conjugations}
by $b_{w}\overset{+}{\sim}b_{w'}$\emph{.}

\item If instead
\[
b_{\tau_{i}}^{\,}b_{w_{i}}^{\,}=b_{w_{i+1}}^{\,}b_{\tau_{i}}^{\,}\textrm{ and }b_{w_{i}}^{\,}\geq b_{\tau_{i}}^{\,}
\]
or
\[
b_{w_{i}}^{\,}b_{\tau_{i}}^{\,}=b_{\tau_{i}}^{\,}b_{w_{i+1}}^{\,}\textrm{ and }b_{\tau_{i}}^{\,}\leq b_{w_{i}}^{\,}
\]
for each $0\leq i\leq n$, then we denote these \emph{cyclic shifts}
by $b_{w}\overset{-}{\sim}b_{w'}$\emph{.}

\item If this sequence is a combination of these, then we denote
these \emph{mixed shifts} by $b_{w}\overset{\times}{\sim}b_{w'}$.

\item Analogously, we define transporters $\mathrm{Tran}_{\mathrm{Br}^{+\prime}}^{\star}(b_{w},b_{w'})$
for parabolic subcategories $\mathrm{Br}^{+\prime}\subseteq\tilde{\mathrm{Br}}^{+}$
\cite[§VII.1.4]{MR3362691} and $\star\in\{+,-,\times\}$.

\end{enumerate}
\end{defn}

Part (ii) of Proposition \ref{prop:mixed-shift} is then a special
case of
\begin{prop}
Let $\mathrm{Br}^{+}$ be a twisted Garside category and pick two
reduced endomorphisms $b_{w}$ and $b_{w'}$. Then for any parabolic
subcategory $\mathrm{Br}^{+\prime}\subseteq\tilde{\mathrm{Br}}^{+}$,
the projection map
\[
\mathrm{Tran}_{\mathrm{Br}^{+\prime}}^{\times}(b_{w},b_{w'})\longrightarrow\mathrm{Tran}_{\mathrm{Br}^{+\prime}}^{\,}(b_{w},b_{w'}):=\{b\in\mathrm{Br}^{+\prime}:bb_{w}=b_{w'}b\}
\]
surjects.
\end{prop}

\begin{proof}
By the previous lemma and proposition, the claim follows inductively
if we prove the statement when $b=b_{v}$ is a reduced morphism $b_{v}$
in $\mathrm{Br}^{+\prime}$. For this we induct on the length of $b_{v}$;
now consider $b_{u}:=\mathrm{DG}(b_{v}^{\,}b_{w}^{\,})b_{w}^{-1}$.
If $b_{u}$ is the identity then $\mathrm{DG}(b_{v}^{\,}b_{w}^{\,})=b_{w}^{\,}$,
so Lemma \ref{lem:mixed-shift}(i) implies that $b_{v}$ is a cyclic
shift. Otherwise, we may decompose $b_{v}=b_{v'}b_{u}$ for a reduced
$b_{v'}$ lying in $\mathrm{Br}^{+\prime}$, as it is parabolic. Then
by Lemma \ref{lem:mixed-shift}(ii) $b_{u}$ is a strong conjugation
and we applying the induction hypothesis to the shorter morphism $b_{v'}$
to finish.
\end{proof}
Employing the induction hypothesis in the final step was necessary:
\begin{example}
Consider $v:=w:=b_{12}$ in type $\mathsf{A}_{2}$. Then $\mathrm{DGN}(b_{v}b_{w})=b_{1}b_{212}$
but 
\[
\mathrm{DGN}(b_{1}^{\,}b_{2}^{\,}b_{w}^{\,}b_{2}^{-1})=\mathrm{DGN}(b_{1}b_{21})=b_{121}.
\]
We won't be using the following statement; in the sequel it is shown
to preserve convexity \cite[Corollary 2.19]{WM-cross}.
\end{example}

\begin{prop}
Let $b$ be an endomorphism in a twisted locally Garside category.
Consecutively conjugating $b$ by $\mathrm{DG}(b^{d+1})\mathrm{DG}(b^{d})^{-1}$
for $d\geq0$, we obtain a sequence
\[
b,\qquad\mathrm{DG}(b^{2})\,b\,\mathrm{DG}(b^{2})^{-1},\qquad\mathrm{DG}(b^{3})\,b\,\mathrm{DG}(b^{3})^{-1},\qquad\ldots,
\]
 of cyclic shifts.
\end{prop}

\begin{proof}
Follows from (ii) of Corollary \ref{cor:dg-prop}.
\end{proof}

\subsection{From braiding sequences to braid powers}

He-Nie proved that for elements $w$ with a decreasing complete sequence
of eigenspaces, the braid $b_{w}^{\mathrm{ord}(w)}$ can be explicitly
decomposed in terms of the longest elements of the standard parabolic
subgroups corresponding to the filtration $F_{i}$ \cite[Theorem 5.3]{MR2999317}.
This subsection begins with generalising this to a statement for elements
with a braiding sequence of eigenspaces; subsequently specialising
it to Sevostyanov's elements, it will imply that they satisfy the
braid equation (\ref{eq:braid-equation}). After that, we prove the
main Theorem.
\begin{notation}
For any element $w$ of a twisted finite Coxeter group, we write $b_{\pm w}:=b_{w^{-1}}b_{w^{\phantom{-1}}}\!\!\!\!\!$.
\end{notation}

The following is well-known, but I believe the usual proof for the
conclusion in (ii) involves embedding $\mathrm{Br}^{+}$ into the
corresponding braid group and invoking (i); this embedding can be
avoided:
\begin{prop}
\label{prop:parabolic-subgroup-centraliser} Let $W$ be a twisted
finite Coxeter group.

\begin{enumerate}[\normalfont(i)]

\item Let $\delta$ be a twist of $\tilde{W}$. Any decomposition
of a maximal length element $\delta w_{\circ}=xy$ into elements $x,y$
in $W$ is a reduced decomposition, e.g.\
\[
(w_{\circ}w^{-1})w=w_{\circ}=(w_{\circ}ww_{\circ})(w_{\circ}w^{-1}).
\]
In particular, $b_{w_{\circ}}^{2}$ is central in $\mathrm{Br}^{+}$.

\item Let $W_{1}$ be a standard parabolic subgroup of $W$ and denote
its longest untwisted element by $w_{1}$. Then for any element $w$
in $W_{1}$ there are reduced decompositions
\[
(w_{\circ}w_{1})w=w_{\circ}w_{1}w=(w_{\circ}w_{1}ww_{1}w_{\circ})(w_{\circ}w_{1})
\]
 and 
\[
(w_{1}w_{0})(w_{\circ}w_{1}ww_{1}w_{\circ})=ww_{1}w_{\circ}=w(w_{1}w_{\circ}).
\]

In particular, the braid $b_{\pm w_{\circ}w_{1}}$ of $\mathrm{Br}^{+}$
centralises the standard parabolic submonoid $\mathrm{Br}_{W_{1}}^{+}\subseteq\mathrm{Br}^{+}$.

\end{enumerate}
\end{prop}

\begin{proof}
(i): The first claim follows from equation (\ref{eq:inversion-set-additive})
and
\[
xy(\mathfrak{R}_{y})=\delta w_{\circ}(\mathfrak{R}_{y})\subseteq-\mathfrak{R}_{+}.
\]
The two given decompositions of $w_{\circ}$ then furnish

\[
b_{w_{\circ}}b_{w}=b_{w_{\circ}ww_{\circ}}b_{w_{\circ}w^{-1}}b_{w}=b_{w_{\circ}ww_{\circ}}b_{w_{\circ}}.
\]

(ii): From $w\in W_{1}$ it follows that $w(\mathfrak{R}_{w})\subseteq-\mathfrak{R}_{w_{1}}$.
The first reduced decomposition $(w_{\circ}w_{1})w$ then follows
from equation (\ref{eq:inversion-set-additive}) and 
\[
w_{\circ}w_{1}w(\mathfrak{R}_{w})\subseteq w_{\circ}w_{1}(-\mathfrak{R}_{w_{1}})=w_{\circ}(\mathfrak{R}_{w_{1}})\subseteq-\mathfrak{R}_{+}.
\]
 The second reduced decomposition of $w_{\circ}w_{1}w$ follows similarly,
using that $w_{1}w\in W_{1}$ implies $w_{1}w(\mathfrak{R}_{w_{\circ}w_{1}})\subseteq\mathfrak{R}_{+}$.
The second claim follows by taking inverses.

Combining these two statements then furnishes
\[
b_{\pm w_{1}w_{\circ}}b_{w}=b_{w_{1}w_{\circ}}b_{w_{\circ}w_{1}}b_{w}=b_{w_{1}w_{\circ}}b_{w_{\circ}w_{1}ww_{1}w_{\circ}}b_{w_{\circ}w_{1}}=b_{w}b_{w_{1}w_{\circ}}b_{w_{\circ}w_{1}}=b_{w}b_{\pm w_{1}w_{\circ}}.\qedhere
\]
\end{proof}
\begin{defn}
\label{def:braiding-sequence-elements}  Let $\underline{\Theta}=(V_{m},\ldots,V_{1})$
be a sequence of eigenspaces of an element of a twisted finite Coxeter
group, which is in good position with respect to the dominant Weyl
chamber. Let $\mathrm{arg}(\Theta):=\{\theta_{m},\ldots,\theta_{1}\}$
denote the set of normalised positive arguments as before. Then we
inductively define for $0\leq i\leq m-1$ the sequence $\underline{\Theta}^{i}$
and nonnegative rational
\begin{equation}
d_{i}':=\mathrm{min}\bigl\{\theta_{j}:\theta_{j}\in\mathrm{arg}(\Theta^{i})\bigr\}-d_{i-1}',\label{eq:eigensequence-integers}
\end{equation}
where $d_{-1}':=0$ and $d_{i}'=0$ if the sequence $\underline{\Theta}^{i}$
is empty, $\underline{\Theta}^{0}:=\underline{\Theta}$ and the sequence
$\underline{\Theta}^{i+1}$ is inductively obtained from $\underline{\Theta}^{i}$
by removing those eigenspaces $V_{k}$ with eigenvalue $\exp\bigl(2\pi i\,\mathrm{min}\{\theta_{j}:\theta_{j}\in\mathrm{arg}(\Theta^{i})\}\bigr)$.
Furthermore for $0\leq j\leq m-1$ set
\begin{equation}
w_{j}':=w_{j}w_{j+1},\qquad\vartheta_{i}:=w_{0}'\cdots w_{k-1}'\hat{w}_{k}'w_{k+1}'\cdots w_{m-1}',\label{eq:eigenspace-elements}
\end{equation}
where $w_{i}$ denotes the longest element of $\tilde{W}_{i}$, which
are defined by the filtration $F_{i}$ as in \ref{eq:parabolic-sequence},
and the hat denotes omission: we remove the $w_{k}'$-terms from that
product if $V_{k}$ no longer occurs in $\underline{\Theta}^{i}$.
\end{defn}

In particular, $\vartheta_{0}=w_{\circ}w_{m}$ and if some eigenspace
$V_{j}$ in $\underline{\Theta}$ is redundant, then $w_{j}'=\mathrm{id}$.
\begin{prop}
\label{prop:eigenspaces-dg-form} For any sequence of eigenspaces
of length $m$ and each integer $1\leq i\leq m-1$, these elements
satisfy $\vartheta_{i-1}\geq\vartheta_{i}$.
\end{prop}

\begin{proof}
By construction, we have a decomposition 
\[
\vartheta_{i-1}=w_{i_{0}}'w_{i_{1}}'\cdots w_{i_{n}}'
\]
for some $0\leq i_{0}<\cdots<i_{n}\leq m-1$, which is reduced by
the same reasoning as in Proposition \ref{prop:parabolic-subgroup-centraliser}(ii)
(allowing that some of these elements equal the identity, or simply
removing redundant eigenspaces). Moreover, we can obtain a reduced
decomposition for $\vartheta_{i}$ from this one by omitting factors
$w_{j_{1}}',\ldots,w_{j_{l}}'$ for some $j_{1}\leq\cdots\leq j_{l}$.
For each $j_{k}$ we inductively conjugate $w_{j_{k}}'$ by all the
$w_{i_{l}}'$ with $i_{l'}<j_{k}$: define $k'$ such that $i_{k'}=j_{k}$,
start with $i_{k'-1}$ and consecutively go downward, i.e.\ we set
\[
w_{j_{k},0}':=w_{j_{k}}',\qquad w_{j_{k},t}':=w'_{i_{k'-t}}w_{j_{k},t-1}'(w_{i_{k'-t}}')^{-1}
\]
for $1\leq t\leq k'$, and we denote the final element by $\vartheta_{j_{k}}':=w_{j_{k},k'}'$.
Then by induction each $w_{j_{k},t}'$ lies in the parabolic subgroup
$W_{i_{k'-t}}$, so the first two decompositions in Proposition \ref{prop:parabolic-subgroup-centraliser}(ii)
yields the reduced decompositions $w_{i_{k'-t}}'w_{j_{k},t-1}'=w_{j_{k},t}'w_{i_{k'-t}}'$.
Then inductively we find
\[
w_{i_{0}}'w_{i_{1}}'\cdots w_{i_{k'-2}}'w_{i_{k'-1}}'w_{j_{k}}'=w_{i_{0}}'w_{i_{1}}'\cdots w_{i_{k'-2}}'w_{j_{k},1}'w_{i_{k'-1}}'=\vartheta_{j_{k}}'w_{i_{0}}'w_{i_{1}}'\cdots w_{i_{k'-2}}'w_{i_{k'-1}}'
\]
which inductively yields a reduced decomposition
\begin{align*}
\vartheta_{i-1} & =(w_{i_{0}}'w_{i_{1}}'\cdots w_{i_{l'-1}}'w_{j_{l}}')\cdots w_{i_{n}}'\\
 & =\vartheta_{j_{l}}'(w_{i_{0}}'w_{i_{1}}'\cdots w_{i_{l'-1}}')\hat{w}_{j_{l}}'\cdots w_{i_{n}}'\\
 & =\vartheta_{j_{l}}'\cdots\vartheta_{j_{k+1}}'(w_{i_{0}}'w_{i_{1}}'\cdots w_{i_{k'-1}}'w_{j_{k}}')w_{i_{k'+1}}'\cdots w_{i_{l'-1}}'\hat{w}_{j_{l}}'\cdots w_{i_{n}}'=(\vartheta_{j_{l}}'\cdots\vartheta_{j_{1}}')\vartheta_{i}.\qedhere
\end{align*}
\end{proof}
\begin{rem}
Let $w$ be a nontrivial element of a twisted finite Coxeter group
and let $\underline{\Theta}$ be a braiding sequence of eigenspaces.
Unwinding the definitions, one finds that the sequence is
\[
\begin{array}{ccc}
\textrm{anisotropic} & \qquad\textrm{if and only if}\qquad & d_{0}'>0\textrm{ and }\mathrm{pb}(w)=\vartheta_{0},\\
\textrm{not anisotropic} & \qquad\textrm{if and only if}\qquad & d_{0}'=0,d_{1}'>0\textrm{ and }\mathfrak{R}_{\vartheta_{1}}\subseteq\mathfrak{R}_{+}\backslash\mathfrak{R}^{w}\subseteq\mathfrak{R}_{\vartheta_{0}}.
\end{array}
\]
\end{rem}

\begin{prop}
\label{prop:braiding-dgn} Let $W=\Omega\ltimes\tilde{W}$ be a twisted
finite Coxeter group with reflection representation $V$. Suppose
$w=\delta\tilde{w}$ is an element of $W$ with a braiding sequence
of eigenspaces $\underline{\Theta}$. Let $d\geq0$ denote any integer
such that for each $\theta_{i}$ we have $d\theta_{i}\in\mathbb{N}_{0}$
(e.g., $d=\mathrm{ord}(w)$). Let $d_{i}'$ and $\vartheta_{i}$ be
as in Definition \ref{def:braiding-sequence-elements}, and set $d_{i}:=d\cdot d_{i}'$
for $0\leq i\leq m-1$. Then

\[
\mathrm{DGN}(b_{w}^{d})=\delta^{d}b_{\pm\vartheta_{m-1}}^{d_{m-1}}\cdots b_{\pm\vartheta_{1}}^{d_{1}}b_{\pm\vartheta_{0}}^{d_{0}},
\]
and if moreover $d$ is even then
\[
\mathrm{DGN}(b_{w}^{d/2})=\delta^{d/2}b_{\pm\vartheta_{m-1}}^{d_{m-1}/2}\cdots b_{\pm\vartheta_{1}}^{d_{1}/2}b_{\pm\vartheta_{0}}^{d_{0}/2},
\]
where if $n$ is a half-integer we mean
\[
b_{\pm\vartheta_{i}}^{n}:=b_{\vartheta_{i}}^{\,}b_{\pm\vartheta_{i}}^{n-1/2}.
\]
\end{prop}

We follow the proof of \cite[Theorem 5.3]{MR2999317}:
\begin{proof}
We induct on the rank of $\tilde{W}$ (or the dimension of $V_{\Theta}$);
the base case essentially follows from Lemma \ref{lem:factorise}(ii).
As before we let $\tilde{W}_{1}$ denote the standard parabolic subgroup
of elements in $\tilde{W}$ fixing $V_{1}$, where $\underline{\Theta}=(\ldots,V_{1})$.
Since $V_{1}$ is nontrivial, the rank of $\tilde{W}_{1}$ is smaller
than the rank of $\tilde{W}$. By Lemma \ref{lem:factorise}(i), we
have $w=xy$ for $y\in\tilde{W}_{1}$ and $x\in W$ an element whose
action preserves the set of simple roots of $\tilde{W}_{1}$. From
Lemma \ref{lem:factorise}(iii), it follows that $\delta_{x}y\in W':=\<\delta_{x}\>\ltimes\tilde{W}_{1}$
has a braiding sequence of eigenspaces $\underline{\Theta}'$ constructed
out of $\underline{\Theta}$. If we relabel the integers (\ref{eq:eigensequence-integers})
corresponding to the sequence $\underline{\Theta}'$ as $\tilde{d}'_{1},\ldots,\tilde{d}_{m-1}'$
and the elements on the right-hand side of (\ref{eq:eigenspace-elements})
as $\tilde{\vartheta}_{1},\ldots,\tilde{\vartheta}_{m-1}$, set $\tilde{d}_{i}:=d\cdot\tilde{d}_{i}'$
and write $k:=|\{\theta_{i}\in\mathrm{arg}(\Theta'):\theta_{1}\geq\theta_{i}\}|\geq0$,
then $k$ is the largest integer such that 
\[
\tilde{k}:=\theta_{1}-\sum_{i=1}^{k}\tilde{d}_{i}'\geq0.
\]
These are related to $d_{i}$ and $\vartheta_{i}$ by
\[
(d_{i},\vartheta_{i})=\begin{cases}
(\tilde{d}_{i+1},w_{\circ}w_{1}\tilde{\vartheta}_{i+1}) & \text{if }i<k,\\
(\tilde{d}_{i+1},\tilde{\vartheta}_{i+1}) & \text{if }\tilde{k}=0,i\geq k,\\
(\tilde{k},w_{\circ}w_{1}\tilde{\vartheta}_{i+1}) & \text{if }\tilde{k}\neq0\text{ and }i=k,\\
(\tilde{d}_{i}-\tilde{k},\tilde{\vartheta}_{i}) & \text{if }\tilde{k}\neq0\text{ and }i=k+1,\\
(\tilde{d}_{i},\tilde{\vartheta}_{i}) & \text{if }\tilde{k}\neq0\text{ and }i>k+1.
\end{cases}
\]

The induction hypothesis on $W'$ yields 
\[
(\delta_{x}b_{y})^{d}=\delta_{x}^{d}b_{\pm\tilde{\vartheta}_{m-1}}^{\tilde{d}_{m-1}}\cdots b_{\pm\tilde{\vartheta}_{1}}^{\tilde{d}_{1}}.
\]
By \cite[Lemma 5.2]{MR2999317} (the first condition in the second
sentence is satisfied because $x$ is a minimal coset representative),
Proposition \ref{prop:parabolic-subgroup-centraliser}(ii) and the
submonoid identification $\mathrm{Br}_{W'}^{+}\subset\mathrm{Br}^{+}$
it now follows that
\begin{align*}
b_{w}^{d} & =b_{x}^{d}b_{x^{1-d}yx^{d-1}}\cdots b_{x^{-1}yx}b_{y}\\
 & =b_{x}^{d}\bigl(\delta_{x}^{-d}(\delta_{x}b_{y})^{d}\bigr)\\
 & =(\delta^{d}b_{\pm w_{\circ}w_{1}}^{d\theta_{0}/2})(b_{\pm\tilde{\vartheta}_{m-1}}^{\tilde{d}_{m-1}}\cdots b_{\pm\tilde{\vartheta}_{1}}^{\tilde{d}_{1}})\\
 & =\delta^{d}b_{\pm\tilde{\vartheta}_{m-1}}^{\tilde{d}_{m-1}}\cdots b_{\pm\tilde{\vartheta}_{1}}^{\tilde{d}_{1}}b_{\pm w_{\circ}w_{1}}^{d\theta_{0}/2}\\
 & =\delta^{d}b_{\pm\tilde{\vartheta}_{m-1}}^{\tilde{d}_{m-1}}\cdots b_{\pm\tilde{\vartheta}_{1}}^{\tilde{d}_{1}-1}b_{\pm w_{\circ}w_{1}}^{d\theta_{0}/2-1}b_{\tilde{\vartheta}_{1}^{-1}}b_{w_{1}w_{\circ}}b_{w_{\circ}w_{1}}b_{\tilde{\vartheta}_{1}}\\
 & =\delta^{d}b_{\pm\tilde{\vartheta}_{m-1}}^{\tilde{d}_{m-1}}\cdots b_{\tilde{\vartheta}_{k}}^{\tilde{d}_{k}}b_{\pm w_{\circ}w_{1}}^{\tilde{k}}b_{\vartheta_{k-1}}^{d_{k-1}}\cdots b_{\vartheta_{0}}^{d_{0}}=\delta^{d}b_{\pm\vartheta_{m-1}}^{d_{m-1}}\cdots b_{\pm\vartheta_{1}}^{d_{1}}b_{\pm\vartheta_{0}}^{d_{0}}
\end{align*}
From Propositions \ref{prop:dg-properties} and \ref{prop:eigenspaces-dg-form},
it follows that the final expression is in Deligne-Garside normal
form.

The claim for $d$ even is proven analogously.
\end{proof}
\begin{cor}
\label{cor:anisotropic-braid} In particular, if the eigenspaces of
$\underline{\Theta}$ are not redundant then
\[
\mathrm{DG}(b_{w}^{d})=\begin{cases}
b_{\mathrm{pb}(w)} & \textrm{if }\underline{\Theta}\textrm{ is anisotropic},\\
b_{\vartheta_{1}} & \textrm{if }\underline{\Theta}\textrm{ is not anisotropic},
\end{cases}
\]
and similarly for $\mathrm{DG}(b_{w}^{d/2})$ when $d$ is even.
\end{cor}

\begin{proof}
This follows immediately by combining the previous two statements.
\end{proof}
We now consider several specialisations of this proposition:
\begin{cor}
\label{cor:decreasing-decomp} Suppose furthermore that the sequence
$\underline{\Theta}$ is decreasing. Then
\[
\mathrm{DGN}(b_{w}^{d})=\delta^{d}b_{\pm w_{m-1}w_{m}}^{d(\theta_{m}-\theta_{m-1})}\cdots b_{\pm w_{1}w_{m}}^{d(\theta_{2}-\theta_{1})}b_{\pm w_{\circ}w_{m}}^{d\theta_{1}}
\]
 and if moreover $d$ is even then
\[
\mathrm{DGN}(b_{w}^{d/2})=\delta^{d/2}b_{\pm w_{m-1}w_{m}}^{d(\theta_{m}-\theta_{m-1})/2}\cdots b_{\pm w_{1}w_{m}}^{d(\theta_{2}-\theta_{1})/2}b_{\pm w_{\circ}w_{m}}^{d\theta_{1}/2}.
\]
If instead $\underline{\Theta}$ is increasing, then
\[
\mathrm{DGN}(b_{w}^{d})=\delta^{d}b_{\pm w_{\circ}w_{1}}^{d(\theta_{1}-\theta_{2})}\cdots b_{\pm w_{\circ}w_{m-1}}^{d(\theta_{m-1}-\theta_{m})}b_{\pm w_{\circ}w_{m}}^{d\theta_{m}}
\]
 and if moreover $d$ is even then
\[
\mathrm{DGN}(b_{w}^{d/2})=\delta^{d/2}b_{\pm w_{\circ}w_{1}}^{d(\theta_{1}-\theta_{2})/2}\cdots b_{\pm w_{\circ}w_{m-1}}^{d(\theta_{m-1}-\theta_{m})/2}b_{\pm w_{\circ}w_{m}}^{d\theta_{m}/2}.
\]

In particular, all of these braids centralise the submonoid $\mathrm{Br}_{\tilde{W}_{k}}^{+}$,
where $k\geq0$ is the largest integer such that $\tilde{W}_{k}$
is nontrivial.
\end{cor}

\begin{proof}
Decreasing implies that $\vartheta_{i}=w_{i}w_{m}$ and $d_{i}=d(\theta_{i+1}-\theta_{i})$
(briefly setting $\theta_{0}:=0$); the increasing case is similar.
The final claim then follows from Proposition \ref{prop:parabolic-subgroup-centraliser}(ii)
in the case when $w_{m}$ is nontrivial, and from the centrality of
$b_{w_{i}}^{2}$ in $\mathrm{Br}_{\tilde{W}_{i}}^{+}$ otherwise.
\end{proof}
\begin{example}
Consider the elliptic element $w=s_{2}s_{3}s_{1}s_{2}s_{3}s_{2}s_{1}$
in type $\mathsf{B}_{3}$. Then by setting $\underline{\Theta}=(V_{-1}^{w},V_{i}^{w})$
we obtain an anisotropic $V_{w}$-admissible sequence of eigenspaces
such that the dominant Weyl chamber is in good position, and as $d:=\mathrm{ord}(w)=4$
is even we find
\[
b_{w}^{2}=b_{12321}^{\,}b_{w_{\circ}}^{\,}.
\]
\end{example}

It has been conjectured that if braids $b,b'$ in $\mathrm{Br}^{+}$
satisfy $b^{d}=(b')^{d}$ for some integer $d\geq1$, then $b$ and
$b'$ are conjugate. In type $\mathsf{A}$ this was proven by employing
the Nielsen-Thurston classification \cite{MR2012967}, which can be
rephrased in terms of the conjugation action on the set of proper
parabolic subgroups:
\begin{defn}
[{\cite[§4]{MR2264551}}] Let $W$ be a twisted finite Coxeter group.
One defines the notion of \emph{(standard) parabolic} subgroups of
the corresponding untwisted braid group $\mathrm{Br}_{\tilde{W}}$
analogous to those of $\tilde{W}$, and considers the set of proper,
nontrivial parabolic subgroups of $\mathrm{Br}_{\tilde{W}}$; elements
of $\mathrm{Br}$ act on it by conjugating. An element $b$ of $\mathrm{Br}$
is then said to be \emph{pseudo-Anosov} if its action on this set
does not have a finite orbit.
\end{defn}

Finally, we deduce the main Theorem:
\begin{proof}
[Proof of Theorem:]

(i): This is a rephrasing of Corollary \ref{cor:anisotropic-braid}.

(ii): By induction it suffices to prove the claim for a single cyclic
shift between two convex elements, so there are reduced decompositions
$w=yx$ to $w'=xy$. Part (i) of Lemma \ref{lem:move-power-bound}
yields
\[
b_{w}^{d+1}=b_{y}b_{w'}^{d}b_{x}=b_{y}b'b_{\mathrm{pb}(w')}b_{x}=b''b_{\mathrm{pb}(w)}
\]
for some elements $b',b''$ in $B^{+}$. The upper bound on $d$ is
obtained from Proposition \ref{prop:dg-bound}(ii).

(iii): The quasiregular eigenspace of a quasiregular element furnishes
a braiding sequence of length 1 for some Weyl chamber; after conjugating,
the claim then follows from Proposition \ref{prop:braiding-dgn} for
certain minimally dominant elements in its conjugacy class. More generally,
suppose we are given a cyclic shift from $w=yx$ to a quasiregular
element $w'=xy$ satisfying this braid equation. Part (i) and (ii)
of Lemma \ref{lem:move-power-bound} yield
\[
b_{w}^{\mathrm{ord}(w)}=(b_{x}^{-1}b_{w'}b_{x})^{\mathrm{ord}(w)}=b_{x}^{-1}b_{w'}^{\mathrm{ord}(w)}b_{x}=b_{x}^{-1}b_{\mathrm{pb}(w')^{-1}}b_{\mathrm{pb}(w')}b_{x}=b_{x}^{-1}b_{\mathrm{pb}(w')^{-1}}b_{x'}b_{\mathrm{pb}(w)}=b_{\mathrm{pb}(w)^{-1}}b_{\mathrm{pb}(w)}.
\]

For the converse, we may assume that $w$ is nontrivial; from part
(iii) of the main Lemma and the braid equation we obtain that
\[
\mathrm{ord}(w)\cdot\ell(\CMcal O_{\mathrm{min}}^{\mathrm{dom}})\geq\mathrm{ord}(w)\cdot\frac{|\mathfrak{R}|-\ell_{f}(w)}{\mathrm{ord}(w)}=2\ell\bigl(\mathrm{pb}(w)\bigr)=\mathrm{ord}(w)\cdot\ell(w)
\]
with equality if and only if $w$ is quasiregular, so $\ell(\CMcal O_{\mathrm{min}}^{\mathrm{dom}})\geq\ell(w)$
with equality if and only if $w$ is quasiregular. But it was shown
by Broué, Digne and Michel that an element satisfying this braid equation
is quasiregular \cite[Lemma 8.4(iii)]{MR3187647}.

(iv): If an element $w$ lies inside of a proper parabolic subgroup,
then all roots outside of the corresponding standard parabolic subroot
system it are stable. Since the highest root does not lie in a standard
parabolic subroot system, it follows that if $w$ is convex then all
roots are stable and hence $w$ must be trivial.

Combined with part (i) of Lemma \ref{lem:standard-parab}, this also
yields $(b)\Rightarrow(a)$. From the definition of firmly convex
and convex one deduces $(c)\Rightarrow(b)$, and finally $(a)\Rightarrow(c)$
follows from part (ii) as those elements are minimally dominant.

(v): By choosing a suitable Weyl chamber, it follows that within $\CMcal O$
there exist an elements with an decreasing braiding sequence of eigenspaces;
depending on choice, they have maximal length, minimal length or are
minimally dominant. According to Corollary \ref{cor:decreasing-decomp},
the element $b_{w}^{\mathrm{ord}(w)}$ centralises a nontrivial ``standard
parabolic'' submonoid of $\mathrm{Br}^{+}$, but then it also centralises
the corresponding standard parabolic subgroup of $\mathrm{Br}$. Since
all minimal length, maximal length and minimally dominant elements
are in the same strong conjugacy or cyclic shift class, their Artin-Tits
braids are conjugate as braids, and hence each of these elements centralises
a nontrivial parabolic subgroup of $\mathrm{Br}$.
\end{proof}
Equation (\ref{eq:bound}) cannot be sharpened to $d\geq\mathrm{ord}(w)$:
\begin{example}
Consider the element $w=xyx^{-1}$ in the conjugacy class of $y=s_{5}s_{4}s_{3}s_{2}s_{1}s_{6}s_{4}s_{3}s_{2}s_{1}$
in type $\mathsf{D}_{6}$, where we conjugate by the element $x=s_{2}s_{3}s_{4}$.
As $y$ is elliptic and has order 6 the same is true for $w$, but
\[
\mathrm{DG}(b_{w}^{7})=w_{\circ}>w_{\circ}s_{2}=\mathrm{DG}(b_{w}^{6}).
\]
\end{example}

And even when a lower power $d$ initially works, cyclic shifts can
increase it:
\begin{example}
Consider the Coxeter elements $w=s_{3}s_{2}s_{1}$ and $w'=s_{2}s_{3}s_{1}$
in type $\mathsf{A}_{3}$. Then 
\[
\mathrm{DG}(b_{w}^{3})=\mathrm{DG}(b_{w'}^{2})=w_{\circ},
\]
 but $\mathrm{DG}(b_{w}^{2})\neq w_{\circ}$.
\end{example}

\begin{example}
Consider $w=s_{2}s_{3}s_{2}s_{1}$ in type $\mathsf{A}_{3}$. This
is a regular element which is neither elliptic nor of minimal length,
but $b_{w}^{\mathrm{ord}(w)}=b_{w_{\circ}}^{2}$.
\end{example}

Convexity and firm convexity are not necessarily preserved under cyclic
shifts:
\begin{example}
The elements $s_{1}s_{2}s_{3}s_{2}$ and $s_{2}s_{1}s_{2}s_{3}$ in
type $\mathsf{B}_{3}$ are conjugate by a cyclic shift, but they fix
different roots; the first element is convex and satisfies $\mathfrak{R}_{\mathrm{DG}(b_{w}^{d})}=\mathfrak{R}_{+}\backslash\mathfrak{R}^{w}$
for $d$ large, the second is not convex and hence by Proposition
\ref{prop:closed-complement} cannot satisfy this equation. The same
is true for the non-elliptic elements $s_{4}s_{1}s_{2}s_{3}s_{2}s_{1}$
and $s_{1}s_{2}s_{1}s_{3}s_{4}s_{3}$ in type $\mathsf{A}_{4}$.
\end{example}

\small

\bibliographystyle{alphakey}
\phantomsection\addcontentsline{toc}{section}{\refname}\bibliography{mindom}

\end{document}